\documentclass{article}
\usepackage{amssymb,amsfonts,amsmath,amsthm,amsopn,amstext,amscd,latexsym,xy,stmaryrd}
\usepackage[hidelinks]{hyperref}
\theoremstyle{plain}
\usepackage{verbatim}
\usepackage{mathrsfs}

\input xy

\usepackage{tikz}
\usetikzlibrary{decorations.markings}
\usetikzlibrary{arrows}

\xyoption{all}
\newtheorem{theorem}{Theorem}[section]
\newtheorem{lemma}[theorem]{Lemma}
\newtheorem{proposition}[theorem]{Proposition}
\newtheorem{corollary}[theorem]{Corollary}

\theoremstyle{remark}
\newtheorem{remark}[theorem]{Remark}
\numberwithin{equation}{section}
\numberwithin{paragraph}{section}

\theoremstyle{definition}

\newtheorem{definition}[theorem]{Definition}
\newtheorem{property}[theorem]{Property}

\makeatletter
\newcommand{\setproptag}[1]{
	\let\oldthetheorem\thetheorem
	\renewcommand{\thetheorem}{#1}}

\g@addto@macro\endproperty{
		\global\let\thetheorem\oldthetheorem
		\addtocounter{theorem}{-1}
}

\makeatother

\newcommand{\mf}{\mathbf}
\newcommand{\ov}{\overline}
\newcommand{\mr}{\mathrm}
\newcommand{\mc}{\mathcal}

\DeclareMathOperator{\Hom}{Hom}
\DeclareMathOperator{\Mod}{Mod}
\DeclareMathOperator{\Ad}{Ad}

\DeclareMathOperator{\Id}{Id}

\DeclareMathOperator{\Gal}{Gal}

\DeclareMathOperator{\ord}{ord}

\DeclareMathOperator{\Fitt}{Fitt}

\DeclareMathOperator{\tr}{tr}
\DeclareMathOperator{\St}{St}

\DeclareMathOperator{\End}{End}
\DeclareMathOperator{\Sym}{Sym}
\DeclareMathOperator{\Frob}{Frob}
\DeclareMathOperator{\Ann}{Ann}

\DeclareMathOperator{\Spec}{Spec}

\DeclareMathOperator{\st}{st}
\DeclareMathOperator{\fl}{fl}
\DeclareMathOperator{\loc}{loc}

\DeclareMathOperator{\uni}{uni}
\DeclareMathOperator{\Der}{Der}
\DeclareMathOperator{\Tor}{Tor}
\DeclareMathOperator{\proj}{proj}
\DeclareMathOperator{\an}{an}
\renewcommand{\min}{\operatorname{min}}
\renewcommand{\mod}{\operatorname{mod}}
\DeclareMathOperator{\kernel}{ker}
\DeclareMathOperator{\image}{im}

\newcommand{\Lotimes}{\otimes^{\mathbf{L}}}
\DeclareMathOperator{\RHom}{RHom}
\newcommand{\cDer}{\widehat{\Der}}

\newcommand{\cA}{{\mathcal A}}

\newcommand{\cD}{{\mathcal D}}

\newcommand{\cI}{{\mathcal I}}

\newcommand{\cL}{{\mathcal L}}

\newcommand{\cO}{{\mathcal O}}

\newcommand{\cR}{{\mathcal R}}

\newcommand{\ffrm}{{\mathfrak m}}

\newcommand{\frp}{{\mathfrak p}}
\newcommand{\frq}{{\mathfrak q}}

\newcommand{\bbF}{{\mathbb F}}

\newcommand{\GL}{\mathrm{GL}}

\newcommand{\PGL}{\mathrm{PGL}}

\newcommand{\SL}{\mathrm{SL}}

\newcommand{\barqp}{{\overline{\Q}_p}}

\newcommand{\Fbar}{{\overline{\F}}}
\newcommand{\Qbar}{{\overline{\Q}}}
\newcommand{\Tbar}{{\overline{\T}}}
\newcommand{\CNLO}{{\mathrm{CNL}_\cO}}
\newcommand{\CNL}{{\mathrm{CNL}}}

\newcommand{\sm}{\smallsetminus}
\newcommand{\es}{\varnothing}

\newcommand{\cOmega}{{\widehat{\Omega}}}
\newcommand{\cotimes}{{\widehat{\otimes}}}
\newcommand{\cbigotimes}[1]{{\widehat{\bigotimes_{#1}}}}
\newcommand{\invlim}{\varprojlim}
\newcommand{\ds}{\displaystyle}
\newcommand{\half}{\mathcal{H}}

\newcommand{\Rt}{\widetilde{R}}
\newcommand{\xt}{\widetilde{x}}
\newcommand{\yt}{\widetilde{y}}
\newcommand{\lambdat}{\widetilde{\lambda}}
\newcommand{\thetat}{\widetilde{\theta}}
\newcommand{\Thetat}{\widetilde{\Theta}}

\newcommand{\rhobar}{\overline{\rho}}

\newcommand{\A}{\mathbf A}
\newcommand{\Q}{\mathbf Q}
\newcommand{\Z}{\mathbf Z}
\newcommand{\T}{\mathbf T}
\newcommand{\F}{\mathbf F}

\newcommand{\into}{\hookrightarrow}
\newcommand{\onto}{\twoheadrightarrow}

\newcommand{\tf}{\mathrm{tf}}

\newcommand{\olQ}{Q}

\newcommand{\sMat}[2][.8]{%
  \scalebox{#1}{%
    \renewcommand{\arraystretch}{.8}%
   $\begin{pmatrix}#2\end{pmatrix}$%
  }
}

\DeclareMathOperator{\depth}{depth}

\DeclareMathOperator{\socle}{{\mathrm{socle}}}
\renewcommand{\st}{{\mathrm{st}}}
\newcommand{\un}{{\mathrm{un}}}
\newcommand{\unr}{{\mathrm{unr}}}
\newcommand{\fun}{{\operatorname{\varphi-uni}}}
\newcommand{\cst}{{$(\st)$}}
\newcommand{\cun}{{$(\un)$}}
\newcommand{\ctun}{{$(\fun)$}}
\newcommand{\uq}{{\underline{q}}}
\newcommand{\us}{{\underline{s}}}
\newcommand{\ut}{{\underline{t}}}
\newcommand{\Ro}{\overline{R}}
\newcommand{\Io}{\overline{I}}
\newcommand{\pio}{\overline{\pi}}

\renewcommand{\St}{\mathrm{st}}
\renewcommand{\uni}{\mathrm{uni}}

\newcommand{\chibar}{\overline\chi}

\DeclareMathOperator{\tcD}{\widetilde{\cD}}
\newcommand{\tR}{\widetilde R}
\newcommand{\tpi}{\widetilde \pi}

\renewcommand{\un}{{\mathrm{uni}}}

\newcommand{\Lambdat}{\widetilde{\Lambda}}

\title{ Wiles defect of Hecke algebras via  local-global arguments  \\[0.3em]\small{}With an appendix by N. Fakhruddin and C. Khare}

\author{Gebhard B\"ockle, Chandrashekhar B. Khare, Jeffrey Manning} 

\begin{document}
\maketitle

\begin{abstract}
In his work on modularity of elliptic curves and Fermat’s Last Theorem,  A.~Wiles introduced two measures of congruences  between  Galois representations and between modular forms. One measure  is related to the order of a Selmer group associated to a newform  $f \in S_2(\Gamma_0(N))$ (and closely linked to  deformations of the Galois representation $\rho_f$ associated to $f$), whilst  the other measure  is related to the congruence module associated to $f$ (and is closely linked to Hecke rings and congruences between $f$ and other newforms in $S_2(\Gamma_0(N))$). The equality  of these two measures led to isomorphisms $R=\T$  between deformation rings and Hecke rings (via a numerical criterion for isomorphisms that Wiles  proved)  and showed these rings to be complete intersections.

 We continue our study begun in \cite{BKM}  of the {\it Wiles defect} of  deformation rings  and Hecke  rings (at a newform $f$)  acting on  the cohomology of Shimura curves over $\Q$: it is defined to be  the difference  between these two measures of congruences. The Wiles defect  thus  arises from the failure of the Wiles numerical criterion at an augmentation $\lambda_f:\T \to \cO$. In  situations we study here  the   Taylor-Wiles-Kisin patching method  gives  an isomorphism $ R=\T$ without the rings being complete intersections. Using novel arguments in commutative algebra and patching, we  generalize significantly and give different proofs of the results in \cite{BKM} that compute the Wiles defect at  $\lambda_f: R=\T \to \cO$, and explain  in an {\it a priori} manner why the answer in \cite{BKM}  is a sum of {\it  local defects}. As a curious application of our work we give a new and more robust approach to the result of Ribet--Takahashi that computes change of degrees of optimal parametrizations of elliptic curves over $\Q$ by Shimura curves as we vary the Shimura curve. 

\end{abstract}

\tableofcontents

\section{Introduction}\label{sec:intro}

In the  work  on modularity of elliptic curves,   Wiles pioneered methods to prove   $R=\T$ theorems where $R$ is a deformation ring and $\T$ a Hecke algebra, thus proving an equality of moduli spaces of  Galois representations to pro-$p$ Artinian rings arising from modular forms with the {\it a priori} larger  moduli space of  corresponding abstract Galois representations, both with  certain prescribed   local (ramification) behavior. 

The  injectivity of the {\it a priori} surjective map $R \onto \T$ was proven by using  two different types of  criteria/methods:  

\begin{description}
	\item[(i)] the numerical criterion of  \cite[Proposition 2 of Appendix]{Wiles}; 
	\item[(ii)] the patching method of \cite{TaylorWiles}.
\end{description}

 In  \cite{TaylorWiles}  the local conditions imposed  on the deformations were smooth. Kisin \cite{Kisin}  later generalized the patching method to allow   local conditions on the deformations   that were  not necessarily  smooth. The generic fiber of the local deformation rings  in question was  smooth  and  Kisin proved a   $R[1/p]=\T[1/p]$ theorem,  thus proving a coarser  equality of moduli spaces of   $p$-adic Galois representations arising from modular forms with the {\it a priori} larger  moduli space of  corresponding abstract Galois representations, both with  certain prescribed   local behavior. When the local conditions are Cohen--Macaulay,  one  sees {\em a posteriori}   that $R$ has no $p$-torsion (see\cite[paragraph before Corollary 4.7]{KW},  \cite[\S 5]{Snowden} or \cite[Theorem 6.3]{BKM}  for instance) and  thus  as $\T$  is also torsion-free one can promote an $R[1/p]=\T[1/p]$ theorem to an   integral $R=\T$ theorem, without the rings in question turning out to be complete intersections.

Wiles  applied  his  numerical criterion for maps between rings to be isomorphisms of complete intersections to prove,    when the local conditions on the deformations  considered were not minimal,  the relevant  surjective map $R \onto \T$ was an  isomorphism of complete intersections (see\cite[Theorem 2.17 of \S 2]{Wiles}).  Wiles deduced in this manner   $R=\T$ theorems in the  non-minimal case   from $R=\T$ theorems  in the minimal case.   The latter were   proved via the  patching method of \cite{TaylorWiles}.  The numerical criterion has been used subsequently  in \cite{K} to prove  $R=\T$ theorems without  any reliance  on patching. The numerical  criterion of Wiles has  not as yet  been generalized to give a  criterion for maps between rings to be an isomorphism when the rings are known  to not  be  complete intersections.  

The work   of this paper, like that of the previous paper \cite{BKM} of this series, arises when considering situations when we have $R=\T$ theorems proved by patching, but  $R$ and $\T$  fail to be complete intersections. In \cite{BKM}  and the present paper we seek to study the failure (quantified in a numerical quantity called the {\em Wiles defect} introduced in \cite{TiUr}, see also \cite[Definition 3.10]{BKM}) of the numerical criterion for being a complete intersection locally at an augmentation $\lambda_f:\T \to \cO$ induced by  a newform $f$.

In \cite{BKM} we studied the  Wiles defect (at $\lambda_f$ of  a certain Hecke ring $\T$ acting on the cohomology of a Shimura curve)   using a combination of patching and level lowering results of Ribet-Takahashi \cite{RiTa}. In the present paper we combine the  new  results in commutative algebra that we prove here with  patching  to  determine the Wiles defect.  The patching method  allows one to show
that the Wiles defect of a global deformation ring at an augmentation $\lambda_f$  depends only on the in induced augmentations of the corresponding  local deformation rings.  This gives   yet another illustration of the versatility of the patching method, and its ability to reduce proving   properties of global deformation rings to proving  properties of the corresponding local deformation rings.

As a curious  consequence we  derive and strengthen  the results of Ribet--Takahashi in \cite{RiTa}  on degrees of optimal parametrizations of elliptic curves over $\Q$ by Shimura curves, via a new  argument. The methods of Ribet--Takahashi use arithmetic geometry,  while the method here uses patching. 
 
 \subsection{A particular case of our main theorem}
 
 In \cite[Theorem 10.1]{BKM} we  determined the Wiles defect  associated to a newform  $f \in S_2(\Gamma_0(NQ))$ of squarefree level $NQ$ that arises by the Jacquet Langlands correspondence from a newform in $S_2(\Gamma^Q_0(N))$. Here  $\Gamma^Q_0(N)$ is  the congruence subgroup of a quaternion algebra that is ramified at the set of primes dividing $Q$, of level $\Gamma_0(N)$, and maximal at the primes in $Q$.

We state an improvement of     \cite[Theorem 10.1]{BKM} referring to it for any of the unexplained   notation in the statement below    (we  do recall the definition of the Wiles defect below). The proof relies on  the Taylor-Wiles-Kisin patching method, but not on \cite{RiTa}, and also explains {\it en passant}  why  the Wiles defect computed below is a sum of  {\em local defects} in a sense we make precise later in the introduction.

\begin{theorem}\label{thm:intro}
Let $N$ and $Q$ be relatively prime squarefree integers. Let $p>2$ be a prime not dividing $NQ$ and let $E/\Q_p$ be a finite extension with ring of integers $\cO$, uniformizer $\varpi$ and residue field $k$. Let $\rho_f:G_{\Q}\to \GL_2(\cO)$ be a Galois representation arising from a newform $f\in S_2(\Gamma_0(NQ))$, and let $\rhobar_f:G_{\Q}\to \GL_2(k)$ be the residual representation. Assume that $\rhobar_f$ is irreducible and $N|N(\rhobar_f)$.

Let $R^{\st}$ be the Galois deformation ring of $\rhobar_f$  parameterizing lifts of $\rhobar_f$ of fixed determinant which are Steinberg at each prime dividing $Q$, finite flat at $p$ and minimal at all other primes.

Let $D$ be the quaternion algebra with discriminant $Q$ and let $\Gamma_0^Q(N)$ be the congruence subgroup for $D$. Let $\T^Q(N)$ and $S^Q(\Gamma_0^Q(N))$ be the Hecke algebra and cohomological Hecke module at level $\Gamma^Q_0(N)$ and let $\ffrm\subseteq \T^Q(N)$ be the maximal ideal corresponding to $ f$. Let $\T^{\st} = \T^Q(N)_\ffrm$ and let $\lambda:\T^{\st}\to \cO$ be the augmentation corresponding to $f$.

Then the Wiles defects  of $\T^{\st}$ and $S^Q(\Gamma_0^Q(N))$ with respect to the map $R^{\st}\onto \T^{\st}$ and the augmentation $\lambda$ are
\[
\delta_\lambda(R^\st)=\delta_\lambda(\T^{\St}) = \delta_\lambda(S^Q(\Gamma_0^Q(N))_{\ffrm}) = \sum_{q|Q}\frac{2n_q}{e}
\]
where $e$ is the ramification index of $\cO$ and for each $q|Q$, $n_q$ is the largest integer for which $\rho_f|_{G_{\Q_q}} \pmod{\varpi^{n_q}}$ is unramified and $\rho_f(\Frob_q)\equiv \pm {\rm Id} \pmod{\varpi^{n_q}}$.
\end{theorem}

The improvement as far as the statement of the theorem is concerned, if one compares to  \cite[Theorem 10.1]{BKM}, is that the assumptions needed there on $Q$:
\begin{enumerate}
	\item $Q$ is a product of an even number of primes (i.e. $D$ is indefinite),and  $(N(\rhobar),Q)>1$;
	\item    $Q$ is a product of an  odd number of primes (i.e. $D$ is definite), and $N>1$;
	\item $N(\rhobar)$ is divisible by at least two primes,
\end{enumerate}
 which arose from our relying on delicate  results in \cite{RiTa}, are no longer needed because of  the innovations introduced in this paper.  We prove a much more general theorem below, see Theorem \ref{thm:mc}, that works with more general local conditions than being Steinberg at {\it trivial primes} (see \cite[\S 2]{BKM})  and with  the  field $\Q$ replaced by any totally real field $F$,   but focus on this special case for the purposes of the introduction to  more easily explain the novelty of our methods  in comparison  to \cite{BKM}.

If we look at the shape of the  formula:
\[
\delta_\lambda(\T^{\St}) = \delta_\lambda(S^Q(\Gamma_0^Q(N))_{\ffrm}) = \sum_{q|Q}\frac{2n_q}{e}
\] we see that the Wiles defect  $\delta_\lambda(\T^{\St}) $  that as defined is a {\it global} quantity arising from the augmentation $\lambda_f:\T^\st \to \cO$ is expressed as a  sum  over the primes dividing $ Q$ of  terms  $2n_q/e$. Furthermore  each of  the integers $n_q$   depends  only on  $\rho_f|_{G_{\Q_q}}$.  In \cite{BKM},  it is  only after  having proved the theorem  that one observes that the formula depends only on  $(\rho_f|_{G_{\Q_q}})_{q \in Q}$. In this paper we show that the Wiles defect $\delta_\lambda(R^\st) $ is {\it a priori} local,   and in fact  is a sum of  the defects of local deformation rings (equivalently, {\em local defects}) at primes in  $Q$ that we define below.   The proof of \cite[Theorem 10.1]{BKM} did not shed light on  the local-global  aspect of the statement of the  theorem. 

Further the proof in \cite{BKM} computed the Wiles defect
using a combination of patching and arguments related to level lowering results of \cite{RiTa}, which were used to first show that \[\delta_\lambda(S^Q(\Gamma_0^Q(N))_{\ffrm}) = \sum_{q|Q}\frac{2n_q}{e}.
\]  Then delicate  results from \cite{Manning} were used to prove  \cite[Theorem 3.10, Theorem 8.1, Corollary 8.3]{BKM} that  \[\delta_\lambda(\T^{\St}) = \delta_\lambda(S^Q(\Gamma_0^Q(N))_{\ffrm}), \] and hence that  \[\delta_\lambda(\T^{\St}) =   \sum_{q|Q}\frac{2n_q}{e}.\]  Here we reverse the logic of the proof in \cite{BKM}, and  first show  (see Theorem \ref{thm:mc}) that \[
\delta_\lambda(R^\st)=\delta_\lambda(\T^{\St})= \sum_{q|Q}\frac{2n_q}{e},\] and deduce from this  (see Theorem \ref{change-of-eta} (ii)  and Proposition \ref{prop:moduledefect}) that 
\[ \delta_\lambda(S^Q(\Gamma_0^Q(N))_{\ffrm}) =\sum_{q|Q}\frac{2n_q}{e}.
\] Thus we show how to use defects of Hecke rings to compute the defects of their ``cohomological’’ modules (arising from  the first cohomology of modular curves  and Shimura curves that they act on).  This also allows one to give a more robust approach   to the main result  proved by Ribet and Takahashi \cite[Theorem 1]{RiTa} that  computes changes of degrees of optimal parametrizations of  elliptic curves over $\Q$ by Shimura curves as one varies the Shimura curve. For all this, see Theorem \ref{change-of-eta}(ii),  Corollary \ref{degree} and Corollary \ref{component} below.

\subsection{Main ideas of proof of Theorem \ref{thm:intro}}

Let $C_\cO$  be the category of  tuples $(R, \lambda_R)$ with:

\begin{itemize}
	\item $R$ a complete, Noetherian  local $\cO$-algebras, with maximal ideal $\ffrm$ and  residue field $k=\cO/\sf m$;
	\item flat over $\cO$ and Cohen--Macaulay;
	\item together with an augmentation $\lambda_R: R \to \cO$ (that is by definition a  continuous surjective $\cO$-algebra homomorphism) that is formally smooth over the generic fiber.
\end{itemize}

The morphisms in the category $C_\cO$ are local homomorphisms of $\cO$-algebras compatible with the augmentation, namely
 local  $\cO$-algebra maps $f \colon R\to S$  such that $\lambda_Sf=\lambda_R$.  (As the augmentation considered will be clear from the context,  we   will often denote $\lambda_R$ by just $\lambda$, and also given a pair $(R,\lambda) \in C_\cO$ we  will sometimes write $R \in C_\cO$.)

We   take a  cue from a formula  discovered by Venkatesh \cite{venkatesh,venkatesh2}  (see Proposition \ref{prop:v} of the appendix) and define  in \S \ref{sec:Venkatesh}  the {\it  Wiles defect} $\delta_\lambda(R)$ for $ (R,\lambda)  \in C_\cO$. The defect  $\delta_\lambda(R)$  is expressed in terms of two invariants first introduced by Venkatesh:
\begin{description}
	\item[(i)] the  length of  the  $\cO$-module  $\cDer_\cO^1(R, E/\cO)$ which  can be directly defined using a continuous version of the Andr\'e--Quillen cohomology of rings (cf. \S \ref{sec:Der1}), (which will agree with the standard Andr\'e--Quillen cohomology module $\Der_\cO^1(R, E/\cO)$ in the case when $R$ has dimension $1$) and
	\item[(ii)] the length of the $\cO$-module  $C_{1,\lambda}(R)$ (cf. \S \ref{sec:c2}, in particular  Corollary \ref{well defined}).
\end{description}

The Wiles defect $\delta_\lambda(R)$  is then defined  (cf. Definition \ref{def:key})  to be \[\ds \delta_\lambda(R) = \frac{\log|\cDer^1_\cO(R,E/\cO)|-\log|C_{1,\lambda}(R)|}{\log|\cO/p|}=\frac{\ell_\cO(\cDer^1_\cO(R,E/\cO))-\ell_\cO(C_{1,\lambda}(R))}{\ell_\cO(\cO/p)}.\]  This definition of the defect we give  for   $R \in C_\cO$  agrees,   by Proposition~\ref{prop:v}  and Proposition \ref{prop:coincidence},  in the case when $R \in C_\cO$  is of dimension one  with the definition of the Wiles defect  given in \cite{BKM}  as \[\delta_\lambda(R) = \frac{\log|\Phi_\lambda(R)|-\log|\Psi_\lambda(R)|}{\log|\cO/p|}.\]

Our main technique for the proof of  Theorem \ref{thm:intro}  is the Taylor--Wiles--Kisin patching method. Specifically, under some mild global hypotheses, one can write $R^\st$ as a quotient $ R_{\loc}^{\St}[[x_1,\ldots,x_g]]/(y_1,\ldots,y_d)$ (see  Theorem \ref{thm:patching} and Theorem \ref{thm:mc}), where $R_{\loc}^{\St}$ is a completed tensor product of  local Galois deformation rings, and is thus determined by local Galois theoretic information. In the case when $R_{\loc}^{\St}$ is Cohen--Macaulay\footnote{Which happens in many cases in which the relevant deformation rings have been explicitly computed, including the case considered in \cite{BKM}, and is conjectured to hold far more generally.} we prove general results  (see Theorem \ref{thm:c_1 local} and  Theorem \ref{thm:Der^1 local}) that  imply that $\cDer^1_\cO(R,E/\cO)$ and $C_{1,\lambda}(R)$ are independent of the choice of ideal $(y_1,\ldots,y_d)$, and thus depend only on the ring $R_{\loc}^{\St}$ and the induced composite  map $R_{\loc}^{\St} \to R^{\St}  \xrightarrow{\lambda}\cO$, which shows that   \[
\delta_\lambda(R^{\St}) =  \delta_\lambda(R_{\loc}^{\st})=\sum_{q|Q} \delta_\lambda(R_q^\st)
\] where $\delta(R_q^\st)$  is the defect of the local deformation ring $R_q^\st \in C_\cO$.  Thus to determine $\delta_\lambda(R^\st)$,  we have to compute the   defects  $\delta_\lambda(R_q^\st)$ of the local deformation rings $R_q^\st$. These computations are quite elaborate  and are  done in Theorem \ref{thm:delta_v^St}     of  \S \ref{sec:computations} (Theorems \ref{thm:delta-PhiUn} and \ref{thm:delta-Un} do analogous computations for local deformation rings defined by conditions of being unipotent, and  unipotent together with a choice of Frobenius eigenvalue).

\subsection{Broader context}

We make some more informal remarks about the broader context of our work and  further questions to pursue in this context.

Our work is in the general context of understanding deformation rings $R$  when they are ``obstructed’’, and are thus not expected to be complete intersections. The Wiles defect is in a sense a measure of the obstructedness of $R$ at a given augmentation $\lambda: R \to \cO$. In the context of the present paper the obstructions are local in nature. The Wiles defect is a global quantity which in our case turns out to be a sum of local defects. This is proved by patching and showing that that the invariants  $\cDer_\cO^1(R, E/\cO)$ and  $C_{1,\lambda}(R)$ remain invariant under going modulo regular sequences. In other situations (as in \cite{TiUr}) the obstructedness  of deformation rings $R$ is because of  global reasons,  in that one is in a situation of positive defect $\ell_0>0$, and the natural ``automorphic cohomology’’ to consider lives in more than one degree. The work in \cite{gal-ven} gives a framework to understand this more deeply via considering derived deformation rings $\cal R$ such that $R=\pi_0(\cal R)$, and $\pi_*(\cal R)$ acts as a graded ring  on the ``automorphic cohomology’’. It seems interesting to explore these ideas in the context of the paper, and for instance  ``derive’’ the local deformation rings at trivial primes. 

We have not dealt with cases when the local deformation ring at $p$ is not a complete intersection in this paper, but our results will still be applicable provided that the local deformation rings are Cohen--Macaulay. For example \cite{Snowden} considers a fixed weight ordinary deformation ring when the residual representation is trivial at $p$ and shows that this ring is Cohen--Macaulay but not a complete intersection (or even Gorenstein). Our methods show that the global Wiles defect is again a sum of local defects in this case. However we have not been able to determine the local defect at $p$ in this case (due largely to the fact that \cite{Snowden} only computes the special fiber of the ring, while computing the local defect would require the integral version of the ring).

In the tame cases we have considered here and in \cite{BKM} the local defect  at $q$ is related to tame regulators  (in the sense of Mazur-Tate) of the $q$-adic   Mumford-Raynaud-Tate periods of  the corresponding abelian variety $A_f$ which has multiplicative reduction at $q$. In the  wild case one imagines that the local defect will be related  to $p$-adic regulators.

Our work should also help in formulating and proving Bloch--Kato conjectures for newforms $f \in S_2(\Gamma_0(N))$ (say $N$ squarefree) and  the $p$-part of  special value of the $L(1,\Ad)$ for the adjoint $L$-function of $f$ for suitable primes $p$.  The algebraic part of the $L$-value is traditionally  related to congruence modules  of $f$ by the work of Hida \cite{Hida}. The Selmer group  for the adjoint motive of $f$ can be related to the cotangent space at the augmentation $\lambda_f: R \to \cO$ where the local deformation problem at primes dividing  $N$ is the unipotent condition. The Wiles defect here by Theorem \ref{thm:mc} is $\sum_{q|N} n_q$, and  is the discrepancy between the length of  the congruence module for $f$ and the Selmer  group for the adjoint motive of $f$.  It will be interesting to see this defect emerge from automorphic considerations. We believe that the Selmer group we  are alluding to here  is the natural (primitive)  Selmer group to consider  for the adjoint motive of $f$, reflecting nature of $\pi_f$ locally at primes dividing $N$.   ( See  \cite[Theorem 5.20]{TiUr}   that relates the ratio of  different  integral normalizations of   periods (cohomological and motivic) of the adjoint motive of a Bianchi form    to the Wiles defect, and to Bloch-Kato conjectures.) Note that if we relax the Selmer conditions at primes dividing $N$ to be unrestricted of fixed determinant,  and consider the corresponding  imprimitive Selmer group, then the Wiles defect becomes  0 and one is in a setting where Wiles-type methods prove the Bloch-Kato conjecture for this imprimitive Selmer group.

 If we consider the Jacquet-Langlands correspondent  $g$ of $f$ on a quaternion algebra  $D_Q$ ramified  at the  set of primes $Q$  dividing $N$ which we assume   to be of even cardinality,  normalized (as in \cite{Prasanna})  using the schematic structure  over $\Z_p$ of the corresponding Shimura curve $X^Q$ over $\Q$,  with $p$ a prime such that $(p,N)=1$,  then one sees easily that the ratio of Petersson inner products \[\frac{(f,f)}{(g,g)}=\frac{\deg(\phi)}{\deg(\phi’)}\] where $\phi,\phi’$ are optimal parametrizations of  abelian varieties in the isogeny class ${\cal A}_f$ over $\Q$ associated to $A_f$. We could ask for a different ``natural’’ normalization $g’$ such that  \[\frac{(f,f)}{(g’,g’)}=\frac{\deg(\phi)}{\deg(\phi’)} \Pi_{q \in Q} \omega^{-2n_q}\] would be the change of the corresponding Selmer groups  (when we check the local conditions at primes in $Q$ from Steinberg to unrestricted with fixed determinant) and thus would incorporate the Wiles defect $\sum_{q \in Q} \frac{2n_q}{e}$.
 
 Our method to compute $p$-parts of change of degrees of parametrizations of elliptic curves over $\Q$ by Shimura curves gives results that are stronger than the ones which can be obtained  using the arithmetic-geometric methods of \cite{RiTa}. To have these results in the fullest possible  generality   is important for  Diophantine applications  (see \cite{Pasten}).

\subsection{Structure of this paper}

We begin by developing the commutative algebra tools that are needed for  our main theorem Theorem \ref{thm:mc}.
In \S \ref{sec:Venkatesh}  we state  a formula for Wiles defects of rings of dimension one that is proved in Appendix  \ref{app}. In  the key \S  \ref{sec:delta invariant} we define and prove properties of  the invariants 
$\cDer_\cO^1(R, E/\cO)$ and   $C_{1,\lambda}(R)$ for rings $R \in C_\cO$. In \S \ref{sec_deformation_theory} we summarize information about local and global deformation rings. In \S \ref{sec:computations} we compute the invariants defined in \S \ref{sec:delta invariant} for the local deformation rings we consider. This is a key input in  computing the Wiles defect of global deformation rings 
in Theorem \ref{thm:mc}. In \S \ref{sec:deformation} we use patching and the work in \S \ref{sec:delta invariant} to show that the Wiles defect of global deformation rings and Hecke rings  we consider is the sum of local defects. As the local defects have been computed in \S \ref{sec:computations} this allows us to complete the proof of our main Theorem \ref{thm:mc}.  In \S \ref{sec:ribet} we apply Theorem \ref{thm:mc} to compute the Wiles defect for modules over Hecke algebras that arise from their action on the cohomology of modular and Shimura curves. This also leads to  a new approach to, and strengthening of, the results in \cite{RiTa} about change of degrees of optimal parametrizations  of elliptic curves by Shimura curves as one changes the Shimura curve. 

In Appendix \ref{app} (written by N.~Fakhruddin and C.~Khare) a proof of a formula stated by Venkatesh is given, which has been proven earlier  in a special case in  \cite[Proposition 4]{TiUr}).

\subsection{Notation}

By $F$ we denote our base field, a totally real number field, by $F_v$ its completion at any place $v$ of $F$, and we choose algebraic closures $\overline F$ of $F$ and $\overline F_v$ if $F_v$ for all places $v$. These choices define the absolute Galois groups $G_F=\Gal(\overline F/F)$ and $G_{F_v}=\Gal(\overline F_v/F_v)$. We write $I_v\subset G_{F_v}$ for the inertia subgroup. We also fix embeddings $\overline F\to\overline F_v$, extending the canonical embeddings $F\to F_v$. This determines for each place $v$ of $F$ an embedding $G_{F_v}\to G_F$. By $\Frob_v$ we denote a Frobenius automorphism in $G_{F_v}$, that is unique up to $I_v$, and we also write $\Frob_v$ for its image in $G_F$. All representations of $G_F$ or of $G_{F_v}$ will be assumed to be continuous. If $v$ is a finite place of $F$, then we write $q_v$ for the cardinality of its residue field.

Throughout the paper, we fix a prime $p>2$, and we denote by $\overline \Q_p$ an algebraic closure of $\Q_p$. We will call a finite extension $E$ of $\Q_p$ inside $\overline\Q_p$ a coefficient field. For a coefficient field $E$, we let $\cO$ be its ring of integers, $k$ its residue field and $\varpi\in\cO$ a uniformizer. We write $\Sigma_p$ for the set of places of $F$ above~$p$.

The category of complete Noetherian local $\cO$-algebras with residue field $k$ is denoted by $\CNL_\cO$, and for any object $R$ in $\CNL_\cO$, we write $\ffrm_R\subset R$ for its maximal ideal. Each object $R\in\CNL_\cO$ will be endowed with its profinite ($\ffrm_R$-adic)~topology. By a complete Noetherian local $\cO$-algebra, we implicitly mean that its residue field is equal to $k$; we feel justified because our rings typically have an augmentation to~$\cO$.

We denote by $\varepsilon_p$ the $p$-adic cyclotomic character $\varepsilon_p:G_F \to \Z_p^\times$; if we compose $\varepsilon_p$ on the right with any map $G_{F_v}\to G_F$ or on the left with $\Z_p^\times\to R^\times$, induced from any morphism $\Z_p\to R$ in $\CNL_{\Z_p}$, then we also write $\varepsilon_p$, by slight abuse of notation.

For an $\cO$-algebra $R$, an augmentation $\lambda$ of $R$ will always mean a surjective $\cO$-algebra homomorphism $\lambda:R\to \cO'$, where $\cO'$ is the ring of integers in a finite extension of $E$ (we will almost always take $\cO=\cO'$). For an $\cO$-module $M$ that is a finite abelian group, we denote by $\ell_\cO(M)$ the length of $M$ as an $\cO$-module.
 For $\alpha \in \cO$, we denote by $\ord_\cO(\alpha)=\ell_\cO(\cO/(\alpha))$.

For a Galois representation $\rhobar:G_{F}\to \GL_2(\overline{\F}_p)$ which is finite flat at $p$, we will let $N(\rhobar)$ represent its Artin conductor.

\subsection{Acknowledgements} We would like to thank  Najmuddin Fakhruddin, Tony Feng, Michael Harris, Srikanth Iyengar,   Akshay Venkatesh for helpful discussions related to this paper. G.B. acknowledges support by Deutsche Forschungsgemeinschaft  (DFG) through CRC-TR 326 `Geometry and Arithmetic of Uniformized Structures', project number 444845124.

\section{Wiles defect for  rings of dimension one}\label{sec:Venkatesh}

In this section we state results from the Appendix \ref{app} in  the form in which they are used in the paper, and also with a view  to generalizing these results to higher dimensional rings in \S \ref{sec:delta invariant}.

For any ring $R$, any ideal $I\subseteq R$ and any $R$-module $M$, we will always use $M[I]\subseteq M$ for the submodule of $I$-torsion elements of $M$. In particular, $R[I]=\Ann_R(I)\subseteq R$ is the annihilator of the ideal $I$.

If $M$ is a finitely generated $R$-module, with generating set $m_1,\ldots,m_n$ inducing a surjection $R^n\onto M$, then we will let $\Fitt_R(M)\subseteq R$ (called the $0^{\mathrm{th}}$ \emph{fitting ideal}) denote the ideal generated by all elements of the form $\det(v_1,\ldots,v_n)\in R$ for $v_1,\ldots,v_n\in \ker(R^n\onto M)$. It is well known that this is independent of the choice of generating set $m_1,\ldots,m_n$, and moreover that $\Fitt_R(M)\subseteq \Ann_R(M)$. When the ring $R$ is clear from context we will sometimes write $\Fitt(M)$ in place of $\Fitt_R(M)$.

Let $R$ be a complete, local Noetherian $\cO$-algebra with $\dim(R) = 1$, and assume that $R$ is finite over $\cO$. Let $\lambda:R\onto\cO$ be any \emph{augmentation} (i.e.\ surjective $\cO$-algebra homomorphism). Let $R^{\tf}$ be the maximal $\varpi$-torsion free quotient of $R$, which is automatically finite free over $\cO$.\footnote{In the rest of this paper we will always work in the case where $R=R^{\tf}$, but we still state the general version in this section for the sake of completeness.} Also use $\lambda$ to denote the augmentation $R^{\tf}
\onto\cO$ induced by $\lambda$. Define
\[\Phi_\lambda(R) = (\ker \lambda)/(\ker\lambda)^2 = \Omega_{R/\cO}\otimes_\lambda\cO\]
and
\[\Psi_\lambda(R) = \cO/\eta_\lambda(R) = \cO/(\lambda(R^{\tf}[\ker \lambda])),\]
which we will call the \emph{cotangent space} and \emph{congruence module} of $R$ (with respect to $\lambda$). From now on we will assume that $\Phi_\lambda(R)$ is finite, which geometrically means that $\lambda$ is smooth on the generic fiber of~$R$.

In \cite{BKM} we define the \emph{Wiles defect} of $R$ with respect to $\lambda$ to be
\[\delta_\lambda(R) = \frac{\log|\Phi_\lambda(R)|-\log|\Psi_\lambda(R)|}{\log|\cO/p|}= \frac{\ell_\cO(\Phi_\lambda(R))-\ell_\cO(\Psi_\lambda(R))}{\ell_\cO(\cO/p)},\]
which is known to be a non-negative rational number. Moreover (cf. \cite{Wiles,Lenstra}) $\delta_\lambda(R) = 0$ if and only if $R=R^{\tf}$ and $R$ is a complete intersection. To deduce this from the literature we note that  $\delta_\lambda(R)=0$ implies that the map $R \to R^{\tf}$ is an isomorphism of complete intersections.  (The reason for the normalization factor of $\log|\cO/p|$ is to ensure that $\delta_\lambda(R)$ is invariant under expanding the coefficient ring $\cO$.)

Venkatesh, in an unpublished note \cite{venkatesh}, observed that $\delta_\lambda(R)$ can be expressed in terms of two other invariants of $R$ (see Appendix \ref{app} of this paper, for a detailed proof of  a more general version of Venkatesh's observation).

First, let $R$ act on $E/\cO$ through its quotient $R\xrightarrow{\lambda} \cO$. Venkatesh's first invariant is simply the first Andr\'e--Quillen cohomology group $\Der^1_\cO(R,E/\cO)$.

To define Venkatesh's second invariant, we will fix an $\cO$-algebra $\Rt$ and a surjection $\varphi:\Rt\onto R$ with the properties that
\begin{itemize}
	\item $\Rt$ is a complete intersection of dimension $1$, finite free over $\cO$.
	\item $\Phi_{\lambda\circ\varphi}(\Rt)$ is finite.
\end{itemize}
(such a ring always exists, as explained in Appendix  \ref{app}). When there is no chance of confusion we will also use $\lambda$ to denote the induced map $\lambda\circ\varphi:\Rt\onto R\onto \cO$. 

Now write $I=\ker\varphi$ so that $\lambda(I) = 0$. As $\Rt$-modules, we have that $\Fitt(I)\subseteq \Rt[I]$, and hence $\lambda(\Fitt(I))\subseteq \lambda(\Rt[I])$ as ideals of $\cO$ (and in fact, both of these ideals are nonzero as explained in Appendix \ref{app}). We then define Venkatesh's second invariant to be the cyclic $\cO$-module
\[
C_{1,\lambda}(R) = \lambda\left(\Rt[I]\right)/\lambda\left(\Fitt(I)\right).
\]
\emph{A priori}, this looks like it will depend on the choice of complete intersection $\Rt$, but the work of Appendix  \ref{app} shows that it in fact depends only on $R$ and $\lambda$. The main result Proposition \ref{prop:v} of Appendix  \ref{app} is the following formula for the Wiles defect $\delta_\lambda(R)$. We  recall as noted earlier  that \cite[Proposition 4]{TiUr} proves a particular case  (when $C_{1,\lambda}(R)$ is trivial)  of this formula.

\begin{theorem}[see \ref{prop:v}]\label{thm:Venkatesh}
If $R$ and $\lambda:R\onto \cO$ are as described above, and $\Phi_\lambda(R)$ is finite, then
\[
\frac{|\Der^1_\cO(R,E/\cO)|}{|C_{1,\lambda}(R)|} = \frac{|\Phi_\lambda(R)|}{|\Psi_\lambda(R)|}.
\]
In particular, $\ds \delta_\lambda(R) = \frac{\log|\Der^1_\cO(R,E/\cO)|-\log|C_{1,\lambda}(R)|}{\log|\cO/p|}$.
\end{theorem}

\begin{remark}
In practice, one is often interested in the Wiles defect $\delta_\lambda(M)$ (as defined in \cite[Section 3]{BKM}) of a particular module $M$ over $R$, as well as, or instead of $\delta_\lambda(R)$. However in many cases relevant to us, the results of \cite{BKM} imply that $\delta_\lambda(R) = \delta_\lambda(M)$, so we will focus mainly on $\delta_\lambda(R)$ in this paper, except in \S \ref{sec:ribet} in which we apply Theorem \ref{thm:mc} which determines defects of Hecke rings to detect the defect of modules that they act on.

We do suspect that there may exist some generalization of Theorem \ref{thm:Venkatesh} which would directly express $\delta_\lambda(M)$ in terms of similar invariants. Such a generalization would allow us to directly study $\delta_\lambda(M)$ in cases when we can not prove it is equal to $\delta_\lambda(R)$, and could possibly work in cases when the results of this paper do not apply. The results of \cite[Theorem 1.2]{BIK} support such a suspicion. 
\end{remark}

\section{Wiles defect for higher dimensional Cohen--Macaulay rings}\label{sec:delta invariant}

We begin this section by remarking that the  definition of the Wiles defect $\delta_\lambda(R)$ in \cite{BKM}, which depends on finiteness of $\Phi_\lambda(R)$,  makes sense  for a complete Noetherian, Cohen--Macaulay local $\cO$-algebra  $R$   only when $R$ is of dimension one.

\begin{lemma}\label{minimal}
 Let $R$ be  a complete Noetherian local $\cO$-algebra  together with an augmentation $\lambda: R \to \cO$,  such that  $\Phi_\lambda(R)$ is a finite abelian group then $\ker(\lambda)$ is a minimal prime ideal. If we further assume that $R$ is Cohen--Macaulay then $R$ is of dimension one.
\end{lemma}

\begin{proof}
Let $\ker(\lambda)=\frak p$ and we observe that the localization $R_{\frak p}$ is a local ring with maximal ideal $m=\frak pR_{\frak p}$ and infinite residue field $E$, and by our assumption that $\ker(\lambda)/\ker(\lambda)^2$ is finite we deduce that $m=m^2$ and thus $m=0$. This implies that $R_\frak p$ is a field, and thus $\frak p$ is a minimal prime ideal of $R$. As Cohen--Macaulay rings are equidimensional,  we deduce the last statement of the lemma.
\end{proof}

In light of this  we  give a  definition   (cf. Definition \ref{def:key}) of  the {\it  Wiles defect} $\delta_\lambda(R)$ for $R \in C_\cO$  motivated by the Venkatesh formula of the defect $\delta_\lambda(R)$ for $R \in C_\cO$ when $R$ is one-dimensional. This is expressed in terms of:
\begin{description}
	\item[(i)] the invariant $\cDer_\cO^1(R,E/\cO)$ which can be directly defined using a continuous version of the Andr\'e--Quillen cohomology of rings (cf. \S \ref{sec:Der1}, in particular  Theorem \ref{thm:Der^1 local}); and
	\item[(ii)] the invariant $C_{1,\lambda}(R)$ (cf. \S \ref{sec:c2}, in particular  Corollary \ref{well defined})  that  is defined   by reduction, via  quotienting by regular sequences,  to the case of dimension 1 treated in  the Appendix \ref{app}, and where it was originally defined by Venkatesh. 
\end{description}

The Wiles defect $\delta_\lambda(R)$ is then defined  (cf. Definition \ref{def:key}) via the formula  \[\ds \delta_\lambda(R) = \frac{\log|\cDer^1_\cO(R,E/\cO)|-\log|C_{1,\lambda}(R)|}{\log|\cO/p|}.\] 

In the case when $R$ is of dimension one this definition of the defect for   $R \in C_\cO$ agrees by Theorem~\ref{thm:Venkatesh}  and Proposition \ref{prop:coincidence}, with the definition of the Wiles defect  defined in \cite{BKM}  as \[\delta_\lambda(R) = \frac{\log|\Phi_\lambda(R)|-\log|\Psi_\lambda(R)|}{\log|\cO/p|}.\] 
(Note that when $R$ is of dimension one, the finiteness of   $|\Phi_\lambda(R)|$ is equivalent to saying that  $\lambda: R \to \cO$ is formally smooth over the generic fiber.)  

We  show  below the independence  of the invariants we define under forming quotients by regular sequences: this  is used to even define the invariant $C_{1,\lambda}(R)$. We also provide formulas for the invariants in terms of certain complete intersection rings that surject onto $ R \in C_\cO$, similar to the treatment in the appendix, but in higher dimensions.

\subsection{Complete intersection (CI)  covers}\label{sec:setup}

Consider $R \in C_\cO$, and a power series ring $S = \cO[[y_1,\ldots,y_d]]$. We will say that an inclusion $\theta:S \into R$ satisfies \ref{prop P} the following conditions hold:

\setproptag{(P)}
\begin{property}\label{prop P}\ 
\begin{itemize}
	\item $\theta:S \into R$ is a continuous $\cO$-algebra homomorphism.
	\item $\theta$ makes $R$ into a finite free $S$-module (so that $(\theta(y_1),\ldots,\theta(y_d),\varpi)$ is a regular sequence for $R$).
	\item $(\theta(y_1),\ldots,\theta(y_d))\subseteq \ker \lambda$.
	\item If $R_\theta = R/(\theta(y_1),\ldots,\theta(y_d)) = R \otimes_{S}\cO$ and $\lambda_\theta:R_\theta\onto \cO$ is the map induced by $\lambda$, then $\Phi_{\lambda_\theta}(R_\theta)$ is finite.
\end{itemize}
\end{property}
\begin{proposition}\label{prop ThetaExists}
A map $\theta$ satisfying property \ref{prop P} exists.
\end{proposition}
We will give the proof later in this subsection after the proof of Lemma~\ref{lemma CI cover}.

\smallskip

Given a $\theta$ satisfying \ref{prop P}, we will usually identify $S$ with its image in $R$, so that in particular $y_1,\ldots,y_d\in R$. As in the statement of Property \ref{prop P}, we will let $R_\theta = R/(y_1,\ldots,y_d) = R \otimes_{S}\cO$ and let $\lambda_\theta:R_\theta\onto\cO$ denote the map induced by $\lambda$. Note that by the second condition of Property \ref{prop P}, the ring $R_\theta$ is a finite and free  module over $\cO$.

Similarly to Section \ref{sec:Venkatesh}, we will sometimes need to write $R$ as a quotient of some auxiliary complete intersection. We will say that a triple $(\Rt, I, \varphi)$ satisfies \ref{prop CI} if:
\setproptag{(CI)}
\begin{property}\label{prop CI}\ 
\begin{itemize}
	\item $\Rt$ is a complete, Noetherian local $\cO$-algebra, flat and equidimensional over $\cO$ of relative dimension~$d$.
	\item $\Rt$ is a complete intersection.
	\item $\varphi:\Rt\to R$ is a continuous surjection of $\cO$-algebras with $I=\ker \varphi$.
	\item The point corresponding to $\lambda\circ \varphi$ in $\Spec \Rt[1/\varpi]$ is a formally smooth point.
\end{itemize}
If $(\Rt,I,\varphi)$ satisfies \ref{prop CI} we will use $\lambdat$ to denote the composition $\lambda\circ\varphi:\Rt\onto R\onto\cO$.
\end{property}

Given $(\Rt,I,\varphi)$ with Property \ref{prop CI}, Theorem~\ref{thm:c_1 local} will give a formula for $C_{1,\lambda}(R)$ in terms of this triple, independent of any $\theta$, and Theorem~\ref{thm:Der exact sequence} will give a similar result for~$\cDer_\cO^1(R,E/\cO)$.

\begin{proposition}\label{prop CI exists}
For any pair $(R,\lambda)$ there exists a triple $(\Rt,I,\varphi)$ satisfying Property \ref{prop CI}.
\end{proposition}
Proposition~\ref{prop CI exists} will be a direct consequence of Lemma~\ref{lemma CI cover}. We begin with the following~lemma:
\begin{lemma}\label{lemma ModifyRseq}
Let $S$ be a complete, Noetherian local $\cO$-algebra with an augmentation $\lambda\colon S\to \cO$ and let $d>0$. Suppose that $S[1/\varpi]$ is formally smooth at $\lambda$ of dimension $n\ge d$ and that there are elements $f_1,\ldots,f_d\in \kernel \lambda$ such that $f_1,\ldots,f_d,\varpi$ is a regular sequence in $S$. Then there exist $h_1,\ldots,h_d\in (\kernel \lambda \cap (f_1,\ldots,f_d,\varpi))$ such that $h_1,\ldots,h_d,\varpi$ is a regular sequence in $S$ and such that for $A=S/(h_1,\ldots,h_d)$ and the induced augmentation $\lambda_A\colon A\to \cO$, the ring $A[1/\varpi]$ is formally smooth at $\lambda_A$ of dimension $n-d$.
\end{lemma}
\begin{proof}
By replacing $(f_1,\ldots,f_d)$ by $(f_1^2,\ldots,f_d^2)$ we may assume that $(f_1,\ldots,f_d)\subset \kernel\lambda^2$; see \cite[15.A, Theorem~26]{Matsumura-CA}. Write  $S[1/\varpi]$ for the localization of $S$ at $\varpi$ and $\widehat{S[1/\varpi]}$ for the completion of the latter at the point corresponding to $\lambda$. By our hypothesis, the ring $\widehat{S[1/\varpi]}$ is a power series ring over $E$ in $n\ge d$ indeterminates. Let  $\widehat I$ denote its maximal ideal. Choose $g_1,\ldots,g_d$ in $\ker\lambda$ whose images in $\widehat I/\widehat I^2$ are linearly independent over $E$. Then $(h_1,\ldots,h_d)$ with $h_i=f_i+\varpi g_i$ has all properties required.
\end{proof}

\begin{lemma}\label{lemma CI cover}
Suppose $B$ is a complete, Noetherian local $\cO$-algebra with $\dim B=d+1$ and $\dim B/\varpi=d$, and $\lambda:B\to\cO$ is an augmentation, such that $\Spec B[1/\varpi]$ is formally smooth at~$\lambda$ of dimension~$d$. Then there exists a Noetherian $\cO$-algebra $A$ and a surjective homomorphism $\pi\colon A\to B$ such that the following~holds:
\begin{enumerate}
\item The ring $A$ is local and complete, a complete intersection, flat over $\cO$ and of relative dimension $d$.
\item The map $\pi[1/\varpi]\colon A[1/\varpi]\to B[1/\varpi]$, obtained from $\pi$ by inverting $\varpi$, induces an isomorphism after completion at the points corresponding to the augmentations $\lambda$ and $\mu=\lambda \pi\colon A\to\cO$, respectively. In particular, $\Spec A[1/\varpi]$  is formally smooth at $\mu$ of dimension~$d$.
\end{enumerate}
\end{lemma}
\begin{proof}
Let $\Pi\colon S=\cO[[z_1,\ldots,z_n]]\to B$ be a surjective ring homomorphism. Let $\frp_\lambda\subset \ffrm_B$ be the prime ideal $\ker\lambda$, and denote by $\frq_\lambda\subset\ffrm_S$ its inverse image under $\Pi$, i.e. $\frq_\lambda=\ker \lambda\Pi$. Let $m=n-d\ge0$.

By hypothesis $B/\varpi$ has dimension $d$. Because $S$ is $\varpi$-torsion free and $S/\varpi$ is regular, we can find a regular sequence $(f_1,\ldots,f_m)$ in $\ker \Pi\subset S$ such that $(f_1,\ldots,f_m,\varpi)$ is a regular sequence. Because $S[1/\varpi]$ is regular of dimension $n$ and $\frq_\lambda[1/\varpi]$ is a maximal ideal of that ring, the ring  $S[1/\varpi]$ is formally smooth at $\frq_\lambda[1/\varpi]$ of dimension $n$.

It follows from Lemma~\ref{lemma ModifyRseq} that there exist $h_1,\ldots,h_m\in \kernel \Pi+\varpi S$ such that $h_1,\ldots,h_m,\varpi$ is a regular sequence in $S$ and such that for $A=S/(h_1,\ldots,h_m)$ and induced augmentation $\lambda_A\colon A\to \cO$ the ring $A[1/\varpi]$ is formally smooth at $\lambda_A$ of dimension $n-m=d$. It follows that one has an induced surjection $A\to B$ where $A$ is a local complete, complete intersection $\cO$-algebra, flat over $\cO$ of relative dimension $d$, and that the induced surjection $A[1/\varpi]\to B[1/\varpi]$ becomes an isomorphism after completion at $\frq_\lambda[1/\varpi]$.
\end{proof}
\begin{proof}[Proof of  Proposition~\ref{prop ThetaExists}]
Because $R$ is Cohen--Macaulay and flat over $\cO$ of relative dimension $d$, we can find a regular sequence $\varpi,f_1,\ldots,f_d$ in $R$. If we replace each $f_i$ by an element in $f_i+\varpi R$ the resulting sequence is again regular. Now using that $\kernel \lambda$ together with $\varpi$ generate the maximal ideal of $R$, we may assume that $f_1,\ldots,f_d$ lie in $\kernel \lambda$. Again by hypothesis $R[1/\varpi]$ is Cohen--Macaulay of dimension $d$ and formally smooth at $\lambda$, and hence it is formally smooth at $\lambda$ of dimension~$d$.

Then by Lemma~\ref{lemma ModifyRseq}, there exist $h_1,\ldots,h_d\in \kernel \lambda$ such that $h_1,\ldots,h_d,\varpi$ is a regular sequence in $R$ and such that for $B=R/(h_1,\ldots,h_d)$ and the induced augmentation $\lambda_B\colon B\to \cO$ the ring $B[1/\varpi]$ is formally smooth at $\lambda_B$ of dimension $0$. It follows that the continuous $\cO$-algebra map $\theta\colon S=\cO[[y_1,\ldots,y_d]]\to R$ with $y_i\mapsto f_i$ makes $R$ into a finite free $S$-module, such that in the notation of \ref{prop P}, we have $B=R_\theta$ and $\lambda_B=\lambda_\theta$, and moreover $R_\theta$ is finite free over $\cO$. Hence $R_\theta[1/\varpi]$ is a product of Artin $E$-algebras, and the smoothness at $\lambda_\theta$ shows that the component corresponding to $\lambda_\theta$ is equal to $E$. From this it follows that $\Phi_{\lambda_\theta}(R_\theta) =\kernel \lambda_\theta/(\kernel\lambda_\theta)^2$ is of finite $\cO$-length, as it is finitely generated over $\cO$ and $\cO$-torsion.
\end{proof}
Next we observe that we can lift regular sequences of $R$ along $\Rt \to R$.
\begin{lemma}\label{lem:lifting reg seq}
Assume that $\theta:S \into R$ satisfies \ref{prop P} and $(\Rt,I,\varphi)$ satisfies \ref{prop CI}. Then $\theta$ lifts to a morphism $\thetat:S \to \Rt$ (making $\varphi$ into a $S$-algebra homomorphism) which makes $\Rt$ into a finite free $S$-module. That is, identifying $S$ with its image in $\Rt$, that $(y_1,\ldots,y_d,\varpi)$ is a regular sequence for both $\Rt$ and $R$.

Moreover if $\Rt_\theta = \Rt/(y_1,\ldots,y_d)$ and $\lambdat_\theta:\Rt_\theta\onto \cO$ is the map induced by $\lambdat$, then $\Rt_\theta$ is a complete intersection of dimension $1$, finite free over $\cO$, and $\Phi_{\lambdat_\theta}(\Rt_\theta)$ is finite.
\end{lemma}
This will follow from the following lemma:
\begin{lemma}\label{lem:lifting element}
	Let $A$ be a Noetherian local ring, and let $B = A/I$ for some ideal $I$ of $A$. Let $x\in \ffrm_B$ be an element not contained in any minimal prime of $B$. Then $x$ lifts to an element $\xt\in\ffrm_A$ which is not contained in any minimal primes of $A$.
\end{lemma}
\begin{proof}
	Pick any lift $\xt_0\in \ffrm_A$ of $x$. Let the set of minimal primes of $A$ be $\{P_1,\ldots,P_n\}$, labeled so that there is some $0\le a\le n$ for which $\xt_0\not\in P_1,P_2,\ldots,P_a$, and $\xt_0\in P_{a+1},\ldots,P_n$.
	
	Now fix any $i>a$, so that $\xt_0\in P_i$. Note that if $I\subseteq P_i$ then $P_i/I$ would be a minimal prime of $B$ containing $x$, contradicting our assumption. Hence $I\not\subseteq P_i$ and so there is some $r_i\in I\sm P_i$. 
	
	Also for any $j\ne i$, $P_j\not\subseteq P_i$, and so there is some $s_{ij}\in P_j\sm P_i$. Now define
	\[y_i:= r_i\prod_{j\ne i} s_{ij}\]
	so that $y_i\in I$, $y_i\in P_j$ for $j\ne i$ and $y_i\not\in P_i$. Finally let
	\[\xt = \xt_0+y_{a+1}+y_{b+2}+\cdots+y_{n}.\]
	Then we have $\xt\equiv \xt_0\equiv x\mod{I}$, $\xt\equiv \xt_0\not\equiv 0\mod{P_i}$ for $i\le a$ and $\xt\equiv y_i\not\equiv 0\mod{P_i}$ for $i>a$. So $\xt$ is our desired lift.
\end{proof}

\begin{proof}[Proof of Lemma \ref{lem:lifting reg seq}]
Identifying $S$ with its image in $R$, we get that $(y_1,\ldots,y_d,\varpi)$, and thus $(\varpi,y_1,\ldots,y_d)$, is a regular sequence for $R$. We claim that we can inductively construct a sequence $\yt_1,\yt_2,\ldots,\yt_d\in \Rt$ such that $\varphi_\infty(\yt_i) = y_i$ for all $i$ and $\dim \Rt/(\varpi,\yt_1,\ldots,\yt_j) = d-j  = \dim R/(\varpi,y_1,\ldots,y_j)$ for all $0\le j\le d$.

As $\Rt$ and $R$ are both flat over $\cO$ of relative dimension $d$, we have $\dim \Rt/(\varpi) = d = \dim R/(\varpi)$. Now assume that $\yt_1,\ldots,\yt_j$ have been constructed for some $j<d$. Let $A_j = \Rt/(\varpi,\yt_1,\ldots,\yt_j)$ and $B_j = R/(\varpi,y_1,\ldots,y_j)$, so that $\varphi_\infty:\Rt\to R$ induces a map $\varphi_j:A_j\to B_j$. As $(\varpi,y_1,\ldots,y_d)$ is a regular sequence for $R$, $y_{j+1}$ is by definition not a zero divisor in $B_j$, and so in particular cannot be contained in any minimal primes of $B_j$. By Lemma \ref{lem:lifting element} it follows there is some $y_{j+1}'\in A_j$ with $\varphi_j(y_{j+1}') = y_{j+1}$ which is not contained in any minimal prime of $A_j$. Let $\yt_{j+1}\in \Rt$ be any lift of $y_{j+1}'$. But now
\[\Rt/(\varpi,\yt_1,\ldots,\yt_j,\yt_{j+1})\cong A_j/(y_{j+1}')\]
which has dimension $\dim A_j - 1 = d-(j+1)$, by the assumption that $y_{j+1}'$ is not contained in any minimal prime of $A_j$. This completes the induction.

Now $(\varpi,\yt_1,\ldots,\yt_d)$ is a system of parameters for $\Rt$. As $\Rt$ is a complete intersection, and thus Cohen--Macaulay, it follows that $(\varpi,\yt_1,\ldots,\yt_d)$, and thus $(y_1,\ldots,y_d,\varpi)$, is a regular sequence for $\Rt$.

So now defining $\thetat:S\to \Rt$ by $\thetat(y_i)=\yt_i$ makes $\Rt$ into a finite free $S$ module, as desired.

The fact that $\Rt_\theta$ is a complete intersection of dimension $1$, and finite free over $\cO$, now follows immediately from the fact that $\Rt_\theta$ was a complete intersection. For the last assertion, the proof of \cite[Theorem 7.16]{BKM} gives rise to a commutative diagram with exact rows:
\begin{center}
	\begin{tikzpicture}
	\node(11) at (1,2) {$\cO^d$};
	\node(21) at (4,2) {$\Phi_{\lambdat}(\Rt)$};
	\node(31) at (7,2) {$\Phi_{\lambdat_\theta}(\Rt_\theta)$};
	\node(41) at (9,2) {$0$};
	
	\node(10) at (1,0) {$\cO^d$};
	\node(20) at (4,0) {$\Phi_{\lambda}(R)$};
	\node(30) at (7,0) {$\Phi_{\lambda_\theta}(R_\theta)$};
	\node(40) at (9,0) {$0$};
	
	\draw[-latex] (11)--(10) node[midway,right]{$=$};
	\draw[-latex] (21)--(20);
	\draw[-latex] (31)--(30);

	\draw[-latex] (11)--(21) node[midway,above]{$\Thetat$};;
	\draw[-latex] (21)--(31);
	\draw[-latex] (31)--(41);
	
	\draw[-latex] (10)--(20) node[midway,above]{$\Theta$};;
	\draw[-latex] (20)--(30);
	\draw[-latex] (30)--(40);
	\end{tikzpicture}
\end{center}
where $\Phi_{\lambda}(R) = (\ker \lambda)/(\ker\lambda)^2 = \cOmega_{R/\cO}\otimes_\lambda\cO$ and $\Phi_{\lambdat}(\Rt) = (\ker \lambdat)/(\ker\lambdat)^2 = \cOmega_{\Rt/\cO}\otimes_{\lambdat}\cO$, and the maps $\Theta$ and $\Thetat$ are given in terms of differentials by $e_i\mapsto dy_i$.

Now as in \cite[Theorem 7.16]{BKM}, the fact that $\Spec R[1/\varpi]$ and $\Spec\Rt[1/\varpi]$ are both equidimensional of dimension $d$ and $\lambda$ and $\lambdat$, respectively, correspond to formally smooth points on these schemes, implies that $\Phi_{\lambda}(R)$ and $\Phi_{\lambdat}(\Rt)$ both have rank $d$ as $\cO$-modules.

But now the fact that $\Phi_{\lambda_\theta}(R_\theta)$ is finite implies that $\Theta$ must be injective. By commutativity, this implies that $\Thetat$ is also injective, which in turn implies that $\Phi_{\lambdat_\theta}(\Rt_\theta)$ is also finite.
\end{proof}

\subsection{Invariance of $C_{1,\lambda_\theta}(R_\theta)$ of $\theta$}\label{sec:c2}

For this section we will fix $\theta$ satisfying \ref{prop P} and $(\Rt, I,\varphi)$ satisfying \ref{prop CI}. We will let $\thetat:S \into \Rt$ be a lift of $\theta$ satisfying the conclusion of Lemma \ref{lem:lifting reg seq}, and we will identify $S$ with its images in $R$ and $\Rt$.

Let $\Rt_\theta$ and $\lambdat_\theta$ be as in Lemma \ref{lem:lifting reg seq}, and let $\varphi_\theta = \varphi \otimes_{S}\cO:\Rt_\theta \onto R_\theta$ (so that $\lambdat = \lambda\circ\varphi$) and let $I_\theta = \ker \varphi_\theta\subseteq \Rt_\theta$. Also let $\pi_\theta:\Rt \to \Rt_\theta$ be the quotient map, so that $\lambdat = \lambdat_\theta\circ \pi_\theta$ and $I_\theta = \pi_\theta(I)$.

The ring $\Rt_\theta$ now satisfies the conditions from Section \ref{sec:Venkatesh}, so we have
\[C_{1,\lambda_\theta}(R_\theta) = \lambdat_\theta\left(\Rt_\theta[I_\theta]\right)/\lambdat_\theta\left(\Fitt(I_\theta)\right).\]

The main result of this subsection is the following:

\begin{theorem}\label{thm:c_1 local}
We have the following:
\begin{enumerate}
	\item $\Rt_\theta[I_\theta] = \pi_\theta(\Rt[I])$
	\item $\Fitt(I_\theta) = \pi_\theta(\Fitt(I))$
\end{enumerate}
So in particular,
\[C_{1,\lambda_\theta}(R_\theta) = \lambdat_\theta\left(\Rt_\theta[I_\theta]\right)/\lambdat_\theta\left(\Fitt(I_\theta)\right) =
 \lambdat_\theta\left(\pi_\theta(\Rt[I])\right)/\lambdat_\theta\left(\pi_\theta(\Fitt(I))\right)
=
\lambdat\left(\Rt[I]\right)/\lambdat\left(\Fitt(I)\right),\]
which depends only on $\Rt$, $R$ and $\lambdat:R \to \cO$, all of which are independent of $\theta$.

Thus if we define  $C_{1,\lambdat}(\Rt)=\lambdat\left(\Rt[I]\right)/\lambdat\left(\Fitt(I)\right)$,
then we have  \[
C_{1,\lambdat}(\Rt)=C_{1,\lambda_\theta}(R_\theta).\]
\end{theorem}

\begin{proof}[Proof of Theorem \ref{thm:c_1 local}(1)]
Clearly we have $\pi_\theta(\Rt[I]) \subseteq \Rt_\theta[I_\theta]$ (since $I_\theta = \pi_\theta(I)$ and so $\Rt_\theta[I_\theta] = \Rt[I]$), so it suffices to prove that $\pi_\theta|_{\Rt[I]}:\Rt[I]\to \Rt_\theta[I_\theta]$ is surjective.

We first note that as $\Rt$ and $\Rt_\theta$ are complete intersections, and thus are Gorenstein, we get the following:

\begin{lemma}\label{lem:gorenstein}
There are isomorphisms $\Psi:\Rt\xrightarrow{\sim} \Hom_{S}(\Rt,S)$ and $\Psi_\theta:\Rt_\theta\xrightarrow{\sim} \Hom_{\cO}(\Rt_\theta,\cO)$ of $\Rt$-modules, fitting into a commutative diagram:
\begin{center}
	\begin{tikzpicture}
	\node(11) at (0,2) {$\Rt$};
	\node(21) at (3,2) {$\Hom_{S}(\Rt,S)$};
	
	\node(10) at (0,0) {$\Rt_\theta$};
	\node(20) at (3,0) {$\Hom_{\cO}(\Rt_\theta,\cO)$};
	
	\draw[-latex] (11)--(10) node[midway,right]{$\pi_\theta$};
	\draw[-latex] (21)--(20) node[midway,right]{$\sigma$};

	\draw[-latex] (11)--(21) node[midway,above]{$\Psi$};
	
	\draw[-latex] (10)--(20) node[midway,above]{$\Psi_\theta$};
	\end{tikzpicture}
\end{center}
where the vertical map $\sigma:\Hom_{S}(\Rt,S)\to \Hom_{S}(\Rt,\cO) = \Hom_{\cO}(\Rt_\theta,\cO)$ is just composition with the map $S \to S/(y_1,\ldots,y_d)=\cO$.
\end{lemma}
\begin{proof}
As $\Rt$ is Cohen--Macaulay and free of finite rank over $S$, we have $\omega_{\Rt} \cong \Hom_{S}(\Rt,S)$. But as $\Rt$ is a complete intersection, it is Gorenstein, and so $\omega_{\Rt}\cong \Rt$. Composing these isomorphisms gives the desired isomorphism $\Psi:\Rt \xrightarrow{\sim} \Hom_{S}(\Rt,S)$.

Now note that (as $\Rt$ is a free $S$-module):
\begin{align*}
\Psi(\ker \pi_\theta) 
&= \Psi(y_1\Rt+\cdots+y_d\Rt) 
 = y_1\Psi(\Rt)+\cdots+y_d\Psi(\Rt)\\
&= y_1\Hom_{S}(\Rt,S)+\cdots+y_d\Hom_{S}(\Rt,S)
 = \Hom_{S}(\Rt,y_1S+\cdots +y_dS)\\
&= \ker \sigma,
\end{align*}
which implies that there is an injection $\Psi_\theta:\Rt\xrightarrow{\sim} \Hom_{\cO}(\Rt,\cO)$ making the above diagram commute. As $\sigma$ is clearly surjective (since $\Rt$ is a projective $S$-module), it follows that $\Psi_\theta$ is also surjective.
\end{proof}

\begin{lemma}\label{lem:Psi(R[I])}
We have 
\[\Psi(\Rt[I]) = \{f:\Rt \to S|f(I) = 0\} = \Hom_{S}(\Rt/I,S)\]
and
\[\Psi_\theta(\Rt_\theta[I_\theta]) = \{f:\Rt_\theta\to \cO|f(I_\theta) = 0\} = \Hom_{\cO}(\Rt_\theta/I_\theta,\cO) = \Hom_{S_\infty}(\Rt/I,\cO).\]
\end{lemma}
\begin{proof}
As $\Psi$ is an isomorphism of $\Rt$-modules, we have $\Psi(\Rt[I]) = \Hom_{S}(\Rt,S)[I]$ and thus
\begin{align*}
\Psi(\Rt[I])
&= \{f:\Rt\to S|rf=0 \text{ for all }r\in I\}\\
&= \{f:\Rt\to S|(rf)(x)=0 \text{ for all }r\in I\text{ and } x\in \Rt\}\\
&= \{f:\Rt\to S|f(rx)=0 \text{ for all }r\in I\text{ and } x\in \Rt\}\\
&= \{f:\Rt\to S|f(I) = 0\}\\
&=\Hom_{S}(\Rt/I,S).
\end{align*}
The proof for $\Psi_\theta(\Rt_\theta[I_\theta])$ is identical.
\end{proof}

Now since $\Rt/I\cong R$ is a projective $S$-module, $\sigma$ induces a surjective map $\Hom_{S}(\Rt/I,S)\to \Hom_{S}(\Rt/I,\cO)$. By Lemma \ref{lem:Psi(R[I])}, this is a surjective map $\sigma|_{\Psi(\Rt[I])}:\Psi(\Rt[I])\to \Psi_\theta(\Rt_\theta[I_\theta])$, so the commutative diagram from Lemma \ref{lem:gorenstein} gives that $\pi_\theta|_{\Rt[I]}:\Rt[I]\to \Rt_\theta[I_\theta]$ is surjective. This completes the proof of (1).
\end{proof}

\begin{proof}[Proof of Theorem \ref{thm:c_1 local}(2)]
By the definition of $I$, we have a short exact sequence of $S$-modules
\[0\to I \to \Rt\xrightarrow{\varphi} R \to 0.\]
Applying $-\otimes_{S}\cO$ to this gives an exact sequence 
\[\Tor_1^{S}(R,\cO)\to I\otimes_{S}\cO\to \Rt_\theta\xrightarrow{\varphi} R_\theta\to 0.\]
and so as $I_\theta = \ker \varphi_\theta$, this gives as exact sequence
\[\Tor_1^{S}(R,\cO)\to I\otimes_{S}\cO\to I_\theta\to 0.\]
But now as $R$ is a finite free $S$-module, $\Tor_1^{S}(R,\cO) = 0$ and so we have an isomorphism $I\otimes_{S}\cO\cong I_\theta$ of $\Rt_\theta$-modules.

Now by \cite[\href{https://stacks.math.columbia.edu/tag/07ZA}{Lemma 07ZA}]{stacks-project} we indeed have:
\[\pi_\theta(\Fitt(I)) = \Fitt(I\otimes_{S}\cO) = \Fitt(I_\theta),\]
as desired. This completes the proof of (2), and hence of Theorem \ref{thm:c_1 local}.
\end{proof}

We note the following corollary.

\begin{corollary}\label{well defined}
With notation as above  \[ C_{1,\lambdat}(\Rt)= \lambdat(\Rt[I])/\lambdat(\Fitt(I))\]
 depends only on its quotient $\Rt/I \simeq   R$  and we define \[  C_{1,\lambda}(R) \stackrel{\mathrm{def}}=C_{1,\lambdat}(\Rt).\] 
\end{corollary}

\begin{proof}
 This follows from Theorems \ref{thm:c_1 local} which shows that  \[C_{1,\lambda_\theta}(R_\theta) =
C_{1,\lambdat}(\Rt),\]
 and the results of Appendix \ref{app} which show  that  $C_{1,\lambda_\theta}(R_\theta)$ is well defined and independent of $\Rt_\theta$.
\end{proof}

\begin{remark}
The above Corollary \ref{well defined} can also be proved directly by using  the proof of Lemma \ref{l:indc2} instead of reducing to the statement of Lemma \ref{l:indc2}.
\end{remark}

For later use, we also state the following result.

\begin{lemma}\label{lem:R[I]=omega}
As $R$-modules one has $\Rt[I]\cong \omega_{R}$.
\end{lemma}
\begin{proof}
As $R$ is Cohen--Macaulay and $\Rt$ is Gorenstein, we have that $\omega_{R}\cong \Hom_{S}(R,S)$ and $\Rt\cong \Hom_{S}(\Rt,S)$ as $\Rt$-modules. Now by \cite[\href{https://stacks.math.columbia.edu/tag/08YP}{Lemma 08YP}]{stacks-project}:
\begin{align*}
\Rt[I] 
&\cong \Hom_{\Rt}(R,\Rt)
 \cong \Hom_{\Rt}(R,\Hom_{S}(\Rt,S))
 \cong \Hom_{S}(R,S)
 \cong \omega_{R}
\end{align*}
as $R$-modules.
\end{proof}

\subsection{Invariance of $\Der^1_\cO(R_\theta,E/\cO)$}\label{sec:Der1}

In this section, we will let $R\in C_\cO$ and $S = \cO[[y_1,\ldots,y_d]]$ be as above. We shall show that for any inclusion $\theta:S\into R$ satisfying \ref{prop P}, the Andr\'e--Quillen cohomology group $\Der^1_\cO(R_\theta,E/\cO)$ does not depend on the choice of $\theta$.

In order to do this, we will need to make use of a continuous version of Andr\'e--Quillen cohomology, as the classical version does not behave well for rings that are not of finite type but only topologically of finite type over the base. We will define this in terms of the analytic cotangent complex defined in \cite[Chapter 7]{GR}.

For any ring $A$, we will let $\Mod_A$ denote the category of $A$-modules, $D(\Mod_A)$ its derived category, and $D^-(\Mod_A)\subseteq D(\Mod_A)$ the subcategory of bounded above complexes.

For any map of ring $A\to B$, let $L_{B/A}\in D^-(\Mod_B)$ denote the relative cotangent complex.

Now consider any $A\in\CNLO$ and let $\wedge:\Mod_A\to \Mod_A$ denote the $\ffrm_A$-adic completion functor. As in \cite[Chapter 7.1]{GR}, let $\wedge:D^-(\Mod_A)\to D^-(\Mod_A)$ denote its left-derived functor.

If $A\to B$ is a \emph{continuous} map of rings in $\CNLO$, then define the analytic relative cotangent complex to be $L^{\an}_{B/A} = (L_{B/A})^{\wedge}$. For any $B$-module $M$ and any $i\ge 0$ we may then define the $i^{th}$ continuous Andr\'e--Quillen cohomology group to be
\[\cDer^i_A(B,M) = H^i(\RHom_B(L_{B/A}^{\an},M)).\]
Similarly if $A\to B$ is any ring map and $M$ is any $B$-module, the $i^{th}$ Ardr\'e--Quillen cohomology group is just
\[\Der^i_A(B,M) = H^i(\RHom_B(L_{B/A},M)).\]
We will begin by recording the basic properties of continuous Andr\'e--Quillen cohomology we will need in our arguments.

\begin{proposition}\label{prop:les}
Given any $A,B,C\in\CNLO$, and continuous ring homomorphisms $A\to B\to C$ and any $C$-module $M$, there is a long exact sequence:
\[0\to \cDer^0_B(C,M)\to \cDer^0_A(C,M)\to \cDer^0_A(B,M)\to \cDer^1_B(C,M)\to \cDer^1_A(C,M)\to \cDer^1_A(B,M)\to\cdots\]
\end{proposition}
\begin{proof}
This follows from the distinguished triangle
\[
C\Lotimes_B L_{B/A}^{\an}\to L_{C/A}^{\an}\to L_{C/B}^{\an}\to C\Lotimes_B L_{B/A}^{\an}[1]
\]
from \cite[Theorem 7.1.33]{GR}.
\end{proof}

\begin{proposition}\label{prop:cDer=Der}
If $A\to B$ is a continuous map of rings in $\CNLO$ which makes $B$ into a finite $A$-module then $L_{B/A}^{\an}\cong L_{B/A}$, and so $\cDer^i_A(B,M)\cong \Der^i_A(B,M)$ for all $i\ge 0$ and all $M\in \Mod_B$.
\end{proposition}
\begin{proof}
As the map $A\to B$ is finite, it is finite type (and not merely \emph{topologically} finite type). By \cite[6.11]{Iyengar}, $L_{B/A}$ is quasi-isomorphic to a bounded above complex of finite free $B$-modules $\cL^\bullet$. Using $\cL^\bullet$ to compute $(L_{B/A})^\wedge$ we get
\[
L_{B/A}^{\an} = (L_{B/A})^\wedge \cong (\cL^\bullet)^\wedge = \cL^{\bullet} \cong L_{B/A},
\]
as finitely generated $B$-modules are already $\ffrm_B$-adically complete. The last claim now follows from the definition of $\cDer^i_A(B,M)$ and $\Der^i_A(B,M)$.
\end{proof}

\begin{proposition}\label{prop:Der0}
If $A\to B$ is a continuous map of rings in $\CNLO$. Then the module $\cOmega_{B/A} = \varprojlim \Omega_{(B/\ffrm_B^n)/A}$ of continuous K\"ahler differentials defined in \cite[Section 7.1]{BKM} is the $\ffrm_B$-adic completion of $\Omega_{B/A}$ and we have $\cDer^0_A(B,M)\cong \Hom_A(\cOmega_{B/A},M)$ for any $B$-module $M$.
\end{proposition}
\begin{proof}
For the first claim, we argue as in \cite[Lemma 7.1]{BKM} (and note that the assumption that $\cR$ is finitely generated over $A$ in that lemma was used only in the last step, to conclude that $\Omega_{\cR/A}$ was finitely generated over $A$). Specifically, for any $n>k$ we have $\Omega_{B/A}/\ffrm_B^k\Omega_{B/A} = \Omega_{B/A}\otimes_B B/\ffrm_B^k\cong \Omega_{(B/\ffrm_B^n)/A}\otimes B/\ffrm_B^k$ and so taking inverse limits gives
\[
\Omega_{B/A}/\ffrm_B^k\Omega_{B/A}
\cong\invlim_n\left(\Omega_{(B/\ffrm_B^n)/A}\otimes_B B/\ffrm_B^k\right)
\cong\invlim_n\left(\Omega_{(B/\ffrm_B^n)/A}\right)\otimes_B B/\ffrm_B^k = \cOmega_{B/A}\otimes_B B/\ffrm_B^k.
\]
Taking inverse limits again and using the fact that $\cOmega_{B/A}$ is finite over $B$, and hence $\ffrm_B$-adically complete gives 
\[\cOmega_{B/A}\cong \invlim\cOmega_{B/A}\otimes_B B/\ffrm_B^k\cong \invlim_k\Omega_{B/A}/\ffrm_B^k\Omega_{B/A}\]
as desired.

In particular, this shows that the module $\cOmega_{B/A}$ is simply the module $\Omega_{B/A}^{\an} = (\Omega_{B/A})^{\wedge}$ from \cite{GR}, and so the second claim follows from \cite[Lemma 7.1.27(iii)]{GR} and the definition of $\cDer^i_A(B,M)$.
\end{proof}

We will also need the following specific computations of continuous Andr\'e--Quillen cohomology:

\begin{lemma}\label{lem:Der power series}
For any $n\ge 0$, and any $\cO[[x_1,\ldots,x_n]]$-module $M$, we have
\[
\cDer^i_\cO(\cO[[x_1,\ldots,x_n]],M) = \begin{cases}
M^n & i=0\\
0 & i\ge 1.
\end{cases}
\]
\end{lemma}
\begin{proof}
By \cite[Proposition 7.1.29]{GR} we have $L^{\an}_{\cO[[x_1,\ldots,x_n]]/\cO} = \cOmega_{\cO[[x_1,\ldots,x_n]]/\cO}[0] = \cO[[x_1,\ldots,x_n]]^n[0]$ and so
\[\RHom_{\cO[[x_1,\ldots,x_n]]}(L^{\an}_{\cO[[x_1,\ldots,x_n]]/\cO},M) =\RHom_{\cO[[x_1,\ldots,x_n]]}(\cO[[x_1,\ldots,x_n]]^n[0],M) = M^n[0]\]
so the claim follows.
\end{proof}

\begin{lemma}\label{lem:Der quotient}
If $A$ is a ring and $B= A/I$ for some ideal $I\subseteq A$, then for any $B$-module $M$, $\cDer^0_A(B,M) = 0$ and $\cDer^1_A(B,M) = \Hom_B(I/I^2,M)$.
\end{lemma}
\begin{proof}
As $B=A/I$ is clearly finite over $A$, Proposition \ref{prop:cDer=Der} gives $\cDer^i_A(B,M) = \Der^i(B,M)$ for all $i\ge 0$ and all $M$. The claim now follows from \cite[6.12]{Iyengar}.
\end{proof}
 
For the remainder of this section, we always treat $E/\cO$ as an $R_\theta$-module (and hence as an $R$-module) via $\lambda_\theta:R_\theta\to \cO$. Our main result is the following:

\begin{theorem}\label{thm:Der^1 local}
We have $\Der^1_\cO(R_\theta,E/\cO)\cong \cDer^1_\cO(R,E/\cO)$.
\end{theorem}

This implies that $\Der^1_\cO(R_\theta,E/\cO)$ depends only on $R$ and $\lambda:R\to E/\cO$, and not on $\theta$, and so will complete the proof of Theorem \ref{thm:delta invariant}.

We first observe the following:

\begin{lemma}
For any $i\ge 0$ and any $R_\theta$-module $M$, we have 
\[\cDer^i_{S}(R,M) \cong\Der^i_{S}(R,M) \cong \Der^i_\cO(R_\theta,M)\cong \cDer^i_\cO(R_\theta,M).\]
\end{lemma}
\begin{proof}
The first and last isomorphisms follow from Proposition \ref{prop:cDer=Der}, as $R$ is finite over $S$ and $R_\theta$ is finite over $\cO$.

For the second isomorphism, first note that as $R$ is a finite free $S$-module, it is a projective resolution for itself in $D(S)$, and so we have $R\Lotimes_{S}\cO = R\otimes_{S}\cO \cong R_\theta$. By \cite[\href{https://stacks.math.columbia.edu/tag/08QQ}{Lemma 08QQ}]{stacks-project} this implies that $L_{R/S}\Lotimes_{R}R_\theta\cong L_{R_\theta/\cO}$. But now \cite[\href{https://stacks.math.columbia.edu/tag/0E1W}{Lemma 0E1W}]{stacks-project} gives that
\[\RHom_{R}(L_{R/S},M) = \RHom_{R_\theta}(L_{R/S}\Lotimes_{R}R_\theta,M)\cong\RHom_{R_\theta}(L_{R_\theta/\cO},M)\]
so the claim follows by definition.
\end{proof}

So to prove Theorem \ref{thm:Der^1 local} it will suffice to prove the following:

\begin{proposition}\label{prop:Der^1_S = Der^1_O}
$\cDer^1_{S}(R,E/\cO)\cong \cDer^1_\cO(R,E/\cO)$.
\end{proposition}
\begin{proof}
Applying Proposition \ref{prop:les} to the ring maps $\cO\to S\to R$ gives an exact sequence:
\begin{align*}
0 &\to \cDer^0_{S}(R,E/\cO)\to \cDer^0_\cO(R,E/\cO)\to \cDer^0_\cO(S,E/\cO)\\
&\to \cDer^1_{S}(R,E/\cO)\to \cDer^1_\cO(R,E/\cO)\to \cDer^1_\cO(S,E/\cO)
\end{align*}
By Lemma \ref{lem:Der power series}, $\cDer^0_\cO(S,E/\cO) = (E/\cO)^d$ and $\cDer^1_\cO(S,E/\cO) = 0$.

But now by the assumption that $\lambda:R \to \cO$ represents a smooth point of $\Spec R[1/\varpi]$ we get that $\cOmega_{R/\cO}\otimes_\lambda\cO$ has rank $d$ as an $\cO$-module (as in \cite[Theorem 7.16]{BKM}), and so
\[\cDer^0_\cO(R,E/\cO) = \Hom_{R}(\cOmega_{R/\cO},E/\cO) = \Hom_{\cO}(\cOmega_{R/\cO}\otimes_\lambda\cO,E/\cO) = (E/\cO)^d\oplus G,\]
for some finite group $G$. Also as $\Phi_{\lambda_\theta}(R_\theta)=\cOmega_{R^\theta/\cO}\otimes_{\lambda_\theta}\cO$ is finite (as $\theta$ satisfies \ref{prop P}), \[\cDer^0_{S}(R,E/\cO) \cong \cDer^0_\cO(R_\theta,E/\cO) = \Hom_{\cO}(\cOmega_{R_\theta/\cO},E/\cO) = \Hom_{\cO}(\Phi_{\lambda_\theta}(R_\theta),E/\cO)\]
is finite as well. Now the exact sequence simplifies to
\[0\to \cDer^0_{S}(R,E/\cO) \to (E/\cO)^d\oplus G\to (E/\cO)^d\to  \cDer^1_{S}(R,E/\cO)\to \cDer^1_\cO(R,E/\cO)\to 0.\]
But comparing ranks in the sequence $0\to \cDer^0_{S}(R,E/\cO) \to (E/\cO)^d\oplus G\to (E/\cO)^d$ implies that $(E/\cO)^d\oplus G\to (E/\cO)^d$ has finite cokernel, and hence must be surjective, as $E/\cO$ does not have any nontrivial finite quotients. This implies that the map $\cDer^1_{S}(R, E/\cO)\to \cDer^1_\cO(R,E/\cO)$ is indeed an isomorphism. This completes the proof of Theorem \ref{thm:Der^1 local}.
\end{proof}

We note that in Theorem \ref{thm:Der^1 local} and Corollary \ref{well defined}, we have proved that
\begin{equation}
\cDer_\cO^1(R,E/\cO)\cong \cDer_\cO^1(R_\theta,E/\cO).
\end{equation}
\begin{equation}
C_{1,\lambda}(R) = C_{1,\lambda_\theta}(R_\theta).
\end{equation}

In order to actually compute $\delta_{\lambda_\theta}(R_\theta)$, we will need a method for computing $\cDer^1_\cO(R,E/\cO)$. For this, take any triple $(\Rt,I,\varphi)$ satisfying \ref{prop CI}. Then we now have the following generalization of \eqref{AQ}:

\begin{theorem}\label{thm:Der exact sequence}
There is a $4$-term exact sequence:
\[0\to \Hom_{R}(\cOmega_{R/\cO},E/\cO)\to \Hom_{\Rt}(\cOmega_{\Rt/\cO},E/\cO) \to \Hom_{R}(I/I^2,E/\cO) \to \cDer^1_\cO(R,E/\cO)\to 0\]
\end{theorem}
\begin{proof}
Applying Proposition \ref{prop:les} to the ring maps $\cO\to \Rt \to R$ gives an exact sequence:
\begin{align*}
0 &\to \cDer^0_{\Rt}(R,E/\cO)\to \cDer^0_\cO(R,E/\cO)\to \cDer^0_\cO(\Rt,E/\cO)\\
  &\to \cDer^1_{\Rt}(R,E/\cO)\to \cDer^1_\cO(R,E/\cO)\to \cDer^1_{\cO}(\Rt,E/\cO)
\end{align*}
and Lemma \ref{lem:Der quotient} implies that $\cDer^0_{\Rt}(R,E/\cO) = 0$ and $\cDer^1_{\Rt}(R,E/\cO) = \Hom_{R}(I/I^2,E/\cO)$, so it's enough to prove that $ \cDer^1_\cO(\Rt,E/\cO) = 0$ (since by Proposition \ref{prop:Der0}, $\cDer^0_\cO(R,E/\cO)=\Hom_{R}(\cOmega_{R},E/\cO)$ and $\cDer^0_\cO(\Rt,E/\cO)=\Hom_{\Rt}(\cOmega_{\Rt},E/\cO)$).

Since $\Rt$ is a complete intersection, we can write $\Rt = P/J$ where $P = \cO[[x_1,\ldots,x_{d+n}]]$ and $J=(f_1,\ldots,f_n)$ is generated by a regular sequence. Applying Proposition \ref{prop:les} to the ring maps $\cO\to P\to \Rt$ gives an exact sequence:
\begin{align*}
0 &\to \cDer^0_{P}(\Rt,E/\cO)\to \cDer^0_\cO(\Rt,E/\cO)\to \cDer^0_\cO(P,E/\cO)\\
  &\to \cDer^1_{P}(\Rt,E/\cO)\to \cDer^1_\cO(\Rt,E/\cO)\to \cDer^1_\cO(P,E/\cO).
\end{align*}
Now Lemma \ref{lem:Der power series} gives $\cDer^0_\cO(P,E/\cO) = (E/\cO)^{d+n}$ and $\cDer^1_\cO(P,E/\cO) = 0$ and Lemma \ref{lem:Der quotient} gives $\cDer^0_{P}(\Rt,E/\cO) = 0$ and $\cDer^1_{P}(\Rt,E/\cO) = \Hom_{\Rt}(J/J^2,E/\cO)$. Moreover as $J$ is generated by a regular sequence of length $n$, it follows that $J/J^2\cong (\Rt)^n$ as $\Rt$-modules, and so $\cDer^1_{P}(\Rt,E/\cO) = \Hom_{\Rt}(J/J^2,E/\cO)\cong (E/\cO)^n$. Thus the above exact sequence simplifies to
\[
0 \to \cDer^0_\cO(\Rt,E/\cO)\to (E/\cO)^{n+d} \to (E/\cO)^n\to \cDer^1_\cO(\Rt,E/\cO)\to 0.
\]
But now, just as in the proof of Proposition \ref{prop:Der^1_S = Der^1_O} above, the fact that $\Spec \Rt[1/\varpi]$ is smooth of dimension $d$ at $\lambdat$ implies that $\cDer^0_\cO(\Rt,E/\cO) \cong (E/\cO)^d\oplus H$ for some finite group $H$, and so comparing ranks gives that $(E/\cO)^{n+d} \to (E/\cO)^n$ has finite cokernel, and hence is surjective. Thus $\cDer^1_\cO(\Rt,E/\cO)=0$, and so the claim follows.
\end{proof}

\subsection{Wiles defect for augmented rings $(R,\lambda) \in C_\cO$}

We make the  following definitions and in particular define the {\it Wiles defect} for tuples
$(R,\lambda) \in C_\cO$.  

\begin{definition}\label{def:key}
Let $R$ be a complete, Noetherian  local $\cO$-algebra which is Cohen--Macaulay and flat over $\cO$ of relative dimension $d$, and  with an augmentation $\lambda:R \to\cO$ such that $\Spec R[1/\varpi]$ is formally smooth at the point corresponding to $\lambda$. 

\begin{itemize}
\item Define
\[D_{1,\lambda}(R) = \frac{\log |\cDer^1_\cO(R,E/\cO)|}{\log|\cO/p|}\] (see Theorem \ref{thm:Der^1 local}).
\item Define  
 \[c_{1,\lambda}(R) = \frac{\log\left| C_{1,\lambda}(R)\right|}{\log|\cO/p|} =\frac{\log\left|\lambdat(\Rt[I])/\lambdat(\Fitt(I))\right|}{\log|\cO/p|},\]
for any triple $(\Rt,I,\varphi)$ satisfying \ref{prop CI}.
\item The \emph{Wiles defect} $ \delta_\lambda(R)$ of $R$ at $\lambda$  is defined to be 
\[\delta_\lambda(R) = D_{1,\lambda}(R)-c_{1,\lambda}(R).\]
\end{itemize}
\end{definition}

\begin{lemma}
The  numbers   \[D_{1,\lambda}(R),  c_{1,\lambda}(R),  \delta_\lambda(R)\] are well defined.
\end{lemma}

\begin{proof}
This is  a consequence of  Theorem \ref{thm:delta invariant},  Theorem \ref{thm:c_1 local},  Corollary \ref{well defined}, and Theorem \ref{thm:Der^1 local}, combined with Remark~\ref{rem:FinitenessOfDer} which confirms the finiteness of  length of the terms involved in the one-dimensional case.
\end{proof}

Here is the main theorem of this section which  uses all the work we have done here.

\begin{theorem}\label{thm:delta invariant}
Let $R$ and $\lambda:R \onto \cO$ be as above, and let $\theta:S\into R$ be a map satisfying \ref{prop P}. Then the invariants $C_{1,\lambda_\theta}(R_\theta)$, $\Der_\cO^1(R_\theta,E/\cO)$ and $\delta_{\lambda_\theta}(R_\theta)$ are independent of the choice of~$\theta$.
\end{theorem}
\begin{proof}
The proofs of the independence statements for $C_{1,\lambda_\theta}(R_\theta)$ and $\Der_\cO^1(R_\theta,E/\cO)$ follow from  Theorems \ref{thm:c_1 local} and \ref{thm:Der^1 local} respectively. The assertion for the Wiles defect  $\delta_{\lambda_\theta}(R_\theta)$ is then immediate from Theorem~\ref{thm:Venkatesh}. \end{proof}

We note the consistency of this definition with the definition of Wiles defect for tuples $(R,\lambda) \in C_\cO$ when $R$ is of dimension one.

\begin{proposition}\label{prop:coincidence}
In the case when $(R,\lambda) \in C_\cO$  and  $R$ is of dimension one,  then 
 \[\delta_\lambda(R) = D_{1,\lambda}(R)-c_{1,\lambda}(R)=\frac{\log|\Phi_\lambda(R)|-\log|\Psi_\lambda(R)|}{\log|\cO/p|} .\]

\end{proposition}

\begin{proof}
This follows from  Proposition \ref{prop:cDer=Der} and Proposition \ref{prop:v} of Appendix \ref{app} (cf. Theorem \ref{thm:Venkatesh}).
\end{proof}

\begin{proposition}\label{prop:Wiles generalization}
 For  $(R,\lambda) \in C_\cO$, $\delta_\lambda(R)=0$ if and only if  $R$ is a complete intersection. In particular, $\delta_\lambda(\cO[[x_1,\ldots,x_n]])=0$ for any $n\ge 1$ and any $\lambda:\cO[[x_1,\ldots,x_n]]\onto\cO$
\end{proposition}

\begin{proof}
 If $R$ is a complete intersection then  $\cDer^1_\cO(R,E/\cO)=0$ by the argument given in the proof of Theorem \ref{thm:Der exact sequence}. Further $C_{1,\lambda}(R)=0$  (as we can take the CI cover  $\widetilde R=R$). This gives that $\delta_\lambda(R)=0$.
 
 Conversely assume $\delta_\lambda(R)=0$. Then  by our results we have a quotient $(R_\theta,\lambda_\theta) \in C_\cO$ of $(R,\lambda) \in C_\cO$ by a regular sequence $(y_1,\ldots, y_d)$,  namely  $R_\theta=R/(y_1,\ldots, y_d)$ and $\lambda_\theta: R \to R_\theta  \to \cO$ (the last map being $\lambda$) with $R_\theta$  of dimension one. Further  $\delta_{\lambda_\theta}(R_\theta)=\delta_\lambda(R)=0$. Thus by the result of Wiles and Lenstra, $R_\theta$ is a complete intersection, which implies that $R$ is a complete intersection.
\end{proof}

\begin{remark}

For $(R,\lambda) \in C_\cO$ and $R$ of dimension one,  by the Wiles--Lenstra result note that the vanishing of $\Der^1_\cO(R,E/\cO)$ implies that $R$ is a complete intersection because of the inequality $|\Phi_\lambda(R)| \geq |\Psi_\lambda(R)|$ which follows from the usual Fitting ideals argument. From this we again deduce  by invariance of $\cDer^1_\cO(R,E/\cO)$ on going modulo regular sequences that  in general for $(R,\lambda) \in C_\cO$  the vanishing of $\cDer^1_\cO(R,E/\cO)$ implies that $R$ is a complete intersection.
\end{remark}

\subsection{Properties of the  Wiles defect}\label{sec:local defect}

Theorem \ref{thm:delta invariant} can be restated as:

\begin{theorem}\label{thm:delta local}
If $(y_1,\ldots,y_d,\varpi)$ is a regular sequence for $R$ with $y_1,\ldots,y_d\in\ker\lambda$, where we will also use $\lambda$ to denote the induced map $R/(y_1,\ldots,y_d)\onto\cO$, then $\delta_\lambda(R) = \delta_\lambda(R/(y_1,\ldots,y_d))$. In particular, $\delta_\lambda(R/(y_1,\ldots,y_d))$ is independent of the choice of regular sequence.
\end{theorem}

We now deduce some additivity  properties of $\delta_\lambda(R)$ that we use later.
\begin{proposition}\label{prop:delta prod}
Let $R_{1}$ and $R_{2}$ be complete, Noetherian, Cohen--Macaulay, reduced $\cO$-algebras, which are flat over $\cO$ of relative dimensions $d_1$ and $d_2$. Pick augmentations $\lambda_i:R_{i}\to \cO$ such that $\Spec R_{i}[1/\varpi]$ is formally smooth at the point corresponding to $\lambda_i$. Let $R= R_1 \cotimes_{\cO} R_{2}$ and $\lambda = \lambda_1\cotimes \lambda_1:R\to \cO$.

Then 
\begin{enumerate}
\item $D_{1,\lambda}(R) = D_{1,\lambda_1}(R_{1})+D_{1,\lambda_2}(R_{2})$ 
\item $c_{1,\lambda}(R) = c_{1,\lambda_1}(R_{1})+c_{1,\lambda_2}(R_{2})$
\item $\delta_{\lambda}(R) = \delta_{\lambda_1}(R_{1})+\delta_{\lambda_2}(R_{2})$.
\end{enumerate}
\end{proposition}
\begin{proof}
By definition, (3) will follow from (1) and (2).

For (1), we will first reduce to dimension $1$. Let $S_1 = \cO[[x_1,\ldots,x_{d_1}]]$ and $S_2 = \cO[[y_1,\ldots,y_{d_2}]]$. By Proposition \ref{prop ThetaExists}, we may find maps $\theta_1:S_1\into R_1$ and $\theta_2:S_2\into R_2$ satisfying \ref{prop P}. Then the map $\theta = \theta_1\cotimes_{\cO}\theta_2:S_1\cotimes_{\cO}S_2\into R$ satisfies \ref{prop P} as well. So consider the rings
\begin{align*}
R_{1,\theta_1} &= R_1\otimes_{S_1}\cO,&
R_{2,\theta_1} &= R_1\otimes_{S_1}\cO,&
&\text{and}&
R_\theta &= R\otimes_{S_1\cotimes_{\cO}S_2}\cO = R_{1,\theta_1}\otimes_\cO R_{2,\theta_2}
\end{align*}
and note that these are all finite free over $\cO$.

By Theorem \ref{thm:Der^1 local} we now have that
\begin{align*}
	\cDer^1_{\cO}(R_1,E/\cO) &= \Der^1_{\cO}(R_{1,\theta_1},E/\cO),\\
	\cDer^1_{\cO}(R_2,E/\cO) &= \Der^1_{\cO}(R_{2,\theta_2},E/\cO),\\
	\cDer^1_{\cO}(R,E/\cO) &= \Der^1_{\cO}(R_{\theta},E/\cO).
\end{align*}
But now by \cite[\href{https://stacks.math.columbia.edu/tag/09DA}{Lemma 09DA}]{stacks-project} as $R_1$ and $R_2$ are both free over $\cO$, and hence Tor-independent we have
\begin{align*}
L_{R_\theta/\cO} &\cong L_{R_{1,\theta_1}\otimes_\cO R_{2,\theta_2}/\cO} \cong L_{R_{1,\theta_1}/\cO}\Lotimes_{R_{1,\theta_1}}R_\theta\oplus  L_{R_{2,\theta_2}/\cO}\Lotimes_{R_{2,\theta_1}}R_\theta.
\end{align*}
Thus
\begin{align*}
\cDer^1_{\cO}(R,E/\cO) 
&\cong \Der^1_{\cO}(R_{\theta},E/\cO) 
 = H^1(\RHom_{R_\theta}(L_{R/\cO},E/\cO))\\
&\cong H^1(\RHom_{R_\theta}(L_{R_{1,\theta_1}/\cO}\Lotimes_{R_{1,\theta_1}}R_\theta\oplus  L_{R_{2,\theta_2}/\cO}\Lotimes_{R_{2,\theta_1}}R_\theta,E/\cO))\\
& \cong H^1(\RHom_{R_\theta}(L_{R_{1,\theta_1}/\cO}\Lotimes_{R_{1,\theta_1}}R_\theta,E/\cO))\oplus 
   H^1(\RHom_{R_\theta}(L_{R_{2,\theta_2}/\cO}\Lotimes_{R_{2,\theta_1}}R_\theta,E/\cO))\\
&\cong H^1(\RHom_{R_{1,\theta_1}}(L_{R_{1,\theta_1}/\cO},E/\cO))\oplus 
   H^1(\RHom_{R_{2,\theta_2}}(L_{R_{2,\theta_2}/\cO},E/\cO))\\
&= \Der^1_{\cO}(R_{1,\theta_1},E/\cO)\oplus \Der^1_{\cO}(R_{2,\theta_2},E/\cO)
 = \cDer^1_{\cO}(R_1,E/\cO)\oplus \cDer^1_{\cO}(R_2,E/\cO)
\end{align*}
and so (1) follows.

It remains to prove (2). Consider triples $(\Rt_{1},I_{1},\varphi_{1})$ and $(\Rt_{2},I_{2},\varphi_{2})$ satisfying \ref{prop CI} (with $(R_1,\lambda_1)$ and $(R_{2},\lambda_2)$, respectively, in place of $(R,\lambda)$).

Define $\Rt= \Rt_{1}\cotimes_{\cO}\Rt_{2}$, and note that $I_{1}\cotimes_{\cO} \Rt_{2}$ and $\Rt_{1}\cotimes_{\cO}I_{2}$ are both ideals of $\Rt$. Let $\varphi=\varphi_{1}\otimes\varphi_{2}:\Rt= \Rt_{1}\cotimes_{\cO}\Rt_{2}\onto R_{1}\cotimes_{\cO} R_{2} = R$, and note that $\ker\varphi= \left(I_{1}\cotimes_{\cO} \Rt_{2}\right)+\left(\Rt_{1}\cotimes_{\cO}I_{2}\right)$. Denoting this ideal $I \subseteq \Rt$, the triple $(\Rt,I,\varphi)$ satisfies \ref{prop CI}. So by the definition of $c_{1,\lambda}$,
\begin{align*}
c_{1,\lambda_1}(R_{1})\log|\cO/p| &= \log\left|\lambda_1(\Rt_{1}[I_{1}])/\lambda_1(\Fitt(I_{1}))\right|\\
c_{1,\lambda_2}(R_{1})\log|\cO/p| &= \log\left|\lambda_2(\Rt_{2}[I_{2}])/\lambda_2(\Fitt(I_{2}))\right|\\
c_{1,\lambda}(R)\log|\cO/p| &= \log\left|\lambda(\Rt[I_{}])/\lambda(\Fitt(I))\right|
\end{align*}
Hence it will suffice to show that
\begin{align*}
\lambda\left(\Rt[I]\right) &= \lambda_1\left(\Rt_{1}[I_{1}]\right)\lambda_2\left(\Rt_{2}[I_{2}]\right),\text { and }\\
\lambda\left(\Fitt(I)\right) &= \lambda_1\left(\Fitt(I_{1})\right)\lambda_2\left(\Fitt(I_{2})\right)
\end{align*}
as ideals of $\cO$. For the first claim, standard properties of annihilators imply that
\begin{align*}
\Rt[I] 
&
 = \Rt\left[\left(I_{1}\cotimes_{\cO} \Rt_{2}\right)+\left(\Rt_{1}\cotimes_{\cO}I_{2}\right)\right]
 = \Rt_{}\left[\left(I_{\infty,1}\cotimes_{\cO} \Rt_{2}\right)\right]\cap \Rt_{}\left[\left(\Rt_{1}\cotimes_{\cO}I_{2}\right)\right]\\
&= \left(\Rt_{1}\left[I_{1}\right]\cotimes_{\cO}\Rt_{2}\right)\cap \left(\Rt_{1}\cotimes_{\cO}\Rt_{2}\left[I_{2}\right]\right) = \Rt_{1}\left[I_{1}\right]\cotimes_{\cO}\Rt_{2}\left[I_{2}\right]
\end{align*}
(where we've used that fact that $\left(A\cotimes_{\cO}\Rt_{2}\right)\cap \left(\Rt_{1}\cotimes_{\cO}B\right) = \left(A\cotimes_{\cO}\Rt_{2}\right) \left(\Rt_{\infty,1}\cotimes_{\cO}B\right) = A\cotimes_{\cO}B$ for any ideals $A\subseteq \Rt_{1}$ and $B\subseteq \Rt_{2}$). Thus
\[\lambda\left(\Rt_{}[I_{}]\right) = (\lambda_1\otimes \lambda_2)\left(\Rt_{1}\left[I_{1}\right]\cotimes_{\cO}\Rt_{2}\left[I_{2}\right]\right) = \lambda_1\left(\Rt_{1}\left[I_{1}\right]\right) 
\lambda_2 \left(\Rt_{2}\left[I_{2}\right]\right).
\]
For the statement about fitting ideals, fix presentations
\begin{align*}
0&\to K_1\to \Rt_{1}^{m} \xrightarrow{A} I_{1}\to 0\\
0&\to K_2\to \Rt_{2}^{n} \xrightarrow{B} I_{2}\to 0
\end{align*}
where $K_i$ is a finitely generated $\Rt_{\infty,i}$-module. Then $A$ and $B$ induce surjective maps $A\otimes\Id:\Rt_{}^{m} = \Rt_{1}^{m}\cotimes_{\cO}\Rt_{2}\to I_{1}\cotimes_{\cO}\Rt_{2}$ and $\Id\otimes B:\Rt_{}^{n} = \Rt_{1}\cotimes_{\cO}\Rt_{2}^n\to \Rt_{1}\cotimes_{\cO}I_{2}$, and so we may combine them to produce a surjective map
\[C = (A\otimes\Id)-(\Id\otimes B) : \Rt_{}^{m+n}= \Rt_{}^m\oplus \Rt_{}^n\to \left(I_{1}\cotimes_{\cO} \Rt_{2}\right)+\left(\Rt_{1}\cotimes_{\cO}I_{2}\right) = I.\]
Write $K \subseteq \Rt_{}^{m+n}$ for the kernel of $C$.

By definition: $\Fitt(I_{1})$ is the ideal of $\Rt_{1}$ generated by all elements of the form $\det\left(u_1,\ldots,u_m\right)\in \Rt_{\infty,1}$ for $u_1,\ldots,u_m\in K_1\subseteq \Rt_{1}^m$; $\Fitt(I_{2})$ is the ideal of $\Rt_{2}$ generated by all elements of the form $\det\left(v_1,\ldots,v_n\right)\in \Rt_{1}$ for $v_1,\ldots,v_n\in K_2\subseteq \Rt_{2}^n$; and $\Fitt(I_{\infty})$ is the ideal of $\Rt_{}$ generated by all elements of the form $\det\left(w_1,\ldots,w_{m+n}\right)\in \Rt_{}$ for $w_1,\ldots,w_{m+n}\in K\subseteq \Rt^{m+n}$.

Now given any $u_1,\ldots,u_m\in K_1$ and $v_1,\ldots,v_n\in K_2$ it's easy to see that $\ds\binom{u_i\otimes 1}{0},\binom{0}{1\otimes v_j}\in K$ for all $i$ and $j$, and so $\Fitt(I_{})$ contains the element
\[
\det\begin{pmatrix}
u_1\otimes 1&\cdots & u_m\otimes 1&0&\cdots&0\\
0&\cdots&0&1\otimes v_1&\otimes&1\otimes v_n
\end{pmatrix}
=\det(u_1,\ldots,u_m)\otimes\det(v_1,\ldots,v_m).
\]
If follows that $\Fitt(I_{1})\cotimes_{\cO}\Fitt(I_{2})\subseteq \Fitt(I_{})$ and so 
$\lambda_1\left(\Fitt(I_{1})\right)\lambda_2\left(\Fitt(I_{2})\right)\subseteq \lambda\left(\Fitt(I_{})\right)$.

For the reverse inclusion, we will use the following simple lemma:
\begin{lemma}
For any $\ds w = \binom{w_1}{w_2}\in K$, for $w_1 \in \Rt^m$ and $w_2\in \Rt^n$, there exist $u\in K_1$ and $v\in K_2$ for which $\lambda(w_1) = \lambda_1(u)$ and $\lambda(w_2) = \lambda_2(v)$.
\end{lemma}
\begin{proof}
As $w\in K$, we have $(A\otimes \Id)(w_1)-(\Id\otimes B)(w_2) = C(w) = 0$ so let $r = (A\otimes \Id)(w_1) = (\Id\otimes B)(w_2) \in \Rt$. By the definitions of $A$ and $B$ we have $r = (A\otimes \Id)(w_1) \in I_{1}\cotimes_{\cO}\Rt_{2}$ and $r = (\Id\otimes B)(w_2) \in \Rt_{\infty,1}\cotimes_{\cO}I_{2}$ and so 
\[r \in \left(I_{1}\cotimes_{\cO}\Rt_{2}\right)\cap \left(\Rt_{1}\cotimes_{\cO}I_{2}\right) = I_{1}\cotimes_{\cO}I_{2}\]
Now as $\lambda_1(I_{1}) = \lambda_2(I_{2})=0$ by assumption, we get that $(\lambda_1\otimes\Id)(r) = (\Id\otimes\lambda_2)(r) = 0$. Now let $u = (\Id\otimes\lambda_2)(w_1)\in \Rt_{1,\infty}^m$ and $v = (\lambda_1\otimes\Id)(w_2)\in \Rt_{2}^n$, so that
\begin{align*}
\lambda_1(u) &= \lambda_1((\Id\otimes\lambda_2)(w_1)) = (\lambda_1\otimes\lambda_2)(w_1) = \lambda(w_1)\\
\lambda_2(v) &= \lambda_2((\lambda_1\otimes\Id)(w_2)) = (\lambda_1\otimes\lambda_2)(w_2) = \lambda(w_2)
\end{align*}
and
\begin{align*}
A(u) &= (A\otimes\Id)(\Id\otimes\lambda_2)(w_1) = (\Id\otimes\lambda_2)(A\otimes\Id)(w_1) =  (\Id\otimes\lambda_2)(r) = 0\\
B(v) &= (\Id\otimes B)(\lambda_1\otimes\Id)(w_2) = (\lambda_1\otimes\Id)(\Id\otimes B)(w_2) = (\lambda_1\otimes\Id)(r) = 0.
\end{align*}
So now $w_1\in\ker A = K_1$ and $w_2 \in \ker B = K_2$, as desired.
\end{proof}
So now take any $w_1,\ldots,w_{m+n}\in K$. The lemma allows us to write $\ds \lambda(w_i) = \binom{\lambda_1(u_i)}{\lambda_2(v_1)}$ for $u_i\in K_1$ and $v_i\in K_2$, which gives
\[
\lambda(\det(w_1,\ldots,w_{m+n})) = 
\det
\begin{pmatrix}
\lambda_1(u_1)&\cdots&\lambda_1(u_{m+n})\\
\lambda_2(v_1)&\cdots&\lambda_2(v_{m+n})
\end{pmatrix}
\]
But now by standard properties of determinants, the determinant of this $(m+n)\times(m+n)$ matrix may be written as an alternating sum in the form
\[
\sum_{X,Y}(\pm1) \det\big((\lambda_1(u_i))_{i\in X}\big)\det\big((\lambda_2(v_j))_{j\in Y}\big)
= \sum_{X,Y}(\pm1) \lambda_1\left(\det\big((u_i)_{i\in X}\big)\right)\lambda_2\left(\det\big((v_j)_{j\in Y}\big)\right)
\]
(where the sum is taken over partitions $X\sqcup Y = \{1,\ldots,m+n\}$ with $|X|=m$ and $|Y|=n$). As this sum is in $\lambda_1\left(\Fitt(I_{1})\right)\lambda_2\left(\Fitt(I_{2})\right)$, it follows that $\lambda\left(\Fitt(I_{})\right)\subseteq\lambda_1\left(\Fitt(I_{1})\right)\lambda_2\left(\Fitt(I_{2})\right)$, giving the desired equality $\lambda\left(\Fitt(I_{})\right)=\lambda_1\left(\Fitt(I_{1})\right)\lambda_2\left(\Fitt(I_{2})\right)$, and completing the proof.
\end{proof}

\section{Galois deformation theory}\label{sec_deformation_theory}

This section recalls basic results on Galois deformation theory and fixes some notation for the remainder of this work. Our main references are \cite[\S~5]{Thorne16} and \cite[\S~4]{BKM}. 

Recall the notation from the end of Section~\ref{sec:intro}. We fix a continuous, absolutely irreducible residual representation 
$$\rhobar : G_{F} \to \GL_2(k)$$ 
with $\det \rhobar=\epsilon_p$, for simplicity. We will assume that $k$ contains the eigenvalues of all elements in the image of $\overline{\rho}$. We also fix a finite set $\Sigma$ of finite places $v$ of $F$ disjoint from $\Sigma_p$ that contains all places $v\notin \Sigma_p$ at which $\rhobar$ is ramified, and possibly further places of~$F$. 

\subsubsection*{Local deformation rings}

Let $v \in \Sigma$. We write $\cD_v^\square : \CNL_{\cO} \to \mathrm{Sets}$ for the functor that associates to $R \in \mathrm{CNL}_{\cO}$ the set of all continuous homomorphisms $r : G_{F_v} \to \GL_2(R)$ such that $r \pmod {\ffrm_R} = \overline{\rho}|_{G_{F_v}}$ and $\det r=\varepsilon_p$. The functor $\cD_v^\square$ is representable by an object $R_v^{\square} \in \CNL_{\cO}$. We will write $\rho_v^\square : G_{F_v} \to \GL_2(R_v^\square)$ for the universal lifting. 

A local deformation problem for $\overline{\rho}|_{G_{F_v}}$ is a subfunctor $\cD_v \subset \cD_v^\square$ satisfying the following conditions:
\begin{enumerate}
\item The functor $\cD_v$ is represented by a quotient $R_v$ of $R_v^\square$.
\item For all $R \in \CNL_{\cO}$, $g \in \ker(\GL_2(R) \to \GL_2(k))$ and $r \in \cD_v(R)$, we have $g r g^{-1} \in \cD_v(R)$.
\end{enumerate}

If a quotient $R_v$ of $R_v^\square$ corresponding to a local deformation problem $\cD_v$ has been fixed, we will write $\rho_v : G_{F_v} \to \GL_2(R_v)$ for the universal lifting of type $\cD_v$. A sufficient condition for a quotient $R_v$ of $R_v^\square$ to be a deformation ring is the following; see \cite[Lemma~5.12]{Thorne}.

\begin{lemma}\label{lem_local}
Let $\pi\colon R_v^\square\to R_v$ be a surjective morphism in $\mathrm{CNL}_{\cO}$ with specialization $r : G_{F_v} \to \GL_2(R_v)$ induced from the universal lifting, and assume the following conditions:
\begin{enumerate}
\item The ring $R_v$ is reduced, and not isomorphic to $k$.
\item 
For all $g \in \ker(\GL_2(R_v) \to \GL_2(k))$, the homomorphism $R_v^\square \to R_v$ associated to the representation $g r g^{-1}$ by universality factors through~$\pi$.
\end{enumerate}
Then the subfunctor of $\cD_ v^\square$ defined by $R_v$ is a local deformation problem.
\end{lemma}

Below, we consider quotients of $R_v^\square$ which are defined as in \cite{Kisin} as reduced, flat over $\cO$ quotients of $R_v^\square$, that are  characterized by the $\overline \Q_p$-valued points of their generic fiber; hence these $R_v$ satisfy Lemma \ref{lem_local} and thus  give rise to a local deformation problem. \cite{Kisin} computes  the  dimension of generic fibers of the quotients we consider, and  proves that they are regular.  

\subsubsection*{Modified local deformation rings}

We shall also need modified deformation problems as introduced in \cite{Calegari}. For this, we fix an eigenvalue $\alpha_v$ of $\rhobar(\Frob_v)$. Note that we have $\alpha_v\in k$ by our hypothesis $\rhobar(G_F)\subset\GL_2(k)$.

\begin{definition}\label{modif-def_gen}
The functor  $\tcD\!{}_v^\square : \CNL_{\cO} \to \mathrm{Sets}$ of modified framed deformations associates to $R \in \mathrm{CNL}_{\cO}$ a pair $(r,a)$ with $r\in\cD_v^\square(R)$ and $a\in R$ a root of the characteristic polynomial of $r$ such that $a\equiv\alpha_v\mod {\ffrm_R}$.
\end{definition}
There is an obvious natural transformation $u_v\colon
 \tcD\!{}_v^\square \Rightarrow \cD\!{}_v^\square$, and $\tcD\!{}_v^\square $ is representable by the localization $\tR{}_v^\square$ of the ring $R_v^\square[x]/(x^2-x\tr \rho_v^\square(\Frob_v)+\det\rho_v^\square(\Frob_v))$ at the maximal ideal generated by $\ffrm_{R_v^\square}$ and $(x-\alpha_v)$. If $\rhobar(\Frob_v)$ has a multiple eigenvalue, the ring $R_v^\square[x]/(x^2-x\tr \rho_v^\square(\Frob_v)+\det\rho_v^\square(\Frob_v))$ is local and hence isomorphic to $\tR{}_v^\square$. This proves the following result; see \cite[Lemma~2.1]{Calegari}.
\begin{lemma}\label{lem:CalOnModifCover}
If $\rhobar(\Frob_v)$ has distinct eigenvalues, the canonical map $R{}_v^\square\to \tR{}_v^\square$ is an isomorphism. Otherwise, the extension $R{}_v^\square \to \tR{}_v^\square$ is a finite flat extension of degree two.
\end{lemma}
The following definition is extracted from \cite[\S~2]{Calegari}.
\begin{definition}
A modified local deformation problem for $\overline{\rho}|_{G_{F_v}}$ is a subfunctor $\tcD_v \subset \tcD\!{}_v^\square$ satisfying the following conditions:
\begin{enumerate}
\item The composition $u_v\circ \tcD_v$ is a subfunctor of $\cD_v^\square$.
\item The functor $\tcD_v$ is represented by a quotient $\tR_v$ of $\tR{}_v^\square$.
\end{enumerate}
\end{definition}

One has the following analog of Lemma~\ref{lem_local}.
\begin{lemma}\label{lem_modiflocal}
Let $\tpi\colon \tR_v^\square\to \tR_v$ be a surjective morphism in $\mathrm{CNL}_{\cO}$, and let the subring $R_v\subset \tR_v$ be the image of $R_v^\square$ with induced surjection $\pi\colon R_v^\square\to R_v$. Suppose that 
\begin{enumerate}
\item The ring $\tR_v$ is reduced, and not isomorphic to $k$.
\item The surjection $\pi$ satisfies condition 2 of Lemma~\ref{lem_local}.
\end{enumerate}
Then the subfunctor $\tcD_v$ of $\tcD\!{}_ v^\square$ defined by $\tR_v$ is a modified local deformation problem. 
\end{lemma}
\begin{proof}
From the explicit description of $\tR_v^\square$ it is clear that there exists $x\in \tR_v$ such that $\tR_v=R_v[x]$ and $x$ satisfies a monic quadratic polynomial over $R_v$. Since we also assume that $\alpha_v$ lies in $k$, and since $\tR_v$ is reduced, condition 1 implies that $R_v$ is different from $k$ and hence by Lemma~\ref{lem_local}, $R_v$ defines a local deformation problem $\cD_v$ represented by $R_v$. It follows that $\tcD\!{}_v$ represented by $\tR_v$ is a modified local deformation functor in the sense of Definition~\ref{modif-def_gen}.
\end{proof}

\subsubsection*{Local deformation conditions}

We now define the local deformation conditions relevant to this work; the resulting lifting rings will be denoted by $R_v^{\tau_v}$, where the superscripts $\tau_v\in \{\fl,\min,\st,\un,\fun,\square\}$ indicate the type of condition used to define~$R_v$, and the corresponding universal lifting by $\rho_v^{\tau_v}$. Our conditions for liftings $r$ of $\rhobar|_{G_v}$ will \emph{always} include the condition $\det r=\varepsilon_p$; we shall not repeat this below. We shall be brief, as we closely follow~\cite[\S~4]{BKM}. 

\smallskip

\noindent For $v\in \Sigma_p$ the extension $F_v/\Q_p$ is unramified, so that Fontaine--Laffaille theory applies. We assume that $\rhobar|_{G_v}$ is flat for all $v\in\Sigma_v$, and we let
\begin{itemize}
\item $R_v^{\fl}$ be the quotient of $R_v^{\square}$ parameterizing \emph{flat} liftings of $\rhobar|_{G_v}$.
\end{itemize}

\noindent For $v\in \Sigma$ we let
\begin{itemize}
\item $R_v^{\min}$ be the quotient of $R_v^\square$  parametrizing \emph{minimally ramified} liftings of $\rhobar|_{G_v}$. If $\rhobar$ is unramified at $v$, then $R_v^{\min}$ parameterizes unramified liftings, and then, occasionally we write $R^{\unr}_v$ for~$R^{\min}_v$.
\end{itemize}

\noindent Let $\olQ\subset \Sigma$ be the subset of those $v$ such that the representation
$\rhobar|_{G_{F_v}}$ is of the form 
\begin{equation}\label{eq:modp-Stbg}
    \left( \begin{array}{cc} \varepsilon_p \chibar&  *\\ 0 & \chibar \end{array} \right)
\end{equation}
with respect to some basis $e_1,e_2$ of $k^2$ and where the character $\chibar$ is unramified;\footnote{Let us note that the set $Q$ here and the sets $Q$ in Sections \ref{sec:deformation} and \ref{sec:ribet} are (related but) in general not the same.} we further assume that the basis is chosen so that $*$ is trivial whenever $\rhobar|_{G_{F_v}}$ is split, which holds if $\rhobar$ is unramified and $\varepsilon_p$ is non-trivial. Also $\chibar$ has to be quadratic and we let $\chi$ be its unique (quadratic) lift to $\cO$. Let $\beta_v=\chi(\Frob_v)$.

\noindent For $v\in \olQ$, we define the Steinberg quotient  $R_v^\st$ of $R_v^\square$  as follows:
\begin{itemize}
\item If $\rhobar$ is ramified at $v$, then $R_v^\st$ is defined to be $R^{\min}_v$.

\item If $\rhobar$ is unramified at $v$, we define $R_v^{\St}$ as the unique reduced quotient of $R_v^\square$ characterized by the fact that the $L$-valued points of its generic fiber, for any finite extension $L/E$, correspond to representations of the form 
 \[ \left( \begin{array}{cc} \varepsilon \chi& \ast \\ 0 & \chi\end{array} \right),\]
and with the additional condition $\chi(\Frob_v)=\beta_v$ in the case $q_v\equiv -1\mod p$. In the latter case, without fixing $\beta_v$,  
$\Spec R_v^{\St}$ would have two components, because here $\varepsilon_p$ is quadratic and unramified; see also~\cite[\S~4]{BKM}.
\end{itemize}

\noindent For $v\in \olQ$ such that $\rhobar|_{G_{F_v}}$ is unramified, we also define:

\begin{itemize}
\item 
The unipotent quotient $R_v^\uni$ of $R_v^\square$ is the unique reduced quotient such that $\Spec R_v^\uni=\Spec R_v^\St\cup\Spec R_v^{\unr}$ inside $\Spec R_v^\square$. If $q\equiv -1\mod p$, then note that $R_v^\st$ depends on $\beta_v$.
\end{itemize}

\begin{itemize}

\item The modified unipotent quotient $\tR_v^\uni$ of $\tR{}_v^\square$ is the  unique reduced quotient of $\tR{}_v^\square$ characterized by the fact that the $L$-valued points of its generic fiber, for any finite extension $L/E$, correspond to pairs $(r,a)$  where $r$ is a representation of the form 
\[\begin{pmatrix} \varepsilon_p\chi &*\\0& \chi\end{pmatrix}\] 
with $\chi$ unramified, and such that $\chi(\Frob_v)=a$.
\end{itemize}
It is clear from the definitions, that the natural map $R_v^\square\to \tR{}_v^\uni$ factors via $R^\square_v\to R_v^\uni\to \tR{}_v^\uni$, and by Lemma~\ref{lem:CalOnModifCover}, the map $R_v^\uni\to \tR{}_v^\uni$ is an isomorphism, unless $q_v\equiv 1\mod p$. 

For a more uniform notation, from now on we write $R_v^\fun$ instead of $\tR{}_v^\uni$.

\medskip

The following result summarizes basic ring theoretic properties of the $R_v^{\tau_v}$.
\begin{proposition}\label{prop:R_v^tau}
The following hold:
\begin{enumerate}
\item We have $R^{\fl}_v \cong  \cO[[x_1,\ldots,x_{3+[F_v:\Q_p]}]]$ for $v\in\Sigma_p$ and $R^{\min}_v \cong \cO[[x_1,x_2,x_3]]$ for $v\in\Sigma$.
\item For $v\in \Sigma$, the ring $R_v^\square$ is a complete intersection, reduced, and flat over $\cO$ of relative dimension~$3$. 
\item For $v\in \olQ$,  the ring $R_v^{\St}$ is Cohen--Macaulay, flat of relative dimension $3$ over $\cO$ and geometrically integral, and if $v$ is not a trivial prime for $\rhobar$, we in fact have $R^{\St}_v\cong \cO[[x_1,x_2,x_3]]$.
\item For each $v\in \olQ$ and each minimal prime $\frp$ of $R_v^\square$, $R_v^\square/\frp$ is flat over $\cO$ and geometrically integral.
\item For $v\in \olQ$ such that in addition $\rhobar$ is unramified at $v$,  the rings $R_v^\uni$ and $R_v^\fun$ are Gorenstein, reduced, and flat over $\cO$ of relative dimension~$3$.
\end{enumerate}
Moreover the rings $R_v^{\tau_v}$ in 1.--5. are the completion of a finite type $\cO$-algebra at a maximal ideal. 
\end{proposition}
\begin{proof}
For all but 5 we refer to  \cite[Prop.~4.3]{BKM} and the references given in its proof. The proof of 5 is given in Lemmas~\ref{Lem:EqnsForUnCase} and~\ref{lemma:props-of-RfunQ} below.
\end{proof}

For each $v\in \Sigma$, fix a $\tau_v\in \{\min,\st,\un,\fun,\square\}$, and let $\tau = (\tau_v)_{v\in \Sigma}$, and define 
\begin{align*}
R_{\loc}^\tau &= \left(\cbigotimes{v\in \Sigma}\, R_v^{\tau_v}\right)\cotimes\left( \cbigotimes{v|p}\,R_v^{\fl}\right).
\end{align*}
We simply write $R_{\loc}$ for $R^\tau_{\loc}$, if $\tau_v=\square$ for all $v$. Note in particular, that for any $\tau$ there is a natural morphism $R_{\loc}\to  R^\tau_{\loc}$, and that it factors via $R^{\tau'}_{\loc}$ where $\tau'$ is obtained from $\tau$ be replacing all $\fun$ by $\uni$.

Proposition~\ref{prop:R_v^tau} and  \cite[Lemma 4.4]{BKM} yield:
\begin{proposition}\label{prop:R_v-Loc}
The ring $R_{\loc}$ is a complete intersection, the ring $R^\tau_{\loc}$ is Cohen--Macaulay, and both are reduced and flat over $\cO$. If $R_v^{\tau_v}$ is Gorenstein for all $v\in\Sigma$, then so is $R^\tau_{\loc}$.

Moreover, each irreducible component of $\Spec R_{\loc}$ is of  the form 
\[\Spec \left[\widehat{\bigotimes_{v\in\Sigma}} \, R_v^\square/\frp^{(v)}\right] 
\widehat{\otimes} \, R^{\fl}_p\]
where each $\Spec R^\square_v/\frp^{(v)}$ is an irreducible component of $\Spec R_v^\square$, i.e., each $\frp^{(v)}$ is a minimal prime of $R_v^\square$.
\end{proposition}

\subsubsection*{Global deformation rings}

Now we set up the notation for the corresponding global deformation rings, following \cite[Section 4.3]{BKM}, where further details can be found.

Let $R$ (resp. $R^\square$) denote the global unframed (resp. framed) deformation ring parameterizing lifts of $\rhobar$ with determinant $\varepsilon_p$  which are unramified outside $\Sigma\cup\Sigma_p$ (together with a choice of basis at every $v\in\Sigma\cup\Sigma_p$), One may non-canonically fix an isomorphism $R^\square = R[[X_1,\ldots,X_{4\#(\Sigma\cup\Sigma_p)+3}]]$, so that we may treat the subring $R$ of $R^\square$ also as a quotient of $R^\square$. One also has a natural map $R_{\loc}\to R^\square$ (and thus a map $R_{\loc}\to R$), by restricting a the global lifting and performing locally a base change. 

Let $\tau = (\tau_v)_{v\in \Sigma}$ be as in the previous subsection. Then we define
\[R^{\square,\tau} = R_{\loc}^\tau\otimes_{R_{\loc}}R^\square \hbox{ and  }R^{\tau} = R_{\loc}^\tau\otimes_{R_{\loc}}R.\]

\section{Computation of  Wiles defect for some  local lifting rings}\label{sec:computations}
\label{sec:ComputeWD}

In this section, $R$ will denote a ring $R_v^{\tau_v}$ as defined in Section~\ref{sec_deformation_theory} for a residual representation $\rhobar_v=\rhobar|_{G_{F_v}}:G_{F_v}\to\GL_2(k)$ at a place $v$ of $F$, and a deformation condition $\tau_v$. We let $q=q_v$ be the cardinality of the residue field of $F$ at $v$ and $e$ the ramification index of $\cO$ over $W(k)$. We also fix an augmentation $\lambda:R\onto\cO$.

Throughout this section, we assume that $q\equiv 1\mod p$ and that $\rhobar_v$ is trivial. 

\begin{definition}
Let $\rho_\lambda:G_{F_v}\to\GL_2(\cO)$ be the representation at $v$ induced from the augmentation $\lambda$. We define the local monodromy invariant $n_v$ of $\lambda$ to be the largest integers $n$ such that $\rho_\lambda(G_{F_v})$ mod $\varpi^n$ has trivial projective image. \footnote{This definition also applies in the case where our current setup is twisted by a character that is quadratic and unramified at $v$. The results of this section also apply to this twisted setup.} 
\end{definition}

The aim of this section is to compute the invariants $D_{1,\lambda}(R)$ and $c_{1,\lambda}(R)$ of Venkatesh and  the Wiles defect $\delta_\lambda(R)$ as attached in Definition~\ref{def:key} to the pair $(R,\lambda)$ for certain types of $\rhobar_v$ and $\tau_v$. The three types of deformation conditions that we shall investigate are weight $2$ Steinberg representations, weight $2$ unipotent representations, and weight $2$ unipotent representations with an additional choice of Frobenius eigenvalue; we call the corresponding cases \cst, \cun\ and \ctun, respectively. We shall see that the invariants will only depend on the monodromy invariant $n_v$ and on the type of deformation condition.

The overall strategy in each case is the same. The actual computations between case \cst\ and cases \cun\ and \ctun\ differ greatly. In each case, we first give (or recall) an explicit description of $R$, as a quotient of a power series ring over $\cO$ modulo some ideal given by explicit relations. Then we need to find a ring $\Rt$ and a morphism $\varphi:\Rt\onto R$ that satisfy property \ref{prop CI}. In the unipotent cases, we also need a morphism $\thetat:S\to \Rt$ as in Lemma~\ref{lem:lifting reg seq}. We greatly benefit from the freedom in choosing $\Rt$ and $\varphi$. Venkatesh's invariants do not depend on this choice. So we do this in a way amenable to computation. Our choices are not `natural', but they `work'.
\footnote{It is shown in  \cite{Shotton2} that  the unrestricted lifting ring $R_v^\square$ of any trivial $\rhobar_v:G_{F_v}\to \GL_n(k)$ is a local complete intersection ring and so the induced surjection $\Rt=R_v^\square\to R$ might appear as a natural candidate for $\thetat$. However for the purpose of computations, this seems not useful. The ring $R_v^\square$ can be significantly more complicated than $R$. For instance in case \cst, the ring $R$ can be defined entirely by quadratic polynomials, whereas the equations defining $R_v^\square$ involve expressions of degree $q$.  The latter makes the sort of computations we need to preform with $\Rt$ quite difficult.
}
They allow us to explicitly compute at least the following objects that by Theorem~\ref{thm:c_1 local} and Theorem~\ref{thm:Der exact sequence} give Venkatesh's invariants: (a) the first two steps in a finite free $\Rt$-resolution of $I=\kernel \varphi$, (b) the $\Rt$-annihilator $\Rt[I]$ of $I$, and (c) the modules of formal differentials $\cOmega_{R}$ and~$\cOmega_{\Rt}$.

The computation of the quantities in (c) is done as in \cite{BKM}. They can be related to $\cO$-linear subspaces of $\cOmega_{\cO[[x_1,\ldots,x_n]]/\cO}$ formed by differentials in the kernel ideal of a surjective presentation $\cO[[x_1,\ldots,x_n]]\to \Rt$, and are not difficult to compute. The resolutions needed for (a) turned out to be manageable, even by hand calculation. The most difficult quantity to compute was (b). In case \cst, we can rely on the rich theory of determinantal rings. In the other cases, we needed explicit bases of $\Rt$ and $R$ as free modules over $S$, and we need to understand the socle of the mod $p$ fiber of the latter rings modulo the standard regular sequence of $S$ and the chain of isomorphisms in the proof of Lemma~\ref{lem:R[I]=omega}.

In the Steinberg case, we were able to perform all computations by hand. For (a) we made use of a standard resolution from commutative algebra, the Eagon-Northcott complex. Also (b) and (c) turned out to be directly computable. The reason is that the ring we consider is the completion of a certain determinantal variety of $2\times 2$-minors of a $4\times 2$-matrix. The equations defining such varieties possess many symmetries and have been much studied in commutative algebra. 

In the unipotent cases, the defining equations had no structure that we could link to well-studied classical varieties. In these two cases, we employed for nearly all computations the computer algebra system Macaulay2.\footnote{We thank Dan Grayson for answering some questions and the Macaulay developers for this useful software.} To do so, we modeled the sequence of maps $S\to \Rt\to R$ by a sequence of rings $ S_A\to \Rt_A \to R_A$ of finite type over $A=\Z[\frac 12]$, in case \cun, or $A=\Z$, in case~\ctun. With the help of Macaulay, and suitable choices of integral models, that we found by experiment, we were able to work out (a)--(c) in fact over $A$ (or over $\Q$ when this was sufficient). Using base change and completion, we convert these computations to answers to (a)--(c) for $S\to \Rt\to R$. Our models in fact work for all primes $p$ simultaneously (with $p\neq2$ in case \cun). The models we find satisfy in particular, that $\Rt_A$ and $R_A$ are finite free over $S_A$, and that certain related models for the mod $p$ fibers of $S\to \Rt\to R$ have the analogous property with the same rank. Finding models that are in addition smooth at the augmentation point in the generic fiber of $\Rt_A$ posed an additional challenge.

Let us also mention here that in Subsection~\ref{Subsect-OnCM}, at the end of this section, we gather some results on Cohen--Macaulay and Gorenstein rings that we use repeatedly. It also contains some elementary results on generating sets on dual modules that were useful in explicit computations in Subsections~\ref{Subsec-FUnip} and~\ref{Subsec-Unip}.

\subsection{Presentations of and basic results on the rings \texorpdfstring{$R$}{RInftyTau}}

\subsubsection*{Case \cst}
In case \cst, the ring $R$ is the Steinberg quotient $R^\st_v$ defined in Section~\ref{sec_deformation_theory}. The set-up is as in \cite[\S~7.2]{BKM} except for two minor differences: Here we choose the base point $(0,0)$ for our coordinates while there it was $(s,t)$. Moreover, here $F_v$
is an arbitrary $l$-adic field, there it was $\Q_l$, where $l$ the prime divisor of~$q$. As recalled in Proposition~\ref{prop:R_v^tau}, the ring $R_v^\st$ is a reduced Cohen--Macaulay domain (but non-Gorenstein), and it is flat over $\cO$ of relative dimension $3$. From \cite[\S~7.3]{BKM} we have the explicit presentation $R_v^\st=\cR/J_\st$ where $\cR=\cO[[a,b,c,\alpha,\beta,\gamma]]$ and $J_\st$ is the ideal of $\cR$ generated by the $2\times 2$-minors of the matrix
\begin{equation}\label{Eq:MinorsOfWhich}
\left(
\begin{array}{cccc}
 \alpha & \beta  & (q\!-\!1\!+\!a)  & b \\
\gamma  & -\alpha  & c  & -a 
\end{array}
\right).
\end{equation}

To describe various explicit calculations to be given below, we denote by $t_{i,j}$ the $2\times 2$-minor for columns $i<j$, and we set 
\[r^\st_1=-t_{1,2}=\alpha^2+\beta\gamma, \quad r^\st_2=t_{2,3}=(q-1+a)\alpha+c\beta, \quad r^\st_3=-t_{3,4}=(q-1+a)a+bc,\] 
and $r^\st_4=-t_{1,3}=(q-1+a)\gamma-c\alpha $, $r^\st_5=-t_{1,4}=a\alpha+b\gamma$, $r^\st_6=-t_{2,4}=a\beta-b\alpha$, so that $J_\st=(r^\st_1,\ldots,r^\st_6)$.

As in \cite[\S~7.2]{BKM}, we consider the augmentation $\lambda\colon R_v^{\st}\to\cO$ given by $\lambda(a)=\lambda(\alpha)=\lambda(c)=\lambda(\gamma)=0$ and $\lambda(b)=s$, $\lambda(\beta)=t$, with $t\in\ffrm_\cO$ non-zero.  

\subsubsection*{Case \ctun}

Fix a lift $\sigma\in G_{F_v}$ of Frobenius. In case \ctun, the ring $R$ is the universal lifting ring $R^\fun_v$ defined in \cite[\S~2.1; called there $R^{\mod}_\ell$]{Calegari} for liftings $\rho$ of $\rhobar_v$ of trivial inertia type together with a choice of eigenvalue $(1+X)$ of $\rho(\sigma)$, and with $\det\rho(\sigma)=q$. In other words, the $p$-adic liftings parameterized by $R^\fun_v$ are those that can be made upper-triangular with unipotent inertia, and with $q(1+X)^{-1}$ and $(1+X)$ as diagonal entries of $\rho(\sigma)$ for some $X$. It is shown in \cite[Lem.~2.4 and its proof]{Calegari}, that we have
\[R^\fun_v=\cR/\cI,\]
where $\cR=\cO[[\alpha,\beta,\gamma,X,a,b,c]]$ and $\cI\subset \cR$ is the ideal generated by the entries of the matrices
\[ N^2, N(A-(1+X)I), (A-q(1+X)^{-1})N , AN-qNA,\det A-q\]
with $A:= \sMat{q(1+X)^{-1}+a&b\\c&1+X-a}$ and $N:=\sMat{\alpha&\beta\\ \gamma &-\alpha}$. The corresponding universal lifting factors through the tame quotient $G^t_q$ of $G_{F_v}$, and if $\tau$ is a topological generator of the inertia subgroup of $G^t_q$, such that $\sigma\tau\sigma^{-1}=\tau^q$, then this lifting is given by $\sigma\mapsto A$ and $\tau\mapsto I+N$.

\begin{lemma}\label{Cor:Ideal-I}
We have $\cI=(r^\fun_1\ldots,r^\fun_9)$ for 
\[ r^\fun_1=\alpha X,\ r^\fun_2=\beta X,\ r^\fun_3=\gamma X,\ r^\fun_4=aq+(a^2+bc)(1+X)-a(1+X)^2, \ r^\fun_5=\alpha^2+\beta\gamma,\]
\[r^\fun_6=\alpha c-\gamma(q-1+a), \ r^\fun_7=\alpha a+\gamma b, \ r^\fun_8= \beta c+\alpha (q-1+a), \ r^\fun_9=\beta a-\alpha b.\]
\end{lemma}

\begin{proof}
We claim that $\cI$ is generated by the elements $\alpha X,\beta X,\gamma X$, $\det A-q$, $\det N$ together with the entries of the $2\times 2$-matrix $N(A-(1+X )I)$ with $X$ specialized to zero. From the claim, and in particular $\alpha X,\beta X,\gamma X\in\cI$, it is straightforward to see that the $r^\fun_i$, $i=1,\ldots,9$ generate $\cI$.

To show the claim, denote for a $2\times2$-matrix $D$ over a ring $R$ by $D^\iota$ the main involution applied to $D$ as in the proof of \cite[Lem.~7.2]{BKM}; recall that it is $R$-linear and satisfies  $ D+D^\iota=\tr D\cdot I$, and that, up to sign, the set of entries of $D$ and $D^\iota$ are the same.

It follows that $N^\iota=-N$ and $A^\iota=-A+(q(1+X)^{-1}+(1+X))I$, and from this one deduces that the matrix $(A-q(1+X)^{-1})N$ is obtained from $N(A-(1+X)I) $ via the main involution. Hence either the entries of $N(A-(1+X)I) $ or those of $(A-q(1+X)^{-1})N$ can be omitted when generating $\cI$.

The vanishing of $N^2$ is easily be seen equivalent to that of $\det N$. It remains to show that assuming $N(A-(1+X)I)=0$, we have $AN=qNA$ $\Longleftrightarrow $ $\alpha X=\beta X=\gamma X=0$: To see `$\Rightarrow$', we compute
\[0=qN\cdot (A-(1+X)I)=qNA -q(1+X)N=AN -q(1+X)N= (A-q (1+X)I)N. \]
Subtracting the latter from $(A-q(1+X)^{-1})N=0$ yields $ q (1+X-(1+X)^{-1})N=0$, and from this it is straightforward to see that $XN=0$, i.a., that $\alpha X=\beta X=\gamma X=0$. 
For `$\Leftarrow$', observe that the steps can be reverted.
\end{proof}

\begin{lemma}\label{lemma:props-of-RfunQ}
The ring $R^\fun_v$ has the following properties:
\begin{enumerate}
\item It is reduced, flat over $\cO$ and of relative dimension~$3$. 
\item Its two minimal primes $I_1$ and $I_2$ can be labeled so that $R^\fun_v/I_1$ parameterizes unramified liftings of $\rhobar$ with a choice of Frobenius eigenvalue, and $R^\fun_v/I_2$ is the Steinberg lifting ring $R_v^\st$ from case~\cst.
\item The elements $\varpi, b-c,b-\beta,X-\gamma$ form a regular system of parameters and $R^\fun$ is Gorenstein.
\end{enumerate}
\end{lemma}
\begin{proof}
Part 1 is \cite[Lem.~2.2]{Calegari}. To see 2, set $\cI_1=\cI+(\alpha,\beta,\gamma)$ and $\cI_2=\cI+(X)$. From the description of $R^\fun_v$ and its universal lifting, it follows that the rings $R^\fun_1/\cI_j$ have the moduli interpretation we claim in 2. It remains to show $\cI\supseteq\cI_1\cap\cI_2$. Observe first that $\cR/\cI_1\cong \cO[[a,b,c,X]]/(aq+(a^2+bc)(1+X)-a(1+X)^2 )$ is a domain because $aq+(a^2+bc)(1+X)-a(1+X)^2 $ cannot be factored in the regular ring $ \cO[[a,b,c,X]]$. Hence $X$ is a non-zero divisor in this quotient. Suppose now that we are given $r+sX$ with $r\in\cI$ and $s\in\cR$ that in addition lies in $\cI_1$. Reducing modulo $\cI_1$ yields $s\in\cI_1$ and hence $sX\in\cI_1\cdot\cI_2\subset \cI$. This concludes~2.

We prove 3. The ring $\cR/\cI_2$ is isomorphic to $R_v^\st$  and hence Cohen--Macaulay of dimension $4$. The rings $\cR/\cI_1$, given explicitly above, and its quotient by $X$, i.e., the ring $\cR/(\cI_1+\cI_2)$, are Cohen--Macaulay of dimension $4$ and $3$, respectively, by Proposition~\ref{prop:matsum-GorCM}.3. Hence $R_v^\fun=\cR/(I_1\cap I_2)$ is Cohen--Macaulay of dimension $4$ by \cite[Exerc.~18.13]{Eisenbud}. In particular systems of parameters of $R_v^\fun$ are regular sequences by~Proposition~\ref{prop:matsum-GorCM}.

Let now $A$ be the quotient  of $R_v^\fun$ modulo the sequence $\varpi,b-c,b-\beta,X-\gamma$. The relations allow one to eliminate the variables $c,\beta,\gamma$, and after some simple manipulations one finds
\[A \cong k[[a,b,X,\alpha]]/ (\alpha X,bX,X^2, a^2-2aX  ,\alpha^2, \alpha b-aX, \alpha a, b^2, ab-aX). \]
It is a $k$-vector space of dimension $6$ with basis $1,a,b,X,\alpha,a^2$ and one computes $\socle A=ka^2$. Hence the sequence $\varpi,b-c,b-\beta,X-\gamma$ is regular and it follows from Proposition~\ref{prop:matsum-GorCM} that $R_v^\fun$ is Gorenstein. 
\end{proof}

We consider the `same' augmentation as in case \cst, namely the $\cO$-algebra map $R_v^\fun\to\cO$ that is the projection $R^\fun_v\to R^\fun_1/I_2=R_v^\st$ from Lemma~\ref{lemma:props-of-RfunQ}.2. composed with the augmentation $R^\st_v\to \cO$ from case \cst. Concretely $\lambda$ is given by 
\[a\mapsto0,X\mapsto0,c\mapsto0,\alpha\mapsto0,\gamma\mapsto0 ,  b\mapsto s, \beta\mapsto t\] for some $s,t\in\ffrm_\cO$ with $t$ non-zero.

\subsubsection*{Case \cun}

One has natural surjections $R^\square_v\to R_v^\st$ and $R^\square_v\to R_v^\unr$. Denote by $I^\st$ and $I^\unr$ the corresponding ideals of $R^\square_v$. Then in the case \cun, we define $R$ as the quotient 
\[R_v^\un =R_v^\square/(I^\st\cap I^\unr),\]
cf.~\cite[Rem.~5.7]{Shotton} for a comparable definition. In other words, $R_v^\un$ is the reduced quotient of $R_v^\square$ such that $\Spec R^\un=\Spec R_v^\st\cup\Spec R_v^\unr\subset \Spec R_v^\square$; see Lemma~\ref{Lem:EqnsForUnCase}. 

The ring $R_v\square$ is can be realized as the quotient $\cR'/\cI'$ for $\cR'=\cO[[\alpha,\beta,\gamma,\delta,a,b,c,X]]$ and $\cI'\subset \cR'$ as the ideal generated by the entries of the ($2\times2$- and $1\times1$-) matrices
\[ AB-B^q A,\det A-q,\det B=1\]
with $A:= \sMat{q+a&b\\c&1-a-X}$ and $B:=\sMat{1+\alpha&\beta\\ \gamma &1+\delta}$. The ideals $I^\unr$ and $I^\st$ both contain $\alpha+\delta$ since these quotient describe situations where either $N=B-I$ is zero, or $N$ is of trace and determinant zero. Therefore $R_v^\un$ can be written as a quotient of $\cR=\cO[[\alpha,\beta,\gamma,a,b,c,X]]$ by an ideal $\cI^\un\subset\cR$; with $\delta=-\alpha$.

We computed in Macaulay2 generators of $I^\unr$ and $I^\st$ by working inside the polynomial ring $\cR_\Z=\Z[\uq,a,b,c,X,\alpha,\beta,\gamma]$, where we represent the prime $q$ in $\Z$ by the indeterminate $\uq+1$ in the polynomial ring. Let $I_\Z^\unr$ and $I^\st_\Z$ denote the corresponding ideals of $\cR_\Z$. Then we let Macaulay also compute the intersection $I_\Z^\un=I_\Z^\unr\cap I_\Z^\st$. The ideal $I_\Z^\un$ is generated by the elements 
\[ r^\un_1=X\gamma,\  r^\un_2=X\beta,\  r^\un_3= X\alpha ,\  r^\un_4=\alpha^2 + \beta\gamma,\  r^\un_5= b\alpha  - a\beta,\ r^\un_6= a\alpha + b\gamma, \]
\[ r^\un_7= c\beta - b\gamma + \uq \alpha,\  r^\un_8= c\alpha - a\gamma - \uq \gamma,\  r^\un_9= a^2 + bc + aX + \uq a + (\uq+1) X.\]
We also have $I^\unr_\Z=( \alpha,\beta,\gamma)$ and $I^\st_\Z=(X,r_4^\un,\ldots,r_9^\un)$. We shall use the same names $r^\un_i$ for the corresponding elements in $\cR$, with the silent assumption that in $\cR$ we replace $\uq$ by~$q$.

\begin{lemma}\label{Lem:EqnsForUnCase}
The ring $R^\un_v=\cR/\cI^\un$ with $\cR=\cO[[\alpha,\beta,\gamma,X,a,b,c]]$ and $\cI^\un=(r^\un_1,\ldots,r^\un_9)$ has the following properties:
\begin{enumerate}
\item We have $\cI^\un=\cI^\unr\cap\cI^\st$ for $\cI^\unr=\cI+(\alpha,\beta,\gamma)$ and $\cI^\st=\cI+(X)$ so that $\cR/\cI^\unr$ and $\cR/\cI^\st$ are identified with the unramified and the Steinberg quotient of $R_v^\un$, respectively.
\item The ring $R_v^\un$ is Cohen--Macaulay, flat over $\cO$ and of relative dimension~$3$ and reduced. 
\item The elements $\varpi, b-c,b-\beta,X-\gamma$ for a regular system of parameters and $R_v^\un$ is Gorenstein.
\end{enumerate}

\end{lemma}
\begin{proof}
Part 1 is clear, except for the containment $\cI^\un\supset\cI^\unr\cap\cI^\st$. Similar to Lemma~\ref{lemma:props-of-RfunQ}, the quotient $\cR/\cI^\unr\cong\cO[[X,a,b,c]]/(r_9^\un)$ is a Cohen--Macaulay domain of dimension $4$. The inclusion $\cI^\un\supset\cI^\unr\cap\cI^\st$ now follows as in the proof of Lemma~\ref{lemma:props-of-RfunQ}, and this completes part 1. Because of part 1, the central factors in the short exact sequence of $\cR$-modules $0\to \cR/\cI^\un\to \cR/\cI^\unr\times\cR/\cI^\st\to\cR/(\cI^\unr+\cI^\st)\to 0$ are domains, and so $R_v^\un$ is reduced. Both and also $\cR/(\cI^\unr+\cI^\st)\cong \cO[[a,b,c]]/(r_3^\st)$ are Cohen--Macaulay of dimensions $4$, $4$ and $3$, respectively. As before we find that $R_v^\un$ is Cohen--Macaulay of dimension $4$ by \cite[Exerc.~18.13]{Eisenbud}.

Finally one verifies, by hand or via Macaulay, that $\cR_\Z/(\cI_\Z^\un+(\uq, b-c,b-\beta,X-\gamma))$ is a free $\Z$-module of rank $6$ with basis $1,a,aX,\alpha,\beta,\gamma$ and socle $aX$. By reduction module any prime number $p$, one deduces that  $R_v^\un/(\varpi,b-c,b-\beta,X-\gamma)$ is a zero-dimensional Gorenstein ring, using Proposition~\ref{prop:matsum-GorCM}; when passing to the reduction, one has to explicitly consider the bilinear paring that results from the multiplication of two arbitrary linear forms in $a,\alpha,\beta,\gamma$, and show that it remains non-degenerate under any reduction. The same proposition then also  implies that $\varpi,b-c,b-\beta,X-\gamma$ is a regular sequence in $R_v^\un$ and that $R_v^\un$ is Gorenstein. In particular $\varpi$ is a non-zero divisor and this shows that $R_v^\un$ is flat over~$\cO$. 
\end{proof}
\begin{remark}
One can also work out the above argument by first working out properties for $\cR_\Z$, $\cI^\st$, $\cI^\unr$ and $\cI^\un$, and then completing at $\ffrm_Z=(p,\uq,\alpha,\beta,\gamma,X,a,b,c)$, and then passing to the quotient modulo $\uq-(q-1)$. The above direct argument is shorter.
\end{remark}

\subsection{Steinberg deformations at trivial primes}
\label{Subsec-St}

\begin{lemma}\label{Sec6-FirstLemma}
\begin{enumerate}
\item The elements $r^\st_1,r^\st_2,r^\st_3,\gamma-\beta,c+b,\beta+b,\varpi$ of $\cR=\cO[[a,b,c,\alpha,\beta,\gamma]]$ form a regular sequence.
\item The complete intersection $\Rt:=\cR/(r^\st_1,r^\st_2,r^\st_3)$ is flat over $\cO$ and of relative dimension $3$.
\item The point  in $\Spec \Rt[\frac1\varpi]$ corresponding to the augmentation $\lambdat\colon \Rt\to \cO$ given by the same prescription as $\lambda$ is formally smooth.
\end{enumerate}\end{lemma}
\begin{proof}
1. It suffices to show that $\cR$ modulo the ideal generated by the given sequence is finite. Modding out $\gamma-\beta,c+b,\beta+b,\varpi$ from $\cR$, we need to show that $k[[a,b,\alpha]]$ modulo the $2\times 2$-minors $t_{1,2}, t_{2,3}, t_{3,4}$ of the matrix
\[
\left(
\begin{array}{cccc}
 \alpha & -b  & a  & b \\
-b  & -\alpha  & -b  & -a 
\end{array}
\right).
\]
is finite. Using the relation $\alpha a+b^2$ as an relation for $b$, it follows that the quotient ring is a degree $2$ extension of $k[[a,\alpha]]/(\alpha^2-a\alpha,a^2+a\alpha)$, and the latter ring is finite, as $p>2$; a $k$-basis is $1,a,\alpha,a\alpha$.

2. The regular sequence in 1.\ remains a regular sequence under any reordering and after any truncation. This shows that $\Rt$ is flat over $\cO$ and of relative dimension $3$ over $\cO$.

3. To see the formal smoothness, we form the Jacobian matrix of $r^\st_1,r^\st_2,r^\st_3$ relative to the variables of $\cR[\frac1\varpi]$ and evaluate at the augmentation. This gives
\[
\left(
\begin{array}{cccccc}
 0 & 0  & 0  & 2\alpha & \gamma & \beta \\
\alpha  & \gamma  &0  & q-1+a & 0 & b\\ 
2a +q-1& c &b & 0 &0 & 0
\end{array}
\right)
\stackrel{\mathrm{eval.\, at\,}\lambdat}\longrightarrow
\left(
\begin{array}{cccccc}
 0 & 0  & 0  & 0 & 0 & t \\
0 & 0  &0  & q-1 & 0 & s\\ 
q-1& 0 &s & 0 &0 & 0
\end{array}
\right)
\]
Columns $1,4,6$ witness the formal smoothness asserted for $\lambdat$.
\end{proof}
\begin{remark}
From the proof of Lemma~\ref{Sec6-FirstLemma}.1, one deduces that as an $\cO$-algebra map $S=\cO[[y_1,y_1,y_3]]\to \Rt$ one can take 
\[y\mapsto \gamma-\beta,\quad y_2\mapsto c+b,\quad y_3\mapsto \beta+b.\]
Similar to the proof of Lemma~\ref{Sec6-FirstLemma}.1, one can show that $R_v^\st/(\varpi,y_1,y_2,y_3)\cong k[a,\alpha,\gamma]/(a,\alpha,\gamma)^2$. Its socle is obviously spanned by $\{a,\alpha,\gamma\}$ and has thus $k$-dimension $3$. Using that $R^\st_v$ is local Cohen--Macaulay of dimension $4$, by combining parts 3, 2 and 1 of Proposition~\ref{prop:matsum-GorCM} one deduces that $R_v^\st$ is not Gorenstein.
\end{remark}

In the following let $\Rt=\cR/(r^\st_1,r^\st_2,r^\st_3)$ and $I=\ker(\Rt \to R_v^\st)$. We need some preparations to give a presentation of $I$ as an $\Rt$-module. Recall that $J_\st$ was defined before formula~\eqref{Eq:MinorsOfWhich}.

\begin{lemma}
The sequence of $\cR$-modules $\cR^8\stackrel{A}\to \cR^6\stackrel{B}\to J_\st\to 0$ is exact, where $B$ is the $1\times 6$-matrix $(r^\st_1,r^\st_2,\ldots,r^\st_6)$ and $A$ is the $8\times 6$-matrix 
\[\left(
\begin{array}{cccccccc}
b& q-1+a &0&0& -a&c&0&0 \\
0&-\alpha &-b&0 &0&-\gamma  &a&0 \\
0&0&\beta&\alpha&0 &0 & -\alpha &\gamma\\
0&\beta&0&-b&0  & -\alpha &0&a \\
\beta&0&0& q-1+a & -\alpha &0&0&c \\
-\alpha&0& q-1+a &0& -\gamma &0&c&0 \\
\end{array}
\right).
\]
\end{lemma}
\begin{proof}
The displayed presentation is part of the Eagon-Northcott complex  attached to the $4\times 2$-matrix from \eqref{Eq:MinorsOfWhich}, considered as an $\cR$-linear map $\nu\colon \cR^4\to\cR^2$, and in the present case, this complex is exact: The Eagon-Northcott complex is described in detail in \cite[\S~11H]{Eisenbud-Syzygies}, which we now recall in parts. We follow the notation of \cite{Eisenbud-Syzygies} and set $G=\cR^2$ and $F=\cR^4$. Then in the case at hand, the Eagon-Northcott complex is the complex
\[0\longrightarrow (\Sym^2 G)^*\otimes \bigwedge^4F \stackrel{d_2}\longrightarrow  (\Sym^1 G)^*\otimes \bigwedge^3F \stackrel{d_1}\longrightarrow (\Sym^0 G)^*\otimes \bigwedge^2F \longrightarrow \bigwedge^2 G;\]
choosing bases $f_1,\ldots,f_4$ of $F$ and $g_1,g_2$ of $G$, the complex is seen to be of the form $0\to \cR^3\to \cR^8\to \cR^6\to \cR$; the right most map of the complex sends the basis element $f_i\wedge f_j$, $i<j$, to the minor $t_{i,j}$ of \eqref{Eq:MinorsOfWhich} formed by column $i$ and column $j$, and thus its image  is the ideal~$J_\st$. 

To describe the maps $d_i$, let $\Gamma_j\colon (\Sym^j G)^*\to G^*\otimes (\Sym^{j-1} G)^*$ be the map dual to the multiplication map $ G\otimes \Sym^{j-1} G \to \Sym^j G$, and let $\Phi_k \colon \bigwedge F^k \to F\otimes \bigwedge^{k-1} F$ be the $\cR$-linear map given on basis elements by $f_{i_1}\wedge\ldots \wedge f_{i_l}\mapsto \sum_{j=1}^k  (-1)^{j-1} f_{i_j} \otimes f_{i_1}\wedge\ldots\wedge \widehat{f_{i_j}} \wedge \ldots \wedge f_{i_l}$. Then for a pure tensor $u\otimes v$ in $ (\Sym^j G)^*\otimes  \bigwedge F^{j+2}$ one has 
\[d_j(u\otimes v)= \sum_{l,m} \big(\nu^*(u'_l)(v'_m) \big) u''_l\otimes v''_l   ,\]
if we write $\Gamma_j(u)=\sum_l u'_l\otimes u''_l $ and $\Phi_{j+2}(v)=\sum_m v'_m\otimes v''_m$. This procedure can be applied to the basis $g_l\otimes f_{i_1}\wedge f_{i_2}\wedge f_{i_3}$, $1\le i_1 < i_2 < i_3 \le 4$, of $(\Sym^1 G)^*\otimes \bigwedge^3F$ to obtain the matrix $A$.

To complete the proof, it remains to show exactness of the Eagon-Northcott complex in the case at hand. By \cite[Thm.~11.35]{Eisenbud-Syzygies} this holds if and only if the grade of the ideal $J_\st$ attains the maximal value possible, namely the height of $J_\st$; see \cite[p.~103]{Matsumura-CA}. Because $\cR/J_\st=R_v^{\st}$ has Krull dimension $4$, the height of $J_\st$ is $3$. The grade of $J_\st$ is the maximal length of a regular sequence of $\cR$ contained in $J_\st$,  see \cite[p.~103]{Matsumura-CA}, and because of Lemma~\ref{Sec6-FirstLemma} this number is at least~$3$.
\end{proof}
\begin{lemma}
Let $\cR^m\stackrel{A}\to \cR^n\stackrel{B}\to J\to 0$ be a right exact sequence of $\cR$-modules for $J$ an ideal of $\cR$. We consider $A$ as an $n\times m$-matrix and $B$ as a $1\times n$-matrix over $\cR$. Decompose $n=n'+n''$ with $n',n''>0$, and decompose correspondingly the matrix $A$ into $A'$ and $A''$ of size $n'\times m$ and $n''\times m$, and the matrix $B$ into matrices $B'$ of size $1\times n'$ and $B''$ of size $1\times n''$, respectively. Let $J'\subset J$ be the image of $\cR^{n'}$ under $B'$. Then the induced sequence of $\cR/J'$-modules
\[ (\cR/J')^m \stackrel{A''\!\!\!\!\!\! \pmod {J'}}\longrightarrow (\cR/J')^{n''} \stackrel{B'' \!\!\!\!\!\!\pmod {J'}}\longrightarrow J/J'\to 0\] 
is right exact.
\end{lemma}
\begin{proof}
By the definition of $J'$, the map defined by $B''\!\!\pmod {J'}$ is clearly surjective. Also, $BA=0$ implies $B'A'=-B''A''$ as maps on $\cR^m$. But $B'\!\!\pmod{J'}$ is the zero map, and hence $(B''\!\!\pmod {J'})(A''\!\!\pmod {J'})=0$. It remains to show that $\ker(B''\!\!\pmod {J'})\subset \image(A''\!\!\pmod {J'})$. For this let $x''\in\cR^{n''}$ represent an element $x''\!\!\pmod {J'}$ in $\ker(B''\!\!\pmod {J'})$. Then $B''x''$ lies in $J'$ and hence we can find $x'\in \cR^{n'}$ such that $B''x''=B'x'$. We let $x=(-x',x'')\in\cR^n$, so that $Bx=0$. By the exactness of the given complex, we can find $y\in \cR^m$ such that $Ay=x$. But then $A''y=x''$ and this implies $x''\!\!\pmod {J'}\in \image(A''\!\!\pmod {J'})$.
\end{proof}
By combining the previous two lemmas, we find.
\begin{corollary}\label{Cor:PresOfIinfty}
As a module over $\Rt$ the ideal $I$ has a presentation
\[ (\Rt)^8\stackrel{A'}\longrightarrow (\Rt)^3\longrightarrow I \to0,  \]
where $A'$ is the matrix
\[\left(
\begin{array}{cccccccc}
0&\beta&0&-b&0  & -\alpha &0&a \\
\beta&0&0& q-1+a & -\alpha &0&0&c \\
-\alpha&0& q-1+a &0& -\gamma &0&c&0 \\
\end{array}
\right).
\]
\end{corollary}
\begin{corollary}\label{Cor:LambdaOfFitt0}
We have $\lambdat( \Fitt_0^{\Rt}(I) )=(q-1)t(s,t,q-1)\subset \cO$. 
\end{corollary}
\begin{proof}
The ideal $\Fitt_0^{\Rt}(I)$ is the ideal generated by the $3\times 3$-minors of the matrix $A'$ from Corollary~\ref{Cor:PresOfIinfty}. Hence its image under $\lambdat$ is the ideal of $\cO$ generated by the $3\times 3$-minors of 
\begin{equation}
\label{Eq:AprimeUnderLambda}
\lambdat(A')=\left(
\begin{array}{cccccccc}
0&t&0&-s&0  & 0&0&0 \\
t&0&0& q-1 & 0&0&0&0 \\
0&0& q-1 &0& 0 &0&0&0 \\
\end{array}
\right).
\end{equation}
This is the ideal $(t^2(q-1),t(q-1)^2,ts(q-1))=t(q-1)(s,t,q-1)$.
\end{proof}
\begin{remark}
One can also show that $\Fitt^{\Rt}_0(I)$ equals 
\[((q-1+a)^2\beta,\!(q-1+a)b\beta,\!(q-1+a)c\beta,\!(q-1+a)\beta^2,\!(q-1+a)\beta\gamma,ac\alpha,ac\beta,bc\alpha,bc\beta,
c^2\alpha,c^2\beta,c\alpha\beta,c\alpha\gamma,c\beta\gamma,c\beta^2).\]
\end{remark}
\begin{corollary}\label{Cor:I/Isquared}
We have $\Hom_{R^\st_v}(I/I^2,E/\cO)\cong \cO/(s,t,q-1) \times  \cO/(t,q-1) \times  (t,q-1))/(t(q-1))$. 
\end{corollary}
\begin{proof}
Note first that 
\[\Hom_{R_v^\st}(I/I^2,E/\cO) \cong \Hom_{R_v^\st}(I\otimes_{\Rt}R_v^\st,E/\cO)
\cong \Hom_{\cO}(I\otimes^{\lambdat}_{\Rt}\cO,E/\cO),\]
where in the second isomorphism, we use that $E$ is regarded as a $R^\st_v$-module via the augmentation $\lambda$. Tensoring now the presentation of $I$ in \autoref{Cor:PresOfIinfty} with $\cO$ over $\Rt$ (via $\lambdat$) gives the right exact sequence of $\cO$-modules
\begin{equation}
\cO^8\stackrel{\lambdat(A')}\longrightarrow \cO^3\longrightarrow I \otimes^{\lambdat}_{\Rt}\cO \to0 
\end{equation}
Using the theory of invariant factors and elementary divisors of matrices over PID's, e.g. \cite[Thm.~3.9]{Jacobson-BA1}, the cokernel of this sequence is seen to be isomorphic to $\prod_{i=1}^3 \cO/d_i\cO$ where $d_1$, $d_1d_2$ and $d_1d_2d_3$ are the gcd's of the $i\times i$-minors of $\lambdat(A')$displayed in \eqref{Eq:AprimeUnderLambda} for $i=1,2,3$. One readily computes
\[d_1=\gcd(s,t,q-1),\]
\[ d_2=\gcd(t^2,t(q-1),t(q-1),ts,(q-1)s,(q-1)^2)=\gcd(s,t,q-1)\gcd(t,q-1),\]
\[ d_3=\gcd(t^2(q-1),t(q-1)s,t(q-1)^2)=t(q-1)\gcd(s,t,q-1),\]
and this implies the assertion of the corollary.
\end{proof}

\begin{lemma}\label{Lem:RI-infty} 
For the ideals $P=(\alpha,\beta)$, $Q=(q-1+a,c)$ and $I'=((q-1+a)\alpha,(q-1+a)\beta,c\alpha,c\beta)$ of $\Rt$ the following hold.
\begin{enumerate}
\item $P$ is a prime ideal and $P=\{x\in \Rt \mid x r^\st_4=0\}$.
\item $Q$ is a prime ideal and $Q=\{x\in \Rt \mid x r^\st_6=0\}$.
\item One has (a) $P\cap Q=\Rt[I]$ and (b) $P\cap Q=I'$.
\end{enumerate}
\end{lemma}
\begin{proof}
1.\ Note first that $\Rt/P$ is isomorphic to the quotient of $\cO[[a,b,c,\beta]]/((q-1+a)a+bc)$. Since this is a domain, $P$ is a prime ideal. Next observe that $\alpha$ and $\beta$ annihilate $r^\st_4$ as follows by considering columns $2$ and $6$ in the relation matrix $A'$ in Corollary~\ref{Cor:PresOfIinfty}. It remains to show that $P=(\alpha,\beta)$ contains $\{x\in \Rt\mid xr^\st_4=0\}$. So suppose that $xr^\st_4=0$ in $\Rt$. The main observation is that $r^\st_4\mod P=(q-1+a)\gamma$ is a non-zero element in the domain $\Rt/P$. Therefore $x\mod P$ is zero, and thus $x\in P$, as had to be shown. The proof of 2.\ is completely parallel to that of 1.\ and left to the reader.

From the definition of $P$, $Q$ and $I'$ it is clear that $I'\subset P\cap Q$. It is also straight forward to see from the columns of $A'$ that $I'$ annihilates $r^\st_5$ (multiply the first column and the fifth column by $c$ or by $(q-1+a)$, and use 1.; or alternatively, multiply the forth and the eighth column by $\alpha$ and $\beta$, and use 2.). We shall now prove (b), and from this and what we already proved, (a) will follow.

To see (b), let $x$ be in $P\cap Q$. Write $x=f_1\alpha+f_2\beta$. To show that $x$ lies in $I'\subset P\cap Q$, we may subtract from $x$ arbitrary elements in $I'$. Writing elements in $\cR$ as power series over $\cO$ in $q-1+a,b,c,\alpha,\beta,\gamma$, we may thus assume $f_1,f_2\in(\alpha,\beta,\gamma,b)$. Shifting multiples of $\alpha$ in $f_2$ to $\alpha f_1$, we may further assume $f_2\in(\beta,\gamma,b)$, and using $r^\st_1$ we can replace $\beta\gamma$ by $\alpha^2$, and this finally allows us to assume that $f_2$ lies in $(b,\beta)$. We now reduce $x\in P\cap Q$ modulo $Q$. This yields $f_1\alpha+f_2\beta=0$ in $\cO[[\alpha,\beta,\gamma,b]]/(\alpha^2+\beta\gamma)$. In other words, we can find $f_3\in \cR':=\cO[[\alpha,\beta,\gamma,b]]$ such that 
\[f_1\alpha+f_2\beta+f_3(\alpha^2+\beta\gamma)=0 \hbox{ in }\cR'.\]
Reducing modulo $\alpha$ and using $f_2\in (b,\beta)$ it follows that $\gamma$ had to divide $f_2$ and hence that $f_2=0$. Since $\cR'$ is a UFD it follows that $r^\st_1=\alpha^2+\beta\gamma$ divides $f_1$ and hence that $f_1=0$ in $\Rt$. Hence we proved that $x$ lies in~$I'$.
\end{proof}
\begin{corollary}\label{Lem:c2-RqSt}
Let $e$ be the ramification index of $E$ over $\Q_l$. Then $\lambdat( \Rt[I])=(q-1)t \subset \cO$ and $c_{1,\lambda}(R^\st_v)=\frac1e\log_p(\cO/(s,t,q-1))=\frac{n_v}e$.
\end{corollary}
\begin{proof}
In Lemma~\ref{Lem:RI-infty} we identified $\Rt[I]$ with $I'$. The image of $I'$ under $\lambda$ is simply $((q-1)t)$. Invoking also Corollary~\ref{Cor:LambdaOfFitt0}, we deduce
\[c_{1,\lambda}(R^\st_v)=\frac1e \cdot \log_p \big( \# ((q-1)t)/((q-1)t(s,t,(q-1)))\big)  = \frac1e \cdot \log_p\big( \# \cO/(s,t,q-1) \big).\]
\end{proof}
To complete the computation of $D_{1,\lambda}(R_v^\st)$, we still have to compute the size of the cokernel of $\Hom_{R_v^\st}(\cOmega_{R_v^\st},E/\cO)\to \Hom_{\Rt}(\cOmega_{\Rt},E/\cO) $. Using the methods of \cite[\S~7.2]{BKM} and its terminology, we need to compute the lattice $\Lambdat\subset\cO^8$ that is the kernel of that natural surjection $\cO^8\cong \cOmega_{\cR/\cO}\otimes_{\cR}^{\lambdat}\cO\to \cOmega_{\Rt/\cO}\otimes_{\Rt}^{\lambdat}\cO$. The lattice $\Lambdat$ is contained in $\Lambda^{\St}$, and the cardinality wanted is $\#(\Lambda^{\St}/\Lambdat)$.
\begin{lemma}\label{Lem:LatticeIndexR'Rst}
The lattice $\Lambdat\subset \cO^8$ is spanned by the rows of the matrix
	\[
	\left( \begin{array}{cccccccc} 
	1&0&0&-1&0&0&0&0\\ 
	0&0&0&0&1&0&0&1\\ 
	0&0&0&0&0&0&t&0\\ 
	0&0&s&q-1&0&0&0&0\\ 
	0&0&t&0&q-1&0&0&0\\ 
	\end{array} \right).
	\]
and the quotient $\Lambda^{\St}/\Lambdat$ as an $\cO$-module is isomorphic to $(s,t,q-1)/(t) \times (s,t,q-1)/(q-1)$.
\end{lemma}
\begin{proof}
In the notation of \cite[\S~7.3]{BKM}, the ring $\Rt$ is given as $\cO[[a,b,c,e,\alpha,\beta,\gamma,\delta]]$ modulo the relations $a-e,\alpha+\delta,\alpha\delta-\beta\gamma, (q-1+a)e+bc,(q-1+a)\delta-c\beta$. The spanning vectors of $\Lambdat$ are then the image of the Jacobian matrix 
\[
	\left( \begin{array}{cccccccc} 
	1&0&0&-1&0&0&0&0\\ 
	0&0&0&0&1&0&0&1\\ 
	0&0&0&0&\delta&-\gamma&-\beta&\alpha\\ 
	e&c&b&q-1+a&0&0&0&0\\ 
	\delta&0&-\beta&0&0&-c&0&q-1+a\\ 
	\end{array} \right)
\]
under the augmentation $\lambdat$. The matrix displayed in the assertion of the lemma is obtained from this image after some simple row operations. By \cite[\S~7.2]{BKM}, the lattice $\Lambda^{\St}$ is spanned by the rows of 
	\[
	\left( \begin{array}{cccccccc} 
	1&0&0&-1&0&0&0&0\\ 
	0&0&0&0&1&0&0&1\\ 
	0&0&0&0&0&0&(s,t,q-1)&0\\ 
	0&0&t&0&q-1&0&0&0\\ 
	0&0&0&(t,q-1)&{\textstyle \frac st}(t,q-1)&0&0&0\\ 
	\end{array} \right) 
	\]
if $ {\rm ord}_\varpi(s)\ge  {\rm ord}_\varpi(t)$; and in the other case, the last two rows have to be replaced by 
	\[
	\left( \begin{array}{cccccccc} 
	0&0&-s&q-1&0&0&0&0\\ 
	0&0&0&{\textstyle \frac ts} (s,q-1)&(s,q-1)&0&0&0\\ 
	\end{array} \right) ,
	\]
In both cases, it is easy to express the basis spanning $\Lambdat$ in terms of the basis spanning $\Lambda^{\St}$, by an upper triangular transition matrix over $\cO$ diagonal entries $1,1,1,\frac{t}{(s,t,q-1)},\frac{(q-1)}{(s,t,q-1)}$. The assertions of the lemma are now clear.
\end{proof}
\begin{corollary}\label{Lem:D2-RqSt}
We have $D_{1,\lambda}(R_v^\st)=\frac1e \cdot \log_p \#( \cO/(s,t,q-1))^3 =3\frac{n_v}e$.
\end{corollary}
\begin{proof}
From Lemma~\ref{Lem:LatticeIndexR'Rst}, the observations preceding it, and from Theorem~\ref{thm:Der exact sequence}, 
we have 
\[\#\kernel(\Hom_{R_v^\st}(I/I^2,E/\cO) \to \cDer^1_\cO(R_v^\st,E/\cO))=\# \Lambda^{\St}/\Lambdat=\#(s,t,q-1)^2/(t(q-1)) \]
In Corollary \ref{Cor:I/Isquared} we computed $\#\Hom_{R_v^\st}(I/I^2,E/\cO)=\#\cO/(t(q-1)(s,t,q-1))$. Forming the quotient, the result follows from Theorem~\ref{thm:Der exact sequence}.
\end{proof}
From $\delta_\lambda(R_v^\st) = D_{1,\lambda}(R_v^\st)-c_{1,\lambda}(R_v^\st)$ and Corollaries~\ref{Lem:c2-RqSt} and~\ref{Lem:D2-RqSt} we deduce.
\begin{theorem}\label{thm:delta_v^St}
We have $\delta_\lambda(R_v^{\st}) = 2\frac{n_v}e$.
\end{theorem}

\subsection{Unipotent deformations with a choice of Frobenius at trivial primes}
\label{Subsec-FUnip}

In the following, $\us$ and $\ut$ will denote indeterminates that we shall specialize to $s$ and $t$, respectively, whenever we pass to $\cO$-algebras. Set 
$s_1^\fun= r_9^\fun+r_6^\fun$,  
$s_2^\fun= r_8^\fun-r_7^\fun +r_2^\fun$, 
$s_3^\fun= r_5^\fun$, 
$s_4^\fun= r_2^\fun+a r_3^\fun +r_4^\fun+ar_6^\fun-br_7^\fun -r_3^\fun$, and $s_4^{\prime\,\fun}= a r_3^\fun +r_4^\fun+ar_6^\fun-br_7^\fun -r_3^\fun$.
Set also $\widetilde\cI=(s_1^\fun,s_2^\fun,s_3^\fun,s_4^\fun)$ and $\widetilde\cI^{\prime}=(s_1^\fun,s_2^\fun,s_3^\fun,s_4^{\prime\,\fun})$.

The next result summarizes some explicit computations done via Macaulay2.
\begin{lemma}\label{Sec6:Fun--FirstLemma}
\begin{enumerate}
\item The ring $\Z[\uq,\us,\ut,a,b,c,X,\alpha,\beta,\gamma]/(\uq,\us,\ut,b-\us-c, \beta-\ut-c, \gamma-X, s_i^\fun, i=1,\ldots,4)$ is free over $\Z$ of rank $16$. The same holds if we replace $s_4^\fun$ by $s_4^{\prime,\fun}$. A basis is $1,a,aX,aX\alpha,a\alpha, b,b\alpha,$ $X, X^2,X^2\alpha,X\alpha,X\alpha^2,X\alpha^3,\alpha,\alpha^2,\alpha^3$. A basis of the socle of the ring modulo any prime is $X\alpha^3$.
\item The ring $\Z[\uq,\us,\ut,a,b,c,X,\alpha,\beta,\gamma]/((\uq,\us,\ut,b-\us-c, \beta-\ut-c, \gamma-X)+\cI^\fun_\Z)$ is free over $\Z$ of rank $6$. A basis is $1,a,b,bX,X,\alpha$. A basis of the socle of the ring modulo any prime is $X\alpha$.
\item Write $x_1,\ldots,x_7$ for $a,b,c,X,\alpha,\beta,\gamma$. Then the ideal in $\Z[\uq,\us,\ut]$ generated by the $4\times 4$-minors of the Jacobian $(\partial s_i^\fun/\partial x_j)_{i=1,\ldots,4; j=1,\ldots,7}$ evaluated at $(x_1,\ldots,x_7)=(0,\us,0,0,0,\ut,0)$ is $(\us+\ut)\ut^2(\uq,\us, \ut)$. If one replaces $s_4^\fun$ by $s_4^{\prime,\fun}$, the resulting ideal is $\us\ut^2(\uq,\us, \ut)$.
\end{enumerate}\end{lemma}

\begin{remark}\label{Rem:OptimalFun}
We note that the number 16 in part 1 is optimal. After reducing the number of variables by those relations that are linear, the $s_i^\fun$ are quadratic relations of a polynomial ring over $\Z$ in $4$ variables. Now the intersection of 4 quadrics in general position consists of $16$ points. Therefore dimension $16$ for the coordinate ring of the corresponding scheme is optimal.
\end{remark}
Let $s,t\in\ffrm$ with $t\neq0$.
\begin{corollary}\label{Sec6:FunFirstCor}
\begin{enumerate}
\item The ring $\Rt=\cO[[a,b,c,X,\alpha,\beta,\gamma]]/(s_i^\fun, i=1,\ldots,4)$ is a complete intersection, flat over $\cO$ and of relative dimension $3$, and this also holds with $s_4^\fun$ replaced by $s_4^{\prime,\fun}$. One has a natural surjection $\Rt\to R^\fun_v$ induced from $(s_i^\fun, i=1,\ldots,4)\subset (r_j^\fun,j=1,\ldots,9)$.
\item Via the ring map $S=\cO[[y_1,y_2,y_3]]\to \Rt$ given by $y_1\mapsto b-s-c, y_2\mapsto \beta-t-c, y_3\mapsto \gamma-X$, the rings $\Rt$ and $R^\fun_v$ are free $S$-modules of rank $16$ and $6$, respectively (for either choice of $\Rt$).
\item The augmentation $\lambdat\colon\Rt\to\cO$ given by $a,c,X,\alpha,\gamma\mapsto 0$, $b\mapsto s$ and $\beta\mapsto t$ defines a formally smooth point of $\Spec \Rt[\frac1\varpi]$, for at least one of the two choices of $\Rt$ from 1, provided that $t\in\cO\setminus\{0\}$. 
\end{enumerate}
\end{corollary}
\begin{proof}
The quotient $\Rt/(\varpi, b-s-c, \beta-t-c, \gamma-X)$ is isomorphic to the ring from Lemma~\ref{Sec6:Fun--FirstLemma}.1 tensored with $k$ over $\Z$ -- since the latter is isomorphic to $k^{16}$ no completion is necessary. This implies that  $(\varpi, b-s-c, \beta-t-c, \gamma-X,s_i^\fun, i=1,\ldots,4)$ is a regular sequence in $\cR$ with quotient $k^{16}$. We deduce part 1 and the first half of part 2. The second half of part 2 uses  Lemma~\ref{Sec6:Fun--FirstLemma}.2 in an analogous way.

To prove part 3, observe that not both, $s$ and $s+t$ can be zero, since otherwise $t=0$ which is ruled out. So we choose $s_4^\fun$ or $s_4^{\prime,\fun}$ accordingly. Then we evaluate the ideal in Lemma~\ref{Sec6:Fun--FirstLemma}.3 at the made choice. This gives either the non-zero value $(s+t)t^2\gcd(t,st,q-1)$ or $st^2\gcd(t,st,q-1)$ for a generator of the corresponding ideal over $\cO$. This implies the stated formal smoothness.
\end{proof}

Our aim is to compute $D_{1,\lambda}(R_v^\fun)$ and $c_{1,\lambda}(R_v^\fun)$. Instead, we shall compute these invariants for the ring $R_v^\fun\otimes_{S}\cO$, where $S$ is the ring from Corollary~\ref{Sec6:FunFirstCor} and where the map $S\to\cO$ is the augmentation $\lambdat$ composed with $S\to \Rt$. This is allowed due to Theorems~\ref{thm:c_1 local} and~\ref{thm:Der^1 local}.\footnote{Our choices $S\to \Rt\to R_v^\fun$ are almost certainly unrelated to any choices that arise from the Taylor-Wiles-Kisin patching.} It is probably not strictly necessary to perform this base change. However it seems easier to work with Gorenstein and  complete intersection rings that are finite flat over $\cO$. In particular, this will allow us to (have Macaulay) compute structural constants of these rings, namely their multiplication tables in a given $\cO$-bases. In the remainder of this subsection, we consider the rings
\[ \cO \longrightarrow \Ro=\Rt/(y_1,y_2,y_3) \stackrel{\pio}\longrightarrow \Ro^\fun_v=R^\fun_v/(y_1,y_2.y_3),\]
and we let $\Io$ be the kernel of $\pio:\Ro\to\Ro^\fun_v$. 

We first explain the part that for us was the most difficult one, namely the computation of $\Ro[\Io]$. Let $(b_i)_{i=1,\ldots,16}$ be an $\cO$-basis of $\Ro$ such that $(b_i)_{i=7,\ldots,16}$ is a basis of the kernel of $\Ro\to\Ro_v^\fun$. Suppose further that $b_6$ and $b_{16}$ are chosen, so that they reduce to a generators of the socle of the finite Gorenstein rings $\Ro_v^\fun/(\varpi)$ and $\Ro/(\varpi)$, respectively; this is always possible. Denote by $(b_i^*)_{i=1,\ldots,16}$ the dual basis. It follows from Proposition~\ref{prop:GensOfDual}, that $b_6^*$ is a generator of $\Hom_\cO(\Ro_v^\fun,\cO)$ as a free $\Ro_v^\fun$-module, and $b_{16}^*$ of $\Hom_\cO(\Ro,\cO)$ as a free $\Ro$-module. Denote by $\Theta$ the isomorphism $ \Ro\to \Hom_\cO(\Ro,\cO),f\mapsto (b_{16}^*(f\cdot), g\mapsto b_{16}^*(fg) )$ and consider the chain of isomorphisms 
\begin{align*}
\Ro[\Io] 
&\cong \Hom_{\Ro}(\Ro_v^\fun,\Ro)
 \cong \Hom_{\Ro}(\Ro_v^\fun,\Hom_\cO(\Ro,\cO))
 \cong \Hom_\cO(\Ro_v^\fun\otimes_{\Ro}\Ro,\cO)
 \cong \Hom_\cO(\Ro_v^\fun,\cO)
 \end{align*}
from Lemma~\ref{lem:R[I]=omega}. The generator $b_6^*$ on the right is successively mapped to, first $h_1\otimes h_2\mapsto b_6^*(h_1\cdot \pio(h_2)) $, second $(h_1\mapsto (h_2\mapsto b_6^*(h_1\cdot \pio(h_2)))$, third $(h_1\mapsto \Theta^{-1}(h_2\mapsto b_6^*(h_1\cdot \pio(h_2)))$, lastly to 
\[\Theta^{-1}(h_2\mapsto b_6^*(\pio(h_2)))=\Theta^{-1}(h_2\mapsto b_6^*(h_2))=\Theta^{-1}\circ b_6^*.\]

Now write $ \Theta^{-1}\circ b_6^*=\sum_i \mu_i b_i$ with $\mu_i\in\cO$. By the definition of $\Theta$, this is equivalent to $b_6^*(f)=b_{16}^*( \sum_i \mu_i b_i f)$ for all $f\in \Ro$. Let $c_{ijk}\in\cO$ be the structural constants for multiplication in $\Ro$ over $\cO$ with respect to the basis $(b_j)$, so that  $b_ib_j=\sum_k c_{ijk} b_k$. Then substituting for $f$ all basis elements of $\Ro$ over $\cO$ gives 
\[  b_6^*(b_j)=b_{16}^*( \sum_i \mu_i b_i b_j) = b_{16}^*( \sum_{i,k} \mu_i c_{ijk}b_k) = \sum_i \mu_i c_{ij16}.\]
Let $C$ be the matrix $ (c_{ij16})_{i,j=1,\ldots,16}$. Then the row vector $(\mu_i)$ is given as the product $e_6C^{-1}$ for $e_6$ the $6$-th standard basis vector of the column vector space $\cO^{16}$. To obtain $C$, consider the following commutative diagram
\[
\xymatrix@C+3pc{
\frac{\Q[\uq,\us,\ut,a,b,c,X,\alpha,\beta,\gamma]}{(b-\us-c, \beta-\ut-c, \gamma-X, s_i^\fun, i=1,\ldots,4)} \ar[r]
&
\Ro[\frac1\varpi]  \\
\ar[u]
\ar[d]
\frac{\Z[\uq,\us,\ut,a,b,c,X,\alpha,\beta,\gamma]}{(b-\us-c, \beta-\ut-c, \gamma-X, s_i^\fun, i=1,\ldots,4)} \ar[r]^{\us\mapsto s,\ut\mapsto t,\uq\mapsto q-1} 
&
\ar[u]
\ar[d]
\Ro =\frac{\cO[[ a,b,c,X,\alpha,\beta,\gamma]]}{(b-\us-c, \beta-\ut-c, \gamma-X, s_i^\fun, i=1,\ldots,4)}
\\
\frac{\Z[\uq,\us,\ut,a,b,c,X,\alpha,\beta,\gamma]}{(p,\uq,\us,\ut,b-\us-c, \beta-\ut-c, \gamma-X, s_i^\fun, i=1,\ldots,4)} \ar[r]
&
\Ro /(\varpi)\\
}\]
Applying Nakayama's Lemma to the right column, we see that the basis in Lemma~\ref{Sec6:Fun--FirstLemma} is an $\cO$-basis of $\Ro$, and thus an $E$-basis of $\Ro[\frac1\varpi] $. The analogous diagram holds for $\Ro^\fun_v$ in place of $\Ro$. Macaulay computations give us the following lemma:
\begin{lemma}\label{Sec6:Fun--SecondLemma}
\begin{enumerate}
\item The ring $R_1=\Q[\uq,\us,\ut,a,b,c,X,\alpha,\beta,\gamma]/(b-\us-c, \beta-\ut-c, \gamma-X, s_i^\fun, i=1,\ldots,4)$ is free over $\Q[\uq,\us,\ut]$ of rank $16$ with the same basis as that given in Lemma~\ref{Sec6:Fun--FirstLemma}.1. The same holds if we replace $s_4^\fun$ by $s_4^{\prime,\fun}$.
\item The ring $R_2=\Q[\uq,\us,\ut,,a,b,c,X,\alpha,\beta,\gamma]/((b-\us-c, \beta-\ut-c, \gamma-X)+\cI^\fun_\Z)$ is free over $\Q[\uq,\us,\ut]$ of rank $6$ with the same basis as that given in Lemma~\ref{Sec6:Fun--FirstLemma}.2.
\item The kernel of the surjective ring homomorphism $R_1\to R_2$ is free over $\Q[\uq,\us,\ut]$ of rank $10$.
\end{enumerate}\end{lemma}
Thus we can compute $C$ as a matrix with entries in $\Q[\uq,\us,\ut]$, i.e., before specialization.  For this we computed new basis elements $b_7,\ldots,b_{16}$ that span  $\kernel (R_1\to R_2)$. To our surprise, we found $\det C=1$, and inverting $C$ posed no problem. This allowed us to compute the tuples of $\mu_i$ and then the $\Ro$-generator $\Theta^{-1}(b_6^*)$ of $\Ro[\Io]$. Under our augmentation, Macaulay evaluated it to $(\us+\ut) \ut$ in $\Q[\uq,\us,\ut]$. This shows:

\begin{corollary}
$\overline\lambda( \Ro[\Io])= (s+t)t \subset \cO$, or  $\overline\lambda( \Ro[\Io])= st \subset \cO$ if $s_4^\fun$ is replaced by $s_4^{\prime,\fun}$.
\end{corollary}

The next steps are the computation of $\lambdat( \Fitt_0^{\Rt}(I) )$ and of $\Hom_{R_v^\fun}(I/I^2,E/\cO)$. For this we proceed essentially as in the Steinberg case, cf.~Corollaries~\ref{Cor:LambdaOfFitt0} and~\ref{Cor:I/Isquared}, except that we rely on Macaulay. Namely, we compute the first two steps of a resolution of $\cI_\Z^\fun$, considered as an ideal of $\cR_\Z[\us,\ut]=\Z[ \uq,\us,\ut,a,b,c,X,\alpha,\beta,\gamma]$. This results in a right exact sequence
\[  \cR_\Z[\us,\ut]^{26}\stackrel{A}\longrightarrow  \cR_\Z[\us,\ut]^{9}\stackrel{}\longrightarrow  \cI_\Z^\fun \longrightarrow0,\]
for some matrix $A$ in $M_{9\times 26}(\cR_\Z[\us,\ut])$ (with rather simple entries). We tensor the sequence over $\cR_\Z[\uq,\us]$ with $R_3=\cR_\Z[\uq,\us]/\widetilde\cI$. Now observe that over $R_3$, the ideal $I_3=\cI_\Z^\fun \otimes_{\cR_\Z[\uq,\us]} R_3$ is generated by the elements $r_1^\fun,r_2^\fun,r_3^\fun,r_6^\fun,r_7^\fun$; because these $5$ elements together with our generators of $\widetilde\cI$ generate $\cI_\Z^\fun$. So we extract a matrix $A'\in M_{5\times 26}(R_3)$, from the specialization of $A$ under $\cR_\Z[\us,\ut]\to R_3$, that gives a short exact sequence
\[  R_3^{26}\stackrel{A'}\longrightarrow  R_3^{5}\stackrel{}\longrightarrow  I_3  \longrightarrow0,\]
Specializing under $R_3\to \Z[\uq,\us,\ut]$ via $a,c,\alpha,\gamma,X\mapsto0$, $b\mapsto \us$, $\beta\mapsto \ut$, and computing the ideal of the resulting $5\times 5$-minors gives the ideal 
$(\us,\ut,\uq)^3\cdot ( \us+\ut) \ut $. If we work with $\widetilde\cI'$ in place of $\widetilde\cI$, the answer is $ (\us,\ut,\uq)^3\cdot \us\ut  $. Continuing with the natural map $\Z[\uq,\us,\ut]\to\cO$, and observing the computations in Corollaries~\ref{Cor:LambdaOfFitt0} and~\ref{Cor:I/Isquared}, we find:
\begin{corollary}\label{Cor:LambdaOfFitt0-Un}
We have $\#\cO/\lambdat( \Fitt_0^{\Rt}(I) )=
\# \Hom_{R_v^\fun}(I/I^2,E/\cO)$ and the number is equal to
\[\#\cO/((s+t) t (s,t,q-1)^3),\]
or to  $\#\cO/(st (s,t,q-1)^3)$, if we work with $\widetilde\cI'$ in place of $\widetilde\cI$.
\end{corollary}
Finally, we had Macaulay work out the analog of Lemma~\ref{Lem:LatticeIndexR'Rst} to determine the lattice $\Lambda^\fun$, which, as to be expected, is rather easy. Following the proof of Corollary~\ref{Lem:D2-RqSt}, one finds.
\begin{corollary}\label{Lem:D2-RqSt-Un}
We have 
\[\#\kernel(\Hom_{R_v^\fun}(I/I^2,E/\cO) \to \cDer^1_\cO(R_v^\fun,E/\cO))=\# \Lambda^{\fun}/\Lambdat=\#(s,t,q-1)^3/((s+t)t) \]
and,  if we work with $\widetilde\cI'$ in place of $\widetilde\cI$, the cardinality is $\#(s,t,q-1)^3/(st)$.
\end{corollary}
As in the Steinberg case, the following result is now an immediate consequence. It is independent of whether we use $\widetilde\cI$ or $\widetilde\cI'$.
\begin{theorem}\label{thm:delta-PhiUn}
Let $e$ be the ramification index of $E$ over $\Q_l$. Then we have
\begin{enumerate}
\item $D_{1,\lambda}(R_v^{\fun})=6\frac{n_v}e$.
\item $c_{1,\lambda}(R_v^{\fun})=3\frac{n_v}e$.
\item $\delta_\lambda(R_v^{\fun}) = 3\frac{n_v}e$.
\end{enumerate}
\end{theorem}

\subsection{Unipotent deformations}
\label{Subsec-Unip}

This case we handled in basically the same way as the previous one. Again we made use of Macaulay to compute various intermediate steps. This was more challenging, since we did not find a good complete intersection that would surject onto the Gorenstein ring that we were interested in over $\Z$, but only over $\Z[\frac12]$, and we could not apply certain Macaulay functionalities directly. We only indicate outcomes of some intermediate steps, but give no further details:

Set 
$s_1^\un= -r_1^\un+2r_2^\un +r_4^\un -2r_5^\un$,  
$s_2^\un=   r_7^\un-r_1^\un$,
$s_3^\un= r_8^\un-r_1^\un$, 
$s_4^\un= r_9^\un - r_1^\un$ and $\widetilde\cI=(s_1^\un,s_2^\un,s_3^\un,s_4^\un)$.
\begin{lemma}\label{Sec6:Un--FirstLemma}
\begin{enumerate}
\item The ring $\Z[\frac12][\uq,\us,\ut,a,b,c,X,\alpha,\beta,\gamma]/(\uq,\us,\ut,
b-\us-\beta+\ut,a+X-\gamma,b-\us-c
, s_i^\un, i=1,\ldots,4)$ is free over $\Z[\frac12]$ of rank $14$. A basis is $1, a, a\alpha, b, bX, b\alpha, X, X^2, X\alpha, \alpha, \alpha^2,\alpha^3$. A basis of the socle of the ring modulo any prime is $4\alpha^3+16a^2+54ab-30X\alpha+133b\alpha-19\alpha^2+111X$.
\item The ring $\Z[\frac12][\uq,\us,\ut,a,b,c,X,\alpha,\beta,\gamma]/((\uq,\us,\ut,
b-\us-\beta+\ut,a+X-\gamma,b-\us-c
)+\cI^\un_\Z)$ is free over $\Z[\frac12]$ of rank $5$. A basis is $1,a,b,X,\alpha$. A basis of the socle of the ring modulo any prime is $X$.
\item Write $x_1,\ldots,x_7$ for $a,b,c,X,\alpha,\beta,\gamma$. Then the ideal in $\Z[\uq,\us,\ut]$ generated by the $4\times 4$-minors of the Jacobian $(\partial s_i^\un/\partial x_j)_{i=1,\ldots,4; j=1,\ldots,7}$ evaluated at $(x_1,\ldots,x_7)=(0,\us,0,0,0,\ut,0)$ is $2\uq\ut(\uq,\us, \ut)$. 
\end{enumerate}\end{lemma}
\begin{remark}
The optimal rank over $\Z[\frac12]$ of a complete intersection cover of the Gorenstein ring we are interested in would be $8$; cf.~Remark~\ref{Rem:OptimalFun}. We could not find such a ring that would make our computations work, uniformly for all maximal ideals in $\Spec\Z[\frac12]$.
\end{remark}
Let $s,t\in\ffrm$ with $t\neq0$.
\begin{corollary}\label{Sec6:UnFirstCor}
\begin{enumerate}
\item The ring $\Rt=\cO[[a,b,c,X,\alpha,\beta,\gamma]]/(s_i^\un, i=1,\ldots,4)$ is a complete intersection, flat over $\cO$ and of relative dimension $3$. One has a natural surjection $\Rt\to R^\un_v$ induced from $(s_i^\un, i=1,\ldots,4)\subset (r_j^\un,j=1,\ldots,9)$.
\item Via the ring map $S=\cO[[y_1,y_2,y_3]]\to \Rt$ given by $y_1\mapsto b-\us-\beta+\ut$, $y_2\mapsto a+X-\gamma$, $y_3\mapsto b-\us-c$, the rings $\Rt$ and $R^\un_v$ are free $S$-modules of rank $12$ and $5$, respectively.
\item The augmentation $\lambdat\colon\Rt\to\cO$ given by $a,c,X,\alpha,\gamma\mapsto 0$, $b\mapsto s$ and $\beta\mapsto t$ defines a formally smooth point of $\Spec \Rt[\frac1\varpi]$. 
\end{enumerate}
\end{corollary}
\begin{lemma}\label{Sec6:Un--SecondLemma}
\begin{enumerate}
\item The ring $R_1=\Q[\uq,\us,\ut,a,b,c,X,\alpha,\beta,\gamma]/(b-\us-\beta+\ut,a+X-\gamma,b-\us-c , s_i^\un, i=1,\ldots,4)$ is free over $\Q[\uq,\us,\ut]$ of rank $12$ with the same basis as that given in Lemma~\ref{Sec6:Un--FirstLemma}.1. 
\item The ring $R_2=\Q[\uq,\us,\ut,,a,b,c,X,\alpha,\beta,\gamma]/((b-\us-\beta+\ut,a+X-\gamma,b-\us-c
)+\cI^\un_\Z)$ is free over $\Q[\uq,\us,\ut]$ of rank $5$ with the same basis as that given in Lemma~\ref{Sec6:Un--FirstLemma}.2.
\item The kernel of the surjective ring homomorphism $R_1\to R_2$ is free over $\Q[\uq,\us,\ut]$ of rank $7$.
\end{enumerate}\end{lemma}
\begin{proposition}
We have
\begin{enumerate}
\item $\lambdat( \Rt[\widetilde{I}])= (q-1)$.
\item $\#\cO/\lambdat( \Fitt_0^{\Rt}(I) )= \# \Hom_{R_v^\un}(I/I^2,E/\cO)=\#\cO/((q-1)\gcd(s,t,q-1))$.
\item $\#\kernel(\Hom_{R_v^\un}(I/I^2,E/\cO) \to \cDer^1_\cO(R_v^\un,E/\cO))=\# \Lambda^{\un}/\Lambdat=\#(\gcd(s,t,q-1) /(q-1))$.
\end{enumerate}
\end{proposition}

\begin{theorem}\label{thm:delta-Un}
Let $e$ be the ramification index of $E$ over $\Q_l$. We have
\begin{enumerate}
\item $D_{1,\lambda}(R_v^{^\un})=2\frac{n_v}e$.
\item $c_{1,\lambda}(R_v^{^\un})=\frac{n_v}e$.
\item $\delta_\lambda(R_v^{^\un}) = \frac{n_v}e$.
\end{enumerate}
\end{theorem}

 \subsection{Recollections about Cohen--Macaulay and Gorenstein  rings}\label{Subsect-OnCM}

Let $R$ be a Noetherian local ring with maximal ideal $\ffrm$ and residue field $k$. In this subsection we want to briefly recall some results on Cohen--Macaulay and Gorenstein rings that occur repeatedly in our arguments or, more importantly, in our computations. We also present a result on generating sets of dual modules that was useful in our computations. For basic notions such as $\depth$, $R$-sequence, Cohen--Macaulay and Gorenstein rings, we refer to  \cite[\S\S~1.2, 2.1, 3.1]{BH},

\begin{definition}
The socle of $R$ is defined as  $\socle R=R[\ffrm]=\{x\in R\mid \ffrm x =0\}$.
\end{definition}

\begin{proposition}[{\cite[Thm.~17.4 and p.~136]{Matsumura-CA}, \cite[2.1.3, 2.1.8, 3.1.19]{BH}}]
\label{prop:matsum-GorCM}
\begin{enumerate}
\item Any local Artin ring $R$ is Cohen--Macaulay. It is Gorenstein if in addition it satisfies $\socle R\cong k$.
\item 
If $R$ is Noetherian local, and if $(x_1\ldots,x_n)$ is an $R$-sequence in $\ffrm$, then $R$ is Cohen--Macaulay or Gorenstein, respectively, if and only if $R/(x_1,\ldots,x_n)$ has this property. In particular, if $R/(x_1\ldots,x_n)$ is Artinian, then $R$ is Cohen--Macaulay, and if moreover $\socle R/(x_1\ldots,x_n)\cong  k$, then $R$ is Gorenstein. 
\item If $R$ is a local Cohen--Macaulay ring, then any system of parameters is a regular $R$-sequence.
\end{enumerate}
\end{proposition}

Let now $(A,\ffrm)$ be a local Artin ring. In this case $I\cap\socle A\supsetneq0$  for any non-zero ideal $I$ of $A$: To see this consider $n\in\Z_{\ge0}$ such that $\ffrm^{n-1}I\neq0$ and $\ffrm^nI=0$. Then $\ffrm^{n -1}I\subset I\cap \socle A$. 

\begin{proposition}\label{lem:DualBasis}
Let $(\bar\psi_i)_{i\in B}$ be a finite tuple in $\Hom_k(A,k)$, such that $(\bar\psi_i)_{i\in B}\colon \socle(A)\to k^{B}$ is injective. Then $(\bar\psi_i)_{i\in B}$ is a set of generators of $\Hom_k(A,k)$ as an $A$-module. In particular, if $A$ is Gorenstein and if $B=\{0\}$ is a singleton, then $\bar\psi_0$ is an $A$-basis of $\Hom_k(A,k)$.
\end{proposition}

The proof relies on the following result from linear algebra.
\begin{lemma}\label{Lem:LinAlg}
Let $V$ be a finite-dimensional $k$-vector space. Let $(V_j)_{j\in J}$ be a finite tuple of sub vector spaces such that $\bigcap_{j\in J} V_j=0$. Then for any $\bar\psi\in\Hom_k(V,k)$, there exist $\bar\psi_j\in \Hom_k(V,k)$ with $V_j\subset \kernel \bar\psi_j$ for $j\in J$, such that $\bar\psi=\sum_{j\in J} \bar\psi_j$.
\end{lemma}
\begin{proof}
We may assume $J=\{1,\ldots,t\}$ for some $t\in\Z_{\ge1}$. We induct over $t$, noting that the case $t=1$ is trivial, since then $V_1=0$. For the induction step suppose $t\ge2$, and let $W=\bigcap_{j=2}^t V_j$. Then $V_1\cap W=0$, and so we can choose a basis for $W$ and one for $V_1$ and then extend the one for $V_1$ to a complementary basis to that of $W$. Then one can find $\bar\psi_1$ and $\bar\phi$ in $\Hom_k(V,k)$ such that $\kernel \bar\psi_1\supseteq V_1$ and $\kernel\bar\phi\supseteq W$, and $\bar\psi=\bar\psi_1+\bar\phi$. Now apply the induction hypothesis to $V/V_1$ and $(V_j/V_1)_{j=2,\ldots,t}$ and $\bar\psi_1$ considered as a map in $\Hom_k(V/V_1,k)$.
\end{proof}
\begin{proof}[Proof of Lemma~\ref{lem:DualBasis}]
Let $N=\sum_{i\in B} A\bar\psi_i$. We shall show that $\Hom_k(A,k)\subseteq N+\ffrm \Hom_k(A,k)$. Then the lemma will follow from Nakayama's Lemma.

Let $\bar\psi$ be in $\Hom_k(A,k)$. By our hypothesis, there is a $k$-linear map $\alpha\colon k^{B}\to k$ such that the restriction $\bar\psi|_{\socle A}$ agree with $\alpha\circ (\bar\psi_i)_{i\in B}$. In other words, the map 
\[\bar\phi:=\bar\psi- \sum_{i\in B} \alpha(\bar e_i)) \bar\psi_i\]
vanishes on $\socle A$.

Next let $x_1,\ldots,x_t$ be a set of $A$-module generators of $\ffrm$, and let $V_i=\{r\in A\mid x_ir=0\}$. Then 
\[\socle A = \bigcap_{i=1,\ldots,t} V_i\]
By Lemma~\ref{Lem:LinAlg} applied to $A/\socle A$, there exist $\bar\phi_i\in \Hom_k(A,k)$ with $\kernel\bar\phi_i\supset V_i$, and $\bar\phi=\sum_{i\in B}\bar\phi_i$.

Now consider the short exact sequence $0\to V_i \to A \stackrel{x_i\cdot }\to x_iA \to 0$. Then the $\bar\psi_i$ induce $k$-linear maps $x_i A\to k$. The latter can be extended to $k$-linear maps $\bar\xi_i\colon A\to k$ under $x_iA\subset A$. In other words $\bar\phi_i=x_i \bar\xi_i$, and this gives
\[\bar\psi- \sum_{i\in B} \bar\psi(e_i)\bar\psi_i = \sum_{j=1,\ldots,t} x_j\bar\xi_j ,\]  
proving the claim from the first line, and hence the lemma.
\end{proof}

Let now $(R,\ffrm)$ be a local complete Noetherian Cohen--Macaulay ring that is an $\cO$-algebra, and suppose that $\mathbf{r}=(\varpi, r_1,\ldots,r_n)$ is a 
system of parameters. Let $(\bar e_i)_{i\in B}$ be a $k$-basis of $A=R/\mathbf{r} R$, let $(e_i)_{i\in B}$ be a tuple of preimages in $R$, and consider the $\cO$-algebra homomorphism $S=\cO[[x_1,\ldots,x_n]]\to R, x_i\mapsto e_i$.

\begin{lemma}
As an $S$-module, $R$ is free with basis $(e_i)_{i=1,\ldots,n}$.
\end{lemma}
\begin{proof}
The ring $S$ is regular local and thus of finite global dimension. Hence $R$ has finite projective dimension over $S$. By Nakayama's Lemma $R$ is also finitely generated as an $S$-module, because $\dim_k R/\mathbf{r}R$ is finite for the system of parameters~$\mathbf {r}$. The sequence $\mathbf{r}$ is in fact regular as $R$ is Cohen--Macaulay. It follows that $\depth_SR=1+n=\dim S$, so that by the Auslander-Buchsbaum theorem $R$ is a finite free $S$-module. One finds that $\psi\colon S^B\to R, (s_i)_{i\in B}\mapsto \sum_i s_i e_i$ is an isomorphism, because $S$ is local and $\psi\mod{\mathbf{r}}$ is bijective.
\end{proof}
The following result gives a generating set (or a basis) over $R$ of the free $S$-module $\Hom_S(R,S)$.
\begin{proposition}\label{prop:GensOfDual}
Let $\psi_i\in \Hom_S(R,S)$, $i\in B$, be a tuple of elements such that the elements $\bar\psi_i:=\psi_i\otimes_RA\colon A\to k$ satisfy the condition of Lemma~\ref{lem:DualBasis}. Then $(\psi_i)_{i\in B}$ is a set of $R$-module generators of $\Hom_S(R,S)$. If moreover $R$ is Gorenstein and $B=\{0\}$ is a singleton, then $\psi_0$ is an $R$-basis of $\Hom_S(R,S)$.
\end{proposition}
\begin{proof}
This is immediate from Nakayama's lemma and Lemma~\ref{lem:DualBasis}.
\end{proof}

\section{Wiles  defect of Hecke algebras and global deformation rings}\label{sec:deformation}

In this section, we'll describe how the commutative algebra results from Sections~\ref{sec_deformation_theory} and~\ref{sec:ComputeWD} can be applied to Galois deformation rings, in the setup of Taylor--Wiles--Kisin patching. For ease of exposition we'll restrict our attention to the case of two dimensional Galois representations over a totally real number field, and moreover ones that are modular of parallel weight $2$, as all of the computations and applications we give in this paper will be concerned with this case. This is not a fundamental limitation on our methods, and indeed everything we describe in this section will generalize automatically to any ``$\ell_0=0$'' patching setup (such as the definite unitary groups considered by \cite{CHT} and others).

Let $F$ be a totally real number field. Fix a finite set $\Sigma$ of finite places of $F$. For each $v\in \Sigma$, fix a $\tau_v\in \{\min,\st,\un,\fun,\square\}$, let $\tau = (\tau_v)_{v\in \Sigma}$, and for $\sigma\in \{\min,\st,\un,\fun,\square\}$ write $\Sigma^\sigma = \{v\in\Sigma|\tau_v=\sigma\}$.

Pick a prime $p>2$ which is not ramified in $F$ and is not divisible by any prime in $\Sigma$. Let $E/\Q_p$ be a finite extension with ring of integers $\cO$, uniformizer $\varpi$ and residue field $k$. Let $\varepsilon_p:G_F\to \cO^\times$ be the cyclotomic character. Let $\rho:G_F\to \GL_2(\cO)$ be a Galois representation for which:
\begin{itemize}
	\item $\rho$ corresponds to a Hilbert modular form of parallel weight 2;
	\item $\det \rho = \varepsilon_p$; 
	\item For every $v\not\in\Sigma$ and $v\nmid p$, $\rho$ is unramified at $v$;
	\item For every $v|p$, $\rhobar|_{G_v}$ is finite flat;
	\item If $v\in \Sigma^{\min}$, then either $|\cO/v|\not\equiv -1\pmod{\ell}$, $\rhobar|_{I_v}$ is irreducible or $\rhobar|_{G_v}$ is absolutely reducible; 
	\item If $v\in \Sigma^{\st}\cup \Sigma^{\un}\cup \Sigma^{\fun}$, then $\rho|_{G_v}$ is Steinberg (i.e. $\rho|_{G_v}\sim \begin{pmatrix}	\chi\,\varepsilon_p&*\\	0 & \chi\end{pmatrix}$ for some unramified quadratic character);
	\item The residual representation $\rhobar:G_F\to \GL_2(k)$ is absolutely irreducible, and moreover that it satisfies the \emph{Taylor--Wiles conditions:} $\rhobar|_{G_{F(\zeta_p)}}$ is still absolutely irreducible; and in the case when $p=5$, $\sqrt{5}\in F$ and the projective image $\proj \rhobar:G_{F}\to \PGL_2(\Fbar_5)$ is isomorphic to $\PGL_2(\bbF_5)$, that $\ker \proj\rhobar\not\subseteq G_{F(\zeta_5)}$.
\end{itemize}

Let $Q = \Sigma^{\st}$, and let $D$ be a quaternion algebra over $F$ ramified at the primes in $Q$ (and no other finite primes) and at either all, or all but one infinite place of $F$ (depending on whether $|Q|+[F:\Q]$ is even or odd). Define a compact open subgroup $K^\tau = \prod_v K_v^\tau \subset (D \otimes \A_{F,f})^\times$ by:
\begin{itemize}
	\item $K_v^\tau = \GL_2(\cO_{F,v})$ if $v\not\in \Sigma$;
	\item $K_v^\tau$ is a maximal compact subgroup of $(D\otimes F_v)^\times$ if $v\in \Sigma^{\st} = Q$;
	\item $K_v^\tau = U_0(v)$ if $v\in \Sigma^{\un}\cup \Sigma^{\fun}$;
	\item $K_v^\tau = U_0(v^{a_v})$ if $v\in \Sigma^{\min}$, where $a_v$ is the Artin conductor of $\rhobar|_{G_v}$;
	\item $K_v^\tau = U_0(v^{a_v+2})$ if $v\in \Sigma^{\square}$.
\end{itemize}
For convenience, we will simply write $K = K^\tau$ and $K_v = K_v^\tau$.

When $D$ is ramified at all but one infinite places (resp. all infinite places) let $X_K$ be the Shimura curve (resp. Shimura set) associated to $K$. Let $\T^D(K)$ be the Hecke algebra acting on $H^1(X_K,\cO)$ in the Shimura curve case and on $H^0(X_K,\cO)$ in the Shimura set case, generated (as an $\cO$-algebra) by the Hecke operators $T_v$ and $S_v$ for all finite primes $v\not\in \Sigma$, and let $\Tbar^{D}(K) = \T^D(K)[U_v|v\in\Sigma^{\fun}]$. Note that $\T^D(K)$ and $\Tbar^D(K)$ are finite $\cO$-algebras.

Let $\T^D(K)^\varepsilon = \T^D(K)/\left(S_v-\varepsilon_p(\Frob_v)\middle|v\not\in \Sigma\right)$ and $\Tbar^D(K)^\varepsilon = \Tbar^D(K)/\left(S_v-\varepsilon_p(\Frob_v)\middle|v\not\in \Sigma\right)$ be the fixed determinant Hecke algebras.

The assumption that $\rho$ corresponds to a Hilbert modular form of parallel weight $2$ gives the following:

\begin{proposition}
	There is an augmentation $\lambda:\Tbar^D(K)^\varepsilon\onto \cO$ with the property that for any $v\not\in \Sigma\cup\Sigma_p$, $\rho(\Frob_v)$ has characteristic polynomial $x^2-\lambda(T_v)x+\lambda(S_v)$. Moreover, $\Phi_{\lambda}(\Tbar^D(K)^\varepsilon)$ is finite.
\end{proposition}

Let $\ffrm = \lambda^{-1}(\varpi\cO)\subseteq \Tbar^D(K)^\varepsilon$ be the maximal ideal of $\Tbar^D(K)^\varepsilon$ corresponding to $\rhobar$. By slight abuse of notation, also write $\ffrm = \ffrm\cap \T^D(K)$ for the maximal ideal of $\T^D(K)$ corresponding to $\rhobar$.

Write $\T^\tau = \T^D(K)^\varepsilon_{\ffrm}$ and $\Tbar^\tau = \Tbar^D(K)^\varepsilon_{\ffrm}$ for the localizations at $\ffrm$ (and note that we are suppressing $\varepsilon$ from our notation). 

Note that any $x:\Tbar^\tau\to \barqp$ corresponds to a Galois representation $\rho_x:G_F\to \GL_2(\barqp)$ lifting $\rhobar$ with $\det\rho_x = \varepsilon_p = \det \rho$ and $\tr\rho_x(\Frob_v) = x(T_v)$ for all $v\not\in \Sigma$ (so that $\rho = \rho_{\lambda}$).

Define $H^\tau = H^1(X_K,\cO)^*$ if $D$ is indefinite and $H^\tau = H^0(X_K,\cO)^*$ if $D$ is definite (where for any $\cO$-module $M$, $M^* = \Hom_{\cO}(M,\cO)$), viewed as a $\Tbar^D(K)$-module, and hence as a $\T^D(K)$-module. Define
\[
M^\tau = \Tbar^\tau\otimes_{\Tbar^D(K)} H^\tau = \T^\tau\otimes_{\T^D(K)} H^\tau = H^\tau/\left((S_v-\varepsilon_p(\Frob_v))x\middle|v\not\in \Sigma,x\in H^\tau\right)
\]

For the convenience of the reader, we recall some notation and results from Sections~\ref{sec_deformation_theory} and~\ref{sec:ComputeWD}. For each prime $v$ of $F$, the universal (fixed determinant) ring, parameterizing framed deformations of $\rhobar|_{G_{F_v}}$ with determinant $\varepsilon_p$ is $R_v^{\square}$. For $v\nmid p$ and $\tau_v \in \{\min,\st,\un,\fun,\square\}$, let $R_v^{\tau_v}$ be the deformation ring defined in Section \ref{sec:computations}, provided it exists (which is does for $v\in\Sigma$ and $\tau = \tau_v$, by assumption). The ring $R_v^{\tau_v}$ is naturally an $R_v^\square$-algebra, and unless $\tau_v = \fun$ it is a quotient of $R_v^{\square}$. Summarizing the results of Proposition \ref{prop:R_v^tau} we~have:

\begin{proposition}\label{prop:R_v^tau recap}
	For each $v\in\Sigma$, the ring $R_v^{\tau_v}$ is a complete, Noetherian $\cO$-algebra which is flat and equidimensional over $\cO$ of relative dimension $3$. Moreover, $R_v^{\tau_v}$ is Cohen--Macaulay, and is a complete intersection whenever $\tau_v=\min$ or $\square$, or whenever $\rhobar|_{G_v}$ is not a scalar.
\end{proposition}

For $v| p$ the ring $R_v^{\fl}$ is the quotient of $R_v^{\square}$ parameterizing \emph{flat} framed deformations of $\rhobar|_{G_v}$ with determinant $\varepsilon_p$. Then as $\rhobar|_{G_v}$ is flat and $F_v/\Q_p$ is unramified, Fontaine--Laffaille theory implies that $R_v^{\fl}\cong \cO[[x_1,\ldots,x_{3+[F_v:\Q_p]}]]$.

Now write
\begin{align*}
R_{\loc} &= \left(\cbigotimes{v\in \Sigma}\,R_v^{\square}\right)\cotimes\left( \cbigotimes{v|p}\,R_v^{\fl}\right),&
&\text{and}&
R_{\loc}^\tau &= \left(\cbigotimes{v\in \Sigma}\,R_v^{\tau_v}\right)\cotimes\left( \cbigotimes{v|p}\,R_v^{\fl}\right)
\end{align*}
so that $R^\tau_{\loc}$ is naturally a $R_{\loc}$-algebra. By Propositions \ref{prop:R_v^tau} and \ref{prop:R_v-Loc}, $R_{\loc}$ is flat over $\cO$ and Cohen--Macaulay.

By $R$ (resp. $R^\square$) we denote the (global) unframed (resp. framed) deformation ring parameterizing lifts of $\rhobar$ with determinant $\varepsilon_p$ which are flat at every prime $v|p$. One may noncanonically fix an isomorphism $R^\square = R[[X_1,\ldots,X_{4j-1}]]$ for some $j$, and thereby treat $R$ as a quotient of $R^\square$. Using the natural map $R_{\loc}\to R^\square$ (and $R_{\loc}\to R$), one defines $R^{\square,\tau} = R_{\loc}^\tau\otimes_{R_{\loc}}R^\square$ and  $R^{\tau} = R_{\loc}^\tau\otimes_{R_{\loc}}R$.

\begin{lemma}\label{lem:R->T}
	There is a surjective map $R^\tau\onto \Tbar^\tau$ inducing a representation $\rho^\tau:G_F\to \GL_2(R^\tau)\onto \GL_2(\T^\tau)$, such that for all $v\not\in \Sigma\cup\Sigma_p$, $\rho^\tau(\Frob_v)$ has characteristic polynomial $t^2-T_vt+S_v$, and for all $v\in \Sigma^{\un}\cup\Sigma^{\fun}$, $\rho^\tau|_{G_{F_v}}$ is unipotent and if $\Frob_v\in G_{F_v}$ is any lift of Frobenius then $\rho^\tau(\Frob_v)$ again has characteristic polynomial $t^2-T_vt+S_v$.
\end{lemma}
\begin{proof}
	If $\Sigma^{\fun} = \es$, this is just \cite[Lemma 2.4]{Manning}.
	
	In general, for each for each $v\in \Sigma$, set $\sigma_v = \tau_v$ if $\tau_v\in\{\min,\st,\uni,\square\}$ and $\sigma_v = \uni$ if $\tau_v = \fun$. Note that under this definition, $K^\sigma = K^\tau = K$ and $\Tbar^\sigma  = \T^\sigma  = \T^\tau$.
	
	It follows that there is a surjection $R^\sigma\onto \Tbar^\sigma = \T^\sigma = \T^\tau$ satisfying the desired conditions on $\rho^\sigma$. By definition, $\Tbar^\tau = \T^\tau[U_v|v\in\Sigma^{\fun}]$. From the identity $U_v^2-T_vU_v+S_v=0$ in $\Tbar^D(K)$ and the definition of modified global deformation rings given in Section \ref{sec_deformation_theory}, it follows that $R^\sigma\onto \T^\tau\to \Tbar^\tau$ induces a map $R^\tau\onto \Tbar^\tau$ sending $\alpha_v$ to $U_v$ for $v\in\Sigma^{\fun}$, which is therefore surjective, and hence is the desired map.
\end{proof}

Now similarly to \cite[Theorem 6.3]{BKM}, the Taylor--Wiles patching method gives the following:

\begin{theorem}\label{thm:patching}
	There exist integers $g,d\ge 0$ and rings
	\begin{align*}
	R_\infty^{\tau} &= R_{\loc}^{\tau}[[x_1,\ldots,x_g]]\\
	S_\infty &= \cO[[y_1,\ldots,y_d]]
	\end{align*}
	satisfying the following:
	\begin{enumerate}
		\item $\dim S_\infty = \dim R_\infty^{\tau}$.
		\item There exists a continuous $\cO$-algebra morphism $i:S_\infty\to R_\infty^\tau$ making $R_\infty^\tau$ into a finite free $S_\infty$-module.
		\item There is an isomorphism $R_\infty^{\tau}\otimes_{S_\infty} \cO \cong R^{\tau}$ and $R^\tau$ is finite free over $\cO$.
		\item The map $R^\tau\onto \Tbar^\tau$ from Lemma \ref{lem:R->T} is an isomorphism. These rings are reduced if $\Sigma^{\fun} = \es$.
		\item If $\lambda$ is the induced map $R_\infty^\tau\onto R^\tau\xrightarrow{\sim} \Tbar^\tau\xrightarrow{\lambda}\cO$, then $\Spec R_\infty^\tau[1/\varpi]$ is formally smooth at the point corresponding to $\lambda$.
	\end{enumerate}
\end{theorem}
\begin{proof}
	This is proved similarly to Theorem 6.3 in \cite{BKM}.
	
	First, we will consider the case when $\Sigma^{\fun} = \es$, and so $\Tbar^D(K) = \T^D(K)$. More precisely, as in the proof of Lemma \ref{lem:R->T}, for each $v\in \Sigma$, define $\sigma_v = \tau_v$ if $\tau_v\in\{\min,\st,\uni,\square\}$ and define $\sigma_v = \uni$ if $\tau_v = \fun$. Note that under this definition, $K^\sigma = K^\tau = K$, $M^\sigma = M^\tau$ and $\Tbar^\sigma  = \T^\sigma  = \T^\tau$.
	
	By assumption, $\rhobar$ satisfies the Taylor--Wiles conditions, and so we may apply the Taylor--Wiles--Kisin patching method (as summarized in \cite[Section 4]{Manning}) to the rings $R^\sigma$ and $\T^\sigma$ and the module $M^\sigma$.
	
	First, as in \cite[Section 4.2]{Manning}, we may add auxiliary level structure at a carefully chosen prime not in $\Sigma$ to remove any isotropy issues, without affecting any of the objects considered considered in this theorem. 
	
	Now, exactly as in the proof of \cite[Theorem 6.3]{BKM} (and the method outlined in \cite[Section 4.3]{Manning}), there exist integers $g,d\ge 0$, satisfying $d+1 = \dim R_{\loc}+g = \dim R_{\loc}^{\sigma}+g$ (see \cite[Lemma 2.5]{Manning} and \cite[Proposition (3.2.5)]{Kisin}) such that for each $n\ge 1$, there is a unframed global deformation ring $R_n^\sigma$ and a framed global deformation ring $R_n^{\sigma,\square}$ (with fixed determinant, the same deformation conditions as $R^\sigma$ at each $v\in \Sigma$, and relaxed deformation conditions at a carefully selected set $Q_n$ of ``Taylor--Wiles'' primes), such that $R_n^{\sigma,\square}$ has the structure of a $S_\infty$-algebra and there is a surjective map $R_\infty\onto R_n^{\sigma,\square}$ and an isomorphism $R_n^{\sigma,\square} \otimes_{S_\infty} \cO \cong R^\sigma$, where $S_\infty$ and $R_\infty$ are as in the theorem statement.
	
	Moreover, for each $n\ge 1$ the construction in \cite[Section 4.2]{Manning} also constructs a compact open subgroup $K_n = \prod_v K_{n,v}\subseteq (D\otimes \A_{F,f})^\times$ (with $K_{n,v} = K_v$ for all $v\not\in Q_n$), and a Hecke algebra $\T^\sigma_n$ and Hecke module $M^\sigma_n$ at level $K_n$ (defined analogously to $\T^\sigma$ and $M^\sigma$ above, by localizing at a particular maximal ideal, and fixing determinants by taking a quotient). One then has a surjection $R^\sigma_n\onto \T^\sigma_n$, making $M^\sigma_n$ into a $R^\sigma_n$-module. Using this surjection, we may define framed versions of these objects: $\T^{\sigma,\square}_n= \T^\sigma_n\otimes_{R^\sigma_n}R^{\sigma,\square}_n$ and $M^{\sigma,\square}_n= M^\sigma_n\otimes_{R^\sigma_n}R^{\sigma,\square}_n$.
	
	Applying the `ultrapatching' construction described in \cite[Section 4.1]{Manning} (as well as in the proof of Lemma 4.8) then produces an $S_\infty$-algebra $\cR_\infty^\sigma$ as well as an $\cR_\infty$-module $M_\infty^\sigma$ (which would be called $\mathscr{P}(\{R^{\sigma,\square}_n\})$  and $\mathscr{P}(\{M^{\sigma,\square}_n\})$ in the notation of that paper), for which:
	\begin{itemize}
		\item $M_\infty^\sigma$ is finite free over $S_\infty$;
		\item $\cR_\infty^\sigma\otimes_{S_\infty}\cO \cong R^\sigma$ and $M_\infty^\sigma\otimes_{S_\infty}\cO\cong M^\sigma$;
		\item There is a surjection $R_\infty^\sigma\onto \cR_\infty^\sigma$ such that the composition
		\[R_{\loc}^\sigma\into R_\infty^\sigma\onto \cR_\infty^\sigma\onto R^\sigma\]
		is the map $R_{\loc}^\sigma\to R^\sigma$ from above.
	\end{itemize}
	
	Just as in the proof of \cite[Theorem 6.3]{BKM}, we may lift the structure map $S_\infty\to \cR_\infty^\sigma$ to a map $i:S_\infty\to R_\infty^\sigma$ making $\pi_\infty:R_\infty^\sigma \to \cR^\sigma_\infty$ into an $S_\infty$-module surjection, and so it follows that $M_\infty^\sigma$ is a maximal Cohen--Maculay $R_\infty^\sigma$-module.
	
	But now by standard properties of maximal Cohen--Macaulay modules, the support of $M_\infty^{\sigma}$ is a union of irreducible components of $\Spec R_\infty^\sigma$. As $R_\infty^\sigma = R_{\loc}^\sigma[[x_1,\ldots,x_g]]$, the irreducible components of $\Spec R_\infty^\sigma$ are in bijection with those of $\Spec R_{\loc}^\sigma$.
	
	By an analogous result to Lemma 6.2 from \cite{BKM} (using Corollary 3.1.7 of \cite{Gee} instead of the results of \cite{DT} that are used there), it follows that each irreducible component of $\Spec R_\infty^\sigma$ contains a point in the support of $M_\infty^\sigma/(i(y_1),\ldots,i(y_d))\otimes_\cO E = M^\sigma\otimes_\cO E$, which is not contained in any other component. Then as in the proof of \cite[Theorem 6.3]{BKM}, as $R_\infty^\sigma$ is reduced, it follows that $R_\infty$ acts faithfully on $M_\infty$ and so $\cR^\sigma_\infty  = R^\sigma_\infty$, and so we indeed have an isomorphism $R_\infty^\sigma\otimes_{S_\infty}\cO\cong R^\sigma$, proving the first part of (3). 
	
	By Proposition \ref{prop:R_v^tau recap}, $R_\infty^\tau$ is Cohen--Macaulay. As in the proof of \cite[Theorem 6.3]{BKM} this, combined with the fact that $M_\infty^\sigma$ is free over $S_\infty$, implies that $R_\infty^\tau$ is free over $S_\infty$, proving (2). As in \cite[Theorem 6.3]{BKM} this also implies that $R^\sigma=R_\infty^\sigma\otimes_{S_\infty}\cO$ is finite free over $\cO$, proving the second part of (3). In particular (as $\T^\sigma$ is finite free over $\cO$ by definition) to show that $R^\sigma\onto \T^\sigma$ is an isomorphism, it will suffice to show that the induced map $R^\sigma[1/\varpi]\onto \T^\sigma[1/\varpi]$ is.	

	Now as in the proof of \cite[Theorem 6.3]{BKM}, $\Spec R_\infty^{\sigma}[1/\varpi]$ is formally smooth at every point in the support of $\Spec M^\sigma\otimes_\cO E$, and so in particular at the point corresponding to $\lambda:R_\infty^\sigma\onto\cO$, proving (5). This is proved as in \cite[Lemma 6.1]{BKM} by using the fact that Galois representations arising from cohomological Hilbert modular forms are known to be generic in the sense of   \cite[Lemma 1.1.5]{Allen}, which follows from the genericity of the corresponding automorphic representation of $\GL_2(\A_F)$ at all finite places and local-global compatibility as recorded in \cite[Theorem 2.1.2]{Allen}.

	The argument of \cite[Theorem 6.3]{BKM} now proves that $R^\sigma[1/\varpi]\onto \T^\sigma[1/\varpi]$ is an isomorphism, and hence $R^\sigma\onto \T^\sigma$ is an isomorphism. This proves (4) in the case when $\Sigma^{\fun}=\es$ (the last claim in (4), that the rings are reduced, is a consequence of the standard fact that the Hecke operators $T_v$ and $S_v$ for $v\not\in\Sigma$ are all simultaneously diagonalizable as operators on $H^\sigma$).

	In the case when $\Sigma^{\fun} = \es$, and hence $\sigma=\tau$, this completes the proof. In the case when $\Sigma^{\fun}\ne \es$, and so $\sigma\ne \tau$, it remains to deduce the statement of the theorem for $\tau$ from the one for $\sigma$.

	First, by the definition of modified global deformation rings given in Section \ref{sec_deformation_theory}, we have that
	\[R^\tau = R_{\loc}^\tau \otimes_{R_{\loc}} R = R_{\loc}^\tau\otimes_{R_{\loc}^\sigma}R^\sigma\]
	and similarly $R^\tau_n = R_{\loc}^\tau\otimes_{R_{\loc}^\sigma}R_n^\sigma$ and $R^{\tau,\square}_n = R_{\loc}^\tau\otimes_{R_{\loc}^\sigma}R_n^{\sigma,\square}$ for all $n\ge 1$. The $S_\infty$-algebra structure on $R_n^{\sigma,\square}$ then induces an $S_\infty$-algebra structure on $R_n^{\sigma,\square}$ and we have
	\[
	R_n^{\tau,\square}\otimes_{S_\infty}\cO 
	= (R_{\loc}^\tau\otimes_{R_{\loc}^\sigma}R_n^{\sigma,\square})\otimes_{S_\infty}\cO
	= R_{\loc}^\tau\otimes_{R_{\loc}^\sigma}(R_n^{\sigma,\square}\otimes_{S_\infty}\cO)
	= R_{\loc}^\tau\otimes_{R_{\loc}^\sigma}R^{\sigma}
	= R^\tau.
	\]
	Also, as $R_n^{\sigma,\square}$ is a quotient of $R_\infty^\sigma$ (as a $R_{\loc}^\sigma$-algebra), if we let
	\[R_\infty^\tau = R_{\loc}^{\tau}[[x_1,\ldots,x_g]] = R_{\loc}^\tau\otimes_{R_{\loc}^\sigma}R_\infty^\sigma\]
	then $R_n^{\tau,\square}$ is a quotient of $R_{\loc}^\tau\otimes_{R_{\loc}^\sigma}R_\infty^\sigma=R_\infty^\tau$ (as a $R_{\loc}^\tau$-algebra).
	
	Now just as in the proof of Lemma \ref{lem:R->T}, the map $R_n^{\sigma}\onto \T^\sigma_n$ induces a map $R_n^{\tau}\onto \Tbar^\tau_n$ making the diagram
	
	\begin{center}
		\begin{tikzpicture}
		\node(01) at (0,2) {$R^\tau_n$};
		\node(11) at (2,2) {$\Tbar^\tau_n$};
		
		\node(00) at (0,0) {$R^\sigma_n$};
		\node(10) at (2,0) {$\T^\sigma_n$};
		
		\draw[->>] (01)--(11);
		\draw[->>] (00)--(10);
		
		\draw[-latex] (00)--(01);
		\draw[-latex] (10)--(11);
		\end{tikzpicture}
	\end{center}
	commute. As the $\T^\sigma_n$-action on $M^\sigma_n$ extends to a $\Tbar^\tau_n$-action (since the $U_v$ operators naturally act on $M^\sigma$), the $R^\sigma_n$-action on $M^\sigma_n$ also extends to a $R^\tau_n$-action on $M^\sigma_n$. Passing to the framed versions (by applying $-\otimes_{R^\sigma_n}R^{\sigma,\square}_n$), it follows that the action of $R^{\sigma,\square}_n$ on $M^{\sigma,\square}_n$ extends to an action of $R^{\tau,\square}_n$. Moreover, it's easy to check that the isomorphism $M^{\sigma,\square}_n\otimes_{S_\infty}\cO\cong M^\sigma$ is compatible with the action of the $U_v$-operators, and so it is an isomorphism of $\Tbar^\tau$-modules, and hence of $R^\tau$-modules.
	
	Combining all of this, we can again use the `ultrapatching' construction of \cite[Section 4.1]{Manning}, with $\{R^{\tau,\square}_n\}$ in place of $\{R^{\sigma,\square}_n\}$ and $R_\infty^\tau$ in place of $R_\infty^\sigma$. This produces a $S_\infty$-algebra $\cR^\tau_\infty$ together with a surjection $R_\infty^\tau\onto \cR^\tau_\infty$ and an isomorphism $\cR_\infty^\tau\otimes_{S_\infty}\cO\cong R^\tau$ such that the composition 
	\[R_{\loc}^\tau\into R_\infty^\tau\onto \cR_\infty^\tau\onto R^\tau\]
	is the map $R_{\loc}^\tau\to R^\tau$.
	
	By the functorality of the ultrapatching construction, the maps $R_n^{\sigma,\square}\to R_n^{\tau,\square}$ induce an $S_\infty$-algebra homomorphism $R_\infty^\sigma = \cR_\infty^\sigma\to \cR_\infty^\tau$. Moreover, the action of $R_n^{\tau,\square}$ on $M^{\sigma,\square}_n$ induces an action of $\cR_\infty^\tau$ on $M_\infty^\sigma$, extending the action of $\cR_\infty^\sigma$. In particular, we may treat $M_\infty^\sigma$ as a $R_\infty^\tau$-module. 
	
	We can now finish the proof. First we have $R_v^\sigma = R_v^\tau$ for $v\not\in \Sigma^\fun$ and $\dim R_v^\sigma = \dim R_v^\tau = 3+1$ for $v\in\Sigma^\fun$, so $\dim R_\infty^\tau = \dim R_\infty^\sigma = \dim S_\infty$, proving (1).
	
	We shall now show (5). First, for $v\in\Sigma\sm \Sigma^\fun$, we have $R_v^{\tau_v} = R_v^{\sigma_v}$ and $\Spec R_v^{\sigma_v}[1/\varpi]$ is formally smooth at the point corresponding to $\lambda:R_v^{\sigma_v}\into R_\infty^\sigma\xrightarrow{\lambda}\cO$ by the above. Thus to show (5), it suffices to show that for each $v\in\Sigma^\fun$, $\Spec R_v^\fun[1/\varpi]$ is also formally smooth at the point corresponding to $\lambda:R_v^\fun\into R_\infty^\tau\xrightarrow{\lambda}\cO$.
	
	Take any such $v\in \Sigma^\fun$. Recall that by assumption the representation $\rho|_{G_v}$ is Steinberg. Thus the point of $\Spec R_v^\un[1/\varpi]$ corresponding to $\lambda:R_v^\un\into R_\infty^\sigma\xrightarrow{\lambda}\cO$ is in the Steinberg component and not in the unramified component (it can't lie on both components, as it corresponds to a formally smooth point of $\Spec R_\infty^\sigma[1/\varpi]$, by the above argument). But now by the explicit descriptions of the rings $R_v^\un$ and $R_v^\fun$ given in Section \ref{sec:computations}, it follows that the point of $\Spec R_v^\fun[1/\varpi]$ corresponding to $\lambda:R_v^\fun\into R_\infty^\tau\xrightarrow{\lambda}\cO$ is also contained in the Steinberg components, but not the unramified component. But the description of $R_v^\fun$ from Section \ref{sec:computations} shows that $\Spec R_v^\fun[1/\varpi]$ is formally smooth at any such point. This proves (5).
	
	As $M_\infty^\sigma$ is maximal Cohen--Macaulay over $R_\infty^\sigma$, it follows that it is also maximal Cohen--Macaulay over $R_\infty^\tau$, and so the support of $M_\infty^\sigma$ as an $R_\infty^\tau$-module is again a union of irreducible components of $\Spec R_\infty^\tau$. But now for each $v\in\Sigma$, the irreducible components of $R_v^\sigma$ are in bijection with those of $R_v^\tau$ (this is trivial for $v\not\in \Sigma^\fun$ and for $v\in\Sigma^{\fun}$ follows from the description of $R^\un_v$ and $R^\fun_v$ given in Section \ref{sec:computations}). By Proposition \ref{prop:R_v-Loc}, it follows that the irreducible components of $\Spec R_\infty^\sigma$ are in bijection with those of $\Spec R_\infty^\tau$. Since $M_\infty^\sigma$ is supported on all of $\Spec R_\infty^\sigma$, it follows that $M_\infty^\sigma$ is supported on all of $\Spec R_\infty^\tau$ as well. Since $R_\infty^\tau$ is reduced, it follows that $R_\infty^\tau$ acts faithfully on $M_\infty^\sigma$. Since the action of $R_\infty^\tau$ on $M_\infty^\sigma$ factors through $R_\infty^\tau\onto \cR_\infty^\tau$, it follows that $R_\infty^\tau\cong \cR_\infty^\tau$.
	
Just as before, (2) and (3) follow from this, and again, the second part of (3) implies that to show that $R^\tau\onto \Tbar^\tau$ is an isomorphism, it will suffice to show that the induced map $R^\tau[1/\varpi]\onto \Tbar^\tau[1/\varpi]$ is.
	
To prove (4), consider the commutative diagram
\begin{center}
	\begin{tikzpicture}
		\node(01) at (0,2) {$R^\tau[1/\varpi]$};
		\node(11) at (3,2) {$\Tbar^\tau[1/\varpi]$};
		
		\node(00) at (0,0) {$R^\sigma[1/\varpi]$};
		\node(10) at (3,0) {$\T^\sigma[1/\varpi]$};
		
		\draw[->>] (01)--(11);
		\draw[->>] (00)--(10);
		
		\draw[-latex] (00)--(01);
		\draw[-latex] (10)--(11);
	\end{tikzpicture}
\end{center}

As the bottom map is an isomorphism of finite free reduced $E$-algebras, to show that the top map is an isomorphism, it will suffice to show that for any $\Qbar_p$ point $\eta:\T^\sigma[1/\varpi]\to \Qbar_p$ of $\Spec \T^\sigma\cong \Spec R^\sigma$ the induced map $R^\tau\otimes_\eta \Qbar_p\onto \Tbar^\tau\otimes_\eta\Qbar_p$ is an isomorphism.

Fix any such $\eta:\T^\sigma[1/\varpi]\to \Qbar_p$. Then $\eta$ corresponds to a modular Galois representation $\rho_\eta:G_F\to \GL_2(\Qbar_p)$ lifting $\rhobar$. For each $v\in\Sigma^\fun$, $\rhobar|_{G_{F_v}}$ must be either Steinberg or unramified. Let $S_\eta\subseteq \Sigma^\fun$ be the set of $v\in\Sigma^\fun$ for which $\rho_\eta|_{G_{F_v}}$ is unramified.

By the definition given in Section \ref{sec:deformation}, $R^\tau = R^\sigma\left[a_v\middle|v\in\Sigma^\fun\right]$, where for each $v\in\Sigma$, $a_v$ is the chosen root of the characteristic polynomial of $\rho_{\eta}(\Frob_v)$. Hence

\[R^\tau\otimes_\eta \Qbar_p = (R^\sigma\otimes_\eta \Qbar_p)\left[a_v\middle|v\in\Sigma^\fun\right] = \Qbar_p\left[a_v\middle|v\in\Sigma^\fun\right]\]
For $v\in \Sigma^\fun\sm S_\eta$ (so that $\rho_\eta|_{G_{F_v}}$ is Steinberg) the definition of $R_v^\fun$ implies that $a_v=\pm1\in \Qbar_p$, so in fact, $R^\tau\otimes_\eta \Qbar_p = \Qbar_p\left[a_v\middle|v\in S_\eta\right]$, and so $R^\tau\otimes_\eta \Qbar_p$ is a quotient of 
\[\Qbar_p\left[x_v\middle|v\in S_\eta\right]/(x_v^2-x_v\tr \rho_v^\square(\Frob_v)+\det\rho_v^\square(\Frob_v)).\]
In particular, we have $\dim_{\Qbar_p}R^\tau\otimes_\eta \Qbar_p\le 2^{|S_\eta|}$.

On the other hand, $\Tbar^\tau = \T^\sigma\left[U_v\middle|v\in\Sigma^\fun\right]$ is a subalgebra of $\End_{\cO}(M^\tau)$, and so 
\[\Tbar^\tau\otimes_\eta \Qbar_p = (\T^\sigma\otimes_\eta \Qbar_p)\left[U_v\middle|v\in\Sigma^\fun\right] = \Qbar_p\left[U_v\middle|v\in\Sigma^\fun\right] = \Qbar_p\left[U_v\middle|v\in S_\eta\right]\]
 is a subalgebra of $\End_{\Qbar_p}(M^\tau\otimes_\eta \Qbar_p)$ (where the last inequality comes from the fact that $U_v$ acts as a scalar on $M^\tau\otimes_\eta \Qbar_p$ if $\rho_\eta|_{G_{F_v}}$ is Steinberg). But now as $\rho_\eta$ is unramified at each $v\in S_\eta$, it corresponds to a Hilbert modular form $f_\eta$ of level not divisible by any $v\in S_\eta$. Standard properties of Hilbert modular forms now imply that $\dim_{\Qbar_p}\Tbar^\tau\otimes_\eta \Qbar_p = \dim_{\Qbar_p}\Qbar_p\left[U_v\middle|v\in S_\eta\right] = 2^{|S_\eta|}$; we are using here  that the $U_v$ for $v \in S_\eta$  act as independent non-scalar endomorphisms on the $2^{|S_\eta|}$  dimensional ($\Qbar_p$-) vector space generated by the image of $f_\eta$ under the standard degeneracy maps arising from the places $v \in S_\eta$.
 Thus $\dim_{\Qbar_p}\Tbar^\tau\otimes_\eta \Qbar_p= 2^{|S_\eta|} \ge \dim_{\Qbar_p}R^\tau\otimes_\eta$, and so as the map $R^\tau\otimes_\eta \Qbar_p\onto \Tbar^\tau\otimes_\eta\Qbar_p$ is surjective, it must be an isomorphism.
 This completes the proof of (4), and thus of the theorem.
\end{proof}

Combining this with Proposition \ref{prop:delta prod} and the computations in Section \ref{sec:computations} gives the following generalization of \cite[Theorem 10.1]{BKM}:

\begin{theorem}\label{thm:mc}
	In setting described in this section we have:
	\[
	\delta(R^\tau) = \delta(\T^\tau) = \sum_{v\in\Sigma^{\st}}\frac{2n_v}{e}+\sum_{v\in\Sigma^{\fun}}\frac{3n_v}{e}+\sum_{v\in\Sigma^{\un}}\frac{n_v}{e}
	\]
	where $n_v$ is as above, and $e$ is the ramification index of $E/\Q_p$.
\end{theorem}
\begin{proof}
	Theorem \ref{thm:patching} implies that the map $\theta:S_\infty\to R_\infty^\tau$ satisfies property \ref{prop P}, and so Theorem \ref{thm:delta invariant} implies that implies that
	\[
	\delta_\lambda(\T^\tau) = \delta_\lambda(R^\tau) = \delta_\lambda(R_\infty^{\tau}\otimes_{S_\infty} \cO) = \delta_\lambda(R_\infty^\tau).
	\]
	Now by Proposition \ref{prop:delta prod} and Proposition \ref{prop:Wiles generalization} we get:
	\begin{align*}
	\delta_\lambda(R_\infty^\tau) 
	&= \delta_\lambda(R_{\loc}^\tau[[x_1,\ldots,x_g]])
	= \delta_\lambda(R_{\loc}^\tau)+\delta_\lambda(\cO[[x_1,\ldots,x_g]])
	= \delta_\lambda(R_{\loc}^\tau)\\
	&= \delta_\lambda\left(\left(\cbigotimes{v\in \Sigma}R_v^{\tau_v}\right)\cotimes\left( \cbigotimes{v|p}R_v^{\fl}\right)\right)
	= \sum_{v\in \Sigma}\delta_\lambda(R_v^{\tau_v})+\sum_{v|p}\delta_\lambda(R_v^{\fl})\\
	&= \sum_{v\in \Sigma}\delta_\lambda(R_v^{\tau_v})+\sum_{v|p}\delta_\lambda(\cO[[x_1,\ldots,x_{3+[F_v:\Q_p]}]])
	= \sum_{v\in \Sigma}\delta_\lambda(R_v^{\tau_v})\\
	&= \sum_{v\in \Sigma^{\min}}\delta_\lambda(R_v^{\min})
	+\sum_{v\in \Sigma^{\st}}\delta_\lambda(R_v^{\st})
	+\sum_{v\in \Sigma^{\fun}}\delta_\lambda(R_v^{\fun})
	+\sum_{v\in \Sigma^{\un}}\delta_\lambda(R_v^{\un})
	+\sum_{v\in \Sigma^{\square}}\delta_\lambda(R_v^{\square}).
	\end{align*}
	Now Proposition \ref{prop:R_v^tau} implies that $R_v^{\min}$ and $R_v^{\square}$ are complete intersections, so Proposition \ref{prop:Wiles generalization} gives $\delta_\lambda(R_v^{\min}) = 0$ and $\delta_\lambda(R_v^{\square})=0$. Thus the claim follows by the computations in Theorems \ref{thm:delta_v^St}, \ref{thm:delta-PhiUn} and \ref{thm:delta-Un}.
\end{proof}

As a side note, while this theorem only computes the ``non-cohomological" Wiles defect, and \cite[Theorem 10.1]{BKM} computes both the cohomological and non-cohomological defects, we still have these defects are equal in the minimal level case (i.e. $\Sigma^{\un}=\Sigma^{\fun}=\Sigma^\square = \es$) by \cite[Theorem 1.2]{Manning} and \cite[Theorem 3.12]{BKM}. 

In the next section we show that in fact our work here, which determines the defect of Hecke algebras and deformation rings, can be used  to show an equality of cohomological and non-cohomological defects in many situations.

\section{Cohomological Wiles defects and  degrees of parametrizations by Shimura curves}\label{sec:ribet}

The main theorem of this paper,  Theorem \ref{thm:mc},  that we have proven above  computes Wiles defects of Hecke algebras acting on the cohomology of modular curves and Shimura curves. We use this to compute in the present section  the Wiles defect of the modules  of the Hecke algebras of Theorem \ref{thm:mc}  that are  given by the cohomology of the Shimura curve on which  the respective Hecke algebras acts faithfully ; Theorem \ref{change-of-eta} and Proposition \ref{prop:moduledefect} below. 

Our methods here  also allow  us to  improve on  the results of \cite{RiTa} about degrees of optimal parametrizations of elliptic curves over $\Q$ by Shimura curves: see Corollaries \ref{degree} and \ref{component} below. Our approach diverges considerably from the one of \cite{RiTa}. Our proofs are rather indirect but seem to fill in a lacuna caused by the basic problem that one does not know in generality  surjectivity of maps on $p$-parts of  component groups at primes $q$  (of multiplicative reduction), induced by optimal parametrization of  an elliptic curve $E$ over $\Q$     by a Shimura curve, when the prime $q$  divides the discriminant of the corresponding quaternion algebra (see the remarks in    \S \ref{remarks}).    We only consider non-Eisenstein primes, namely  primes  $p$  such that $E[p]$ is irreducible.  The arguments in \cite[page 11113]{RiTa} which prove this in special cases  rely on auxiliary hypotheses: for instance that there is a prime $q$  such that the image of an inertia group $I_q$ at $q$ acting on $E[p]$ has image of order $p$. The difficulty of proving the  surjectivity alluded to above  is specially vexing when considering component groups at a prime $q$ that is trivial for $E[p]$.  Both corollaries are deduced from  Theorem \ref{change-of-eta} and Proposition \ref{prop:moduledefect}.

We work  with the setup in \cite[Section 5]{BKM}   and thus operate (mainly  for simplicity)  at less generality than the work in the previous sections (for instance we will asume $F=\Q$.)  There are slight differences between the setup here and that of  \cite[Section 5]{BKM} that we begin by  highlighting.

Fix $Q$ a finite set of primes and let $D_Q$ be the quaternion algebra over $\Q$ considered in \cite[\S 5]{BKM}: it is  definite if $Q$ has odd cardinality and  indefinite if $Q$  is of even cardinality.  (By abuse of notation, we will also frequently use $Q$ to denote the product of all the primes in the set $Q$. The context will make clear which meaning is intended.) We assume  here that $Q$ has even cardinality and thus $D_Q$ is an indefinite quaternion algebra.
For a  positive integer $N$ with $(N,Q)=1$  let $\Gamma_0^{Q}(N)$   be the congruence subgroup for $D_{Q}^\times$,  which is maximal compact at primes in $Q$, and upper triangular mod $\ell$ for all $\ell|N$. We consider also the usual congruence subgroups $\Gamma_0(NQ)$ and $\Gamma_0(N^2Q^2)$ of $\SL_2(\Z)$. Let $K_0(N^2Q^2)\subseteq \GL_2(\A_{\Q,f})$, and $K_0^{Q}(NQ)\subseteq D_{Q}^\times(\A_{\Q,f})$ be the corresponding compact open subgroups.

 We consider  $X^{Q}_0(N)$  the (compact) Riemann surface 
\[D_{Q}^\times(\Q)\backslash \left(D_{Q }^\times(\A_{\Q,f})\times \half\right)/K_0^{Q}(N)\] (where $\half$ is the complex upper half plane). Give $X^{Q}_0(N)$ its canonical structure as an algebraic curve over $\Q$. Let  as before  $p$  be a prime not dividing $2NQ$,  and  we fix a finite extension $E/\Q_p$,  with  $\cO$ the ring of integers in $E$, $\varpi$ a uniformizer, $k=\cO/\varpi$ the residue field, and $e$ the ramification index of $E/\Q_p$. We will assume below that $E$ is sufficiently large  so that $\cO$ contains the Fourier coefficients of all newforms in $S_2(\Gamma_0(N^2Q^2))$.  Consider  $S^{Q}(\Gamma_0^{Q}(N))=H^1(X^{Q}_0(N),\cO)$, with its natural $\T^{Q}(N)[G_\Q]$-module structure.  We also denote $S(N^2Q^2)=H^1(X_0(N^2Q^2),\cO)$ and $S(NQ)=H^1(X_0(NQ),\cO)$.

Let $\T(N^2Q^2)$, $\T(NQ)$  and $\T^{Q}(N)$ be the  $\cO$-algebras at level $\Gamma_0(NQ^2)$, $\Gamma_0(NQ)$  and $\Gamma_0^{Q}(N)$, respectively, generated by  the Hecke operators $T_r$ for  primes $r$ coprime to $pNQ$ acting on $S(N^2Q^2)$, $S(NQ)$ and $S^{Q}(\Gamma_0^{Q}(N))$. Note that by the Jacquet--Langlands correspondence, $\T^{Q}(N)$ is a quotient of $\T(N^2Q^2)$, and this quotient factors through $\T(NQ)$.

Let $f \in S_2(\Gamma_0(NQ))$ be a newform of level $NQ$ such that all its Fourier coefficients lie in $E$, and consider the corresponding $\cO$-algebra homomorphisms $\lambda_f:\T(N^2Q^2) \to \cO$ and   (abusing notation slightly)  $\lambda_f:\T(NQ) \to \cO$.   We will fix this newform and our main results will be in relation to $f$. By the Jacquet--Langlands correspondence, this also gives a related homomorphism $\T^{Q}(N) \to \cO$ that we  again denote by the same symbol $\lambda_f$.    We denote the corresponding maximal ideals which contain the prime ideal $\ker(\lambda_f)$ by the same symbol  $\ffrm$. Let $\rho_f : G_\Q \to \GL_2(\cO)$ be the Galois representation associated by Eichler and Shimura to $f$ and assume   that the corresponding residual  Galois representation  $\rhobar_f=\rhobar :G_\Q\to \GL_2(k)$ is  absolutely irreducible. By enlarging $\cO$ if necessary,  we may assume that $k$ contains all eigenvalues of $\rhobar(\sigma)$ for all $\sigma\in G_\Q$.  The Galois representation $\rho_f:G_\Q \to \GL_2(E)$, with irreducible residual representation $\rhobar$,  is locally at $q$ of the form \[   \left( \begin{array}{cc} \epsilon& \ast \\ 0 & 1 \end{array} \right),\]   up to twist by an unramified character $\chi$ of order dividing 2.  The $\beta_q \in \{\pm 1\}$  of Section \ref{sec_deformation_theory} will be  chosen so that  $\rho_f|_{G_q}$ gives rise to a point of  $\Spec   R_q^{\st}$ in what follows (and thus depends on whether $\chi$ is trivial or not).

There is an oldform $f^{NQ}$ in $S_2(\Gamma_0(N^2Q^2))$  with corresponding newform $f$ which is characterised by the property that it is an eigenform for the  Hecke operators  $T_\ell$ for $\ell$ prime with $(\ell,NQ)=1$ and $U_\ell $ for $\ell|NQ$ and  such that $a_\ell(f^{NQ})=0$, i.e., $f^Q|U_\ell=0$,  for $\ell|NQ$. Let  $\lambda_{f^{NQ}}:  \T^{\rm full}(N^2Q^2) \to \cO$   be the induced homomorphism of the full Hecke algebra $\T^{\rm full}(N^2Q^2)$   acting on $H^1(X_0(N^2Q^2),\cO)$ which is generated as an  $\cO$-algebra by the  action of the Hecke operators $T_\ell$ for $(\ell,NQ^2)=1$ and $U_\ell $ for $\ell|NQ$ on $S(N^2Q^2)=H^1(X_0(N^2Q^2),\cO)$. We denote by $\ffrm_{Q}$    the maximal ideal of $\T^{\rm full}(N^2Q^2)$ that contains the kernel of $\lambda_{f^{NQ}}$. 

The homomorphism $\lambda_f: \T^{Q}(N) \to \cO$ 
 extends to the full Hecke algebra  acting on  $S^Q(\Gamma_0^{Q}(N))$, and we denote   by $\ffrm_{Q}$ again  the maximal ideal
 of $\T^{Q}(N)$ which contains the kernel of the extended homomorphism.
  We define  $\T$, $\T^{\uni}$ and $\T^{\St,{Q}} $  to be the image of  $\T(NQ^2)$ and $\T^{Q}(N)$  in the endomorphisms of the finitely generated $\cO$-modules  $S(N^2Q^2)=H^1(X_0(N^2Q^2),\cO)_{\ffrm_Q}$,  $S(NQ)=H^1(X_0(NQ),\cO)_{\ffrm}$ and $S^{Q}(\Gamma_0^{Q}(N))_{\ffrm_{Q}}$ respectively. We denote by $R, R^{\uni}, R^{\st,Q}$ the corresponding universal deformation rings and thus we have surjective maps $R \onto \T$, $R^\uni \onto \T^{\uni}$ and $R^{\st,Q} \onto \T^{\st, Q}$ of $\cO$-algebras.  (Thus in each of these cases the type $\tau=(\tau_v)$ for $v|NQ$  is such that $\tau_v$ is unrestricted, or unipotent,  or unipotent at $v|N$ and Steinberg at $v|Q$.)
We have the corresponding universal modular deformation $\rho^{\mod}:G_\Q \to \GL_2(\T))$ by results of Carayol which is a specialization of  a universal representation $G_\Q \to \GL_2(R)$. 

Define
\[M(N^2Q^2)= \Hom_{\T[G_\Q]}(\rho^{\mod},S(N^2Q^2)_{\ffrm_Q}^*),\] 
\[M(NQ)= \Hom_{\T[G_\Q]}(\rho^{\mod},S(NQ)_{\ffrm}^*),\] 
\[M^{\st,{Q}}(N)= \Hom_{\T[G_\Q]}(\rho^{\mod},S^{Q}(\Gamma_0^{Q}(N))_{\ffrm_{Q}}^*).\]

As in  Lemma 5.1 of \cite{BKM}, we have using  \cite{Carayol2} that the evaluation map $M(N^2Q^2)\otimes_\T \rho^{\mod}\to S(N^2Q^2)_{\ffrm_Q}^*$ is an isomorphism, as is $M^{\St,{Q}}(N)\otimes_\T \rho^{\mod}\to S^{Q}(\Gamma_0^{Q}(N))_{\ffrm_{Q}}^*$. In particular, as $\T$-modules we have $S(N^2Q^2)_{\ffrm_Q}^* = M(N^2Q^2)^{\oplus 2}$ and $S^{Q}(\Gamma_0^{Q}(N))_{\ffrm_{Q}}^* = M^{\st,{Q}}(N)^{\oplus 2}$. 

We have the following lemma proved using Proposition 4.7 of \cite{DDT} (see proof of Theorem 5.2 of \cite{BKM}).

 \begin{lemma}
\begin{description}
	\item[(i)] The Hecke module  $M(N^2Q^2)[{1 \over p}]$   is  free of rank one over $\T[{1 \over p}]$. 
	\item[(ii)]     The $\T$-modules  \[M(N^2Q^2), 
	M(NQ), M^{\St,{Q}}(N)\] are self-dual. 
	\item[(iii)] The $\cO$-modules  \[M(N^2Q^2)[\ker(\lambda_{f^{NQ}})], 
	M(NQ)[\ker(\lambda_f)], M^{\St,{Q}}(N)[\ker(\lambda_f)]\] are each free  of rank 1 over $\cO$.  
\end{description}
 \end{lemma}

  \begin{proof}
  The first part follows from the arguments in Proposition 4.7 of \cite{DDT} (see proof of Theorem 5.2 of \cite{BKM}).
  For the second part we use that $f$ is a newform of level $NQ$, and the explicit description of $f^{NQ}$ and the corresponding maximal ideal
  $\ffrm_Q$  that is used to define $M(N^2Q^2)$.
\end{proof}

\begin{remark}
 In general the modules   $M(NQ), M^{\St,{Q}}(N)$, because of the presence of oldforms,   are not generically free over the  anemic  Hecke algebras acting on them that do  not have the operators $U_v$ for $v|NQ$ in them. This generic freeness  holds  for $M^{\St,{Q}}(N)$ if $N|N(\rhobar)$ which was the assumption in \cite{BKM}.  They are  generically  free over the  full  Hecke algebras acting on them that  have the operators $U_v$ for $v|NQ$ in them.
\end{remark}

\begin{remark}
The definition of the modules $M(N^2Q^2)$, $M(NQ)$ and $M^{\st,{Q}}(N)$ differs slightly from the definition of the modules $M^\tau$ from Section \ref{sec:deformation}. In particular, we do not quotient by the elements $S_v-\varepsilon_p(\Frob_v)$ (or even explicitly use the Hecke operators $S_v$). The definition of $M^\tau$ from Section \ref{sec:deformation} is needed when $F\ne \Q$ in order to make the patching argument work (for subtle reasons involving the unit group $\cO_F^\times$). In this section, we are only considering the case $F=\Q$ for convenience, and so we are still able to use patching arguments  with  the simpler definitions of the modules given in this section.

Also here we `factor out' the Galois representation $\rho^{\mod}$ as above, while we do not do so in Section \ref{sec:deformation}. This also does not significantly affect the patching argument. See \cite[Theorem 6.3]{BKM} or \cite[Section 4]{Manning} for more details on patching arguments in which the Galois representation is factored out.

In particular, one can prove completely analogous versions of Theorems \ref{thm:patching} and \ref{thm:mc} for the modules defined in this section, with only minimal modifications to the proofs. We will leave the details of this to the interested reader, and for the remainder of the section we will simply cite the results of Section \ref{sec:deformation} as if they literally applied to the modules considered in this section.
\end{remark}

We denote by $\langle \ , \ \rangle$ certain  $\cO$-valued, perfect $\T$-equivariant pairings  on the $\T$-modules  \[M(N^2Q^2), 
  M(NQ), M^{\St,{Q}}(N)\]  that are induced by Poincare duality   (see \cite[\S 9]{BKM}).  We then recall from 
  \cite[\S 3, Lemma 3.5]{BKM}, that if $X,Y,Z$ are generators of  the rank one $\cO$-modules \[M(N^2Q^2)[\ker(\lambda_{f^{NQ}})], 
  M(NQ)[\ker(\lambda_f)], M^{\St,{Q}}(N)[\ker(\lambda_f)],\]   we have the following relationship:
    \[\Psi_\lambda(M(N^2Q^2))= \cO/(\langle X , X \rangle) , 
 \Psi_\lambda(M(NQ))=\cO/(\langle Y, Y \rangle) , \Psi_\lambda(M^{\St,{Q}}(N))=\cO/(\langle Z,Z \rangle).\]  Here we are abbreviating all the augmentations arising from the newform $f$  to $\lambda$.

Let $\cA_f$ stand for the isogeny class of the abelian variety $A_f$ (which is an optimal quotient of $J_0(NQ)$). The residual representations arising from the class $\cA_f$ with respect to the fixed embedding $K_f \hookrightarrow  \overline \Q_p$ are all isomorphic to our fixed absolutely irreducible $\rhobar$.

 \begin{theorem}\label{change-of-eta}
    We have   the equality of lengths of $\cO$-modules:
  \[\ell_\cO(\Psi_\lambda(M(N^2Q^2)))= \ell_\cO(\Psi_\lambda(M^{\st,Q}(N))) +  \sum_{\ell|N} \ord_\cO(\ell^2-1)+ \sum_{q \in Q} (m_q+\ord_\cO(q^2-1)), \] and 
  \[\ell_\cO(\Psi_\lambda(M(NQ)))= \ell_\cO(\Psi_\lambda(M^{\st,Q}(N))) + \sum_{q \in Q} m_q. \] 
  
We have equality of defects  $\delta_{\lambda,\T^{\st,Q}}(M^{\st,Q}(N)) =\delta_{\lambda,\T^{\st,Q}}(\T^{\st,Q})=\sum_{\ell|N} \frac{n_\ell}{e}+\sum_{q \in Q} \frac{2n_q}{e}$.

     \end{theorem}

\begin{proof}

The proof follows from the following facts:

\begin{enumerate}

\item We use the exact computation of the length of a   relative cotangent space, namely  \[\ell_\cO(\Phi_{R/R^{\st,Q}})=  \ell_\cO(\Phi_{\T/\T^{\st,Q}})=\sum_{\ell|N} ( \ord_\cO (\ell^2-1)  -n_{\ell}) +\sum_{q \in Q}(m_q+\ord_\cO(q^2-1)-2n_q)\] by a slight variant of the arguments in the proof of  \cite[Corollary 7.15]{BKM}  using as key input Theorem \ref{thm:patching} (there the level considered when we relax ramification conditions is  $NQ^2$ rather than $N^2Q^2$, and  it is assumed that $N|N(\rhobar)$, but the arguments carry  over to our slightly different situation  {\em mutatis mutandis});

\item $\delta_{\lambda, \T}(\T)=\delta_{\lambda, \T}(M(N^2Q^2))=0$. This follows from the arguments in \cite[Theorem 5.2]{BKM}  (see also
\cite[Remark 5.3, 5.4]{BKM}) which  is proved using the arguments of \cite[Theorem 3.4]{DiamondMult1}.

\item The inequality \[\ell_\cO(\Psi_\lambda(M(N^2Q^2))) \leq \ell_\cO(\Psi_\lambda(M^{\st,Q}(N)))  + \sum_{\ell|N} ( \ord_\cO (\ell^2-1) )+ \sum_{q \in Q} (m_q+\ord_\cO(q^2-1))\]  that follows from  the  following two inequalities:

\begin{itemize}

\item   \[\ell_\cO(\Psi_\lambda(M(NQ))) \leq \ell_\cO(\Psi_\lambda(M^{\st,Q}(N))) + \sum_{q \in Q} m_q\] which follows from \cite[Theorem 2]{RiTa}. To justify this, as  noted above    as a consequence of  \cite[\S 3, Lemma 3.5]{BKM}, we have 
$\ell_\cO(\Psi_\lambda(M(NQ)))=\ord_\cO(\langle Y, Y \rangle)$ and $\ell_\cO(\Psi_\lambda(M^{\st,Q}(N)))=\ord_\cO(\langle Z, Z \rangle)$. Further the  ideals generated by the inner products  $(\langle Y, Y \rangle)$  and $(\langle Z, Z \rangle)$ can be  read off from the optimal  quotients   $\xi$ and $\xi’$ of  the isogeny class of abelian varieties ${\cal A}_f$ by  the Jacobians  of $X_0(NQ)$ and  $X^Q_0(N)$ as follows. The composition  $\xi_* \xi^*$ of the  pull-back $\xi^*$ and  push-forward of the maps induced by $\xi$ on the ${\rm Ta}_\wp(A)_{\ffrm}=\cO^2$  is identified  with multiplication by a scalar in $\cO$. We denote the ideal of $\cO$ generated by this scalar  by  $(\xi_* \xi^*)$. Then $(\langle Y, Y \rangle)=(\xi_* \xi^*)$. Similarly  $(\langle Z, Z \rangle)=(\xi’_* {\xi’} ^*)$.  Then   (using  \cite[Theorem 2]{RiTa} (in the case when ${\cal A}_f$ is an isogeny class of elliptic curves, and its generalization to optimal abelian variety quotients in  \cite{K})  we see that the ideal $(\xi_* \xi^*)(\xi’_* {\xi’} ^*)^{-1}$ divides the ideal  $(\Pi_{q \in Q}\omega^{m_q})$ of $\cO$ which justifies our claim.
 
\item \[\ell_\cO(\Psi_\lambda(M(N^2Q^2))) \leq \ell_\cO(\Psi_\lambda(M(NQ)))  + \sum_{\ell|NQ}  \ord_\cO (\ell^2-1)\] 
This statement, in the stronger form of an equality follows  easily from the arguments in Step 2 of proof of \cite[Proposition 9.1]{BKM}.

\end{itemize}
\item The inequality  \[\ell_\cO(\Psi_\lambda(M^{\st,Q}))  \leq \ell_\cO(\Psi_\lambda(\T^{\st,Q})),\] which is equivalent to the inequality
 \[\delta_\lambda(M^{\st,Q}(N))  \geq \delta_\lambda(\T^{\st,Q}).\] This follows  from  \cite[Theorem 3.12]{BKM}.
\item The equality $\delta_{\lambda,\T^{\St,Q}} = \sum_{\ell |N} \frac{n_\ell}{e} + \sum_{q|Q} \frac{2n_q}{e}$  which is a consequence of our main theorem Theorem \ref{thm:mc}. (To deduce this from our main theorem, we use for  $\ell|N$ the local deformation condition described by $R_\ell^{\uni}$ and for $q \in Q$ that described by $R_q^{\St}$.)
\end{enumerate}

 Using the first three points (1), (2) and (3) we conclude that $\delta_\lambda(M^{\St,Q}(N)) \leq \sum_{\ell |N} \frac{n_\ell}{e} + \sum_{q|Q} \frac{2n_q}{e}$. Using  (4) and (5)  we  deduce the series of (in)equalities  
 \[\sum_{\ell|N} \frac{n_\ell}{e}+\sum_{q \in Q} \frac{2n_q}{e}= \delta_{\lambda,\T^{\st,Q}} (\T^{\st,Q}) \leq  \delta_\lambda(M^{\St,Q}(N))   \leq \sum_{\ell|N} \frac{n_\ell}{e}+\sum_{q \in Q} \frac{2n_q}{e}\] and hence \[\delta_{\lambda,\T^{\st,Q}}(M^{\st,Q}(N)) =\delta_{\lambda,\T^{\st,Q}}(\T^{\st,Q})=\sum_{\ell|N} \frac{n_\ell}{e}+\sum_{q \in Q} \frac{2n_q}{e}.\]  From this  using  (1) and (2) we conclude that  \[\ell_\cO(\Psi_\lambda(M(N^2Q^2)))= \ell_\cO(\Psi_\lambda(M^{\st,Q}(N))) + \sum_{\ell|N} \ord_\cO(\ell^2-1))+ \sum_{q \in Q} (m_q+\ord_\cO(q^2-1)).\] Finally using the two inequalities that occurred in proof  of (3) above we deduce that   \[\ell_\cO(\Psi_\lambda(M(NQ)))= \ell_\cO(\Psi_\lambda(M^{\st,Q}(N))) + \sum_{q \in Q} m_q, \] finishing the proof of all parts of the theorem. 

\end{proof}

We note a variant of the result above  which computes  defects  for the module $M(NQ)$  when considered as a module for an anemic Hecke algebra and a full Hecke algebra.
The module $M(NQ)$ is a module  for  the (anemic)  Hecke algebra $\T^\uni$, and it is also a module for   the  (full) Hecke algebra $\overline{\T^\uni}$ (and thus  $U_v \in \overline \T^\uni$ for all primes $v$ dividing $NQ$) that acts faithfully on $M(NQ)$.  The augmentation $\lambda:\T^\uni \to \cO$ extends uniquely  to $\lambda’:\overline \T^\uni \to \cO$, and $\lambda’(U_v)=\pm 1$  for $v|NQ$.
We  determine next  the defects 
 $\delta_{\lambda’,\overline \T^\uni}(M(NQ))$ and $\delta_{\lambda,\T^\uni}(M(NQ))$.
 
\begin{proposition}\label{prop:moduledefect}
	\begin{description}
		\item[(i)]$\delta_{\lambda,\overline \T^\uni}(M(NQ))=\delta_\lambda(\overline \T^\uni)=\sum_{v|NQ} \frac{3n_v}{e}$. 
		\item[(ii)] $\delta_{\lambda,\T^\uni}(M(NQ))=\delta_{\lambda,\T^\uni}(\T^\uni)=\sum_{v|NQ} \frac{n_v}{e}$.
	\end{description}  
\end{proposition}

\begin{proof}
(i)  By Theorem \ref{thm:mc}, $\delta_\lambda(\overline \T^\uni)=\sum_{v|NQ} \frac{3n_v}{e}$.
Using   arguments pioneered by Mazur to prove mod $p$ multiplicity one statements (see  for instance \cite[Theorem 2.1]{Wiles} for an example of this type of argument, note that under our hypothesis $(p,NQ)=1$),  one sees that $M(NQ)$ is a free $\overline \T^\uni$-module, and thus $\delta_{\lambda, \overline \T^\uni}(M(NQ))=\delta_\lambda(\overline \T^\uni)$.

(ii)  In this case we argue as in the proof of Theorem \ref{change-of-eta} except that the proof is  easier.
Namely we  first observe that  \[\ell_\cO(\Phi_{R/R^\uni})=\sum_{\ell|NQ} ( \ord_\cO (\ell^2-1)  -n_{\ell}) )\] by a slight variant of the arguments in the proof of  \cite[Corollary 7.15]{BKM}.  Further   \[\ell_\cO(\Psi_\lambda(M(N^2Q^2))) = \ell_\cO(\Psi_\lambda(M(NQ)))  + \sum_{\ell|NQ}  \ord_\cO (\ell^2-1).\] This together with $\delta_{\lambda, \T}(M(N^2Q^2))=0$, proves that  $\delta_{\lambda,\T^\uni}(M(NQ))=\sum_{v|NQ} \frac{n_v}{e}$.
Theorem \ref{thm:mc} gives that $\delta_\lambda(\T)=\sum_{v|NQ}\frac{n_v}{e}$, and thus altogether we get that $\delta_{\lambda,\T^\uni}(M(NQ))=\delta_{\lambda,\T^\uni}(\T^\uni)=\sum_{v|NQ} \frac{n_v}{e}$.
\end{proof}

\begin{remark}
We could prove Proposition \ref{prop:moduledefect}(i) by a different method that exploits the equality of congruence modules  $\Psi_{\lambda,\overline \T^\uni}(M(NQ))=\Psi_{\lambda, \T^\uni}(M(NQ))$. This should follow from \cite[Lemm 3.4]{BKM} (see also \cite[Lemma 3.7]{BIK}) on using the fact that $M(NQ)[\ker(\lambda)]=M(NQ)[\ker(\lambda’)]=\cO$. Then we have to compute the change of the local cotangent space at $v$ when we consider the induced augmentations of the map of  local deformation rings $R_v^\uni \to \overline R_v^\uni$. We have not done this computation but one can make the  educated guess that the difference of the lengths of the respective cotangent spaces is $2n_v$. This would also compute the defects when we consider $M(NQ)$ as a module for Hecke algebras that have $U_v$ in them for only a subset  $\Sigma$ of places that divide $NQ$, our  educated guess for this defect  is  \[\sum_{v \in \Sigma} \frac{3n_v}{e} + \sum_{v|NQ, v \notin \Sigma} \frac{n_v}{e}.\]
\end{remark}

From Theorem \ref{change-of-eta} it is easy to deduce the formula for the change of degrees of optimal parametrizations of elliptic curves by Shimura curves  which may be summarized in the following formula (compare to \cite[Theorem 1]{RiTa}).

\begin{corollary}\label{degree}
 Let  $\cal E$ be an isogeny class of  semistable elliptic curves over $\Q$  of conductor $N$, and  $p$ be a prime such that the mod $p$ representation arising from $\cal E$ is irreducible. We also assume that $p$ is prime to $N$. Conisder a factorisation $N=D\cdot(N/D)$ with $D$ having an even number of prime factors  an optimal parametrization  $X_0^D(N/D) \to E$ with $ E \in \cal E$, and let $\delta_D$ be its degree. Then for primes $q,r$ such that $qr|D$, 
 the $p$-part of \[\ds {\delta_{D/qr} \over \delta_D}\] and the $p$-part of $c_qc_r$ are equal where $c_q,c_r$ are the orders of the component groups of  any $E \in \cal E$ at the primes $q$ and $r$.
\end{corollary}

\begin{proof}
The result follows from the first part of  Theorem \ref{change-of-eta}  and the well known relation  between congruence modules and degrees. We leave the details to the interested reader.
\end{proof}

We get results about the surjectivity of maps on component groups at primes $q$ of multiplicative reduction of elliptic curves $E$ that are induced by parametrizations of  $E$ by Shimura curves whose Jacobians have purely toric reduction at $q$ (compare to the the arguments on \cite[page 11113]{RiTa}).

  \begin{corollary}\label{component}
    With the notation of the previous corollary, for a prime $q|D$,  the  map induced   by an optimal parametrization $X_0^D(N/D) \to E$ on the $p$-parts of the component groups  $\phi_q(J_0^D(N/D)) \to \phi_q(E)$ is surjective.
  
  \end{corollary}
 
 \begin{proof}
 This follows from the corollary above and  \cite[Proposition 2]{RiTa}.
 \end{proof}

  \subsection{Some remarks about Theorem \ref{change-of-eta}}\label{remarks}

\begin{itemize}

\item The first part  of Theorem \ref{change-of-eta}  was proved  in \cite[Proposition 9.1]{BKM}, using the methods of  \cite{RiTa}, in particular \cite[Theorem 1]{RiTa}.   We  have reverse engineered the arguments of \cite[Proposition 9.1]{BKM}, and  are able to deduce      \cite[Theorem 1]{RiTa}  below by a different method which is more robust.  We  still use \cite[Theorem 2]{RiTa}  to prove upper bounds on change of congruence modules (or equivalently degrees of parametrizations)
\[\ell_\cO(\Psi_\lambda(M(NQ))) \leq \ell_\cO(\Psi_\lambda(M^{\st,Q}(N))) + \sum_{q \in Q} m_q\] 
 but not the less robust and delicate  methods   of  the proof of the second part of  \cite[Theorem 1, see also page 11113]{RiTa}, which show that these upper bounds in fact give exactly the change of  lengths of the congruence modules. We view   the correct upper bounds on  change of congruence modules,  when we relax deformation conditions at primes in $Q$ (from Steinberg to unrestricted with fixed determinant) as   ``easier’’ than the corresponding  correct lower bounds  (correctness lying in the fact that the bounds are expected to turn into equalities).  In the analogous case of  lengths of relative cotangent spaces, the inequality   \[\ell_\cO(\Phi_{R/R^\st}) \leq \sum_{\ell|N} ( \ord_\cO (\ell^2-1)  -n_{\ell}) +\sum_{q \in Q}(m_q+\ord_\cO(q^2-1)-2n_q)\]  follows purely from local arguments: see \cite[Proposition 7.9]{BKM} for the local computation, and also note that the surjectivity  of the map $\Phi_{\lambda,R_\infty/R_\infty^{\St} } \rightarrow \Phi_{\lambda,R/R^{\St}}$ of \cite[Theorem 7.14]{BKM}  is elementary.  The injectivity of this map which is proved  in \cite[Theorem 7.14]{BKM}  lies deeper and  uses   patching arguments. Thus the heuristic that we justify by our work  here is  that (correct)  upper bounds on change of congruence modules, or change of cotangent spaces, are ``easy’’ and our methods allow one to convert these upper bounds to equalities using the methods of this paper.
 
 \item The  proof of  \cite[Theorem 1, part 2]{RiTa}  on page 11113 depends on  the  hypothesis that  $N/D$ is  not prime  (that is used to ``permute’’ primes around there) which we can dispense with in Corollary \ref{degree}.

\item  Using  (a straightforward modification  of) of \cite[Theorem 5.2]{BKM} and \cite[Theorem 8.1, Cor. 8.3]{BKM}  (which considered $M(NQ^2)$ rather than $M(N^2Q^2)$), and under the  assumption that $N|N(\rhobar)$ of \cite{BKM} we know  from \cite{BKM}   that 
 \[ \Psi_\lambda(M(N^2Q^2))=\Psi_\lambda(\T), \Psi_\lambda(M^{\st,Q}(N))=\Psi_\lambda(\T^{\st,Q}).\] 
 On the other hand using Theorem \ref{thm:mc}, together with \cite[Proposition 7.9, Corollary 7.15]{BKM} we know that
  \[\ell_\cO(\Psi_\lambda(\T))= \ell_\cO(\Psi_\lambda(\T^{\st,Q})) + \sum_{\ell|N} \ord_\cO(\ell^2-1)+ \sum_{q \in Q} (m_q+\ord_\cO(q^2-1)).\]
Combining this we  can deduce the first part   \[\ell_\cO(\Psi_\lambda(M(NQ^2)))= \ell_\cO(\Psi_\lambda(M^{\st,Q}(N))) + \sum_{\ell|N} \ord_\cO(\ell^2-1) + \sum_{q \in Q} (m_q+\ord_\cO(q^2-1))\] of the theorem above. The arguments given in Theorem \ref{change-of-eta} use Theorem \ref{thm:mc} to deduce {\em numerically}  the equality of  cohomological and ring theoretic defects, or equivalently of lengths as $\cO$-modules of  ring theoretic and cohomological congruence modules, seem more versatile and apply  in cases where the arguments of  \cite[Corollary 8.3 ]{BKM}  do not apply, and do not use the assumption that $N|N(\rhobar)$.

\item Corollary \ref{component}   does   not seem  accessible via the  methods of \cite{RiTa}  in the case of $D$ divisible by some  trivial primes for the mod $p$ representation arising from $E$  and $D$ is prime to $N(\rhobar)$.    Both the above corollaries should easily generalize to  elliptic curves over $\Q$ that are not semistable. In \cite{RiTa} ingenious arguments are used to get round the fact that one does not have general proofs of surjectivity of  maps $\phi_q(J_0^D(N/D)) \to \phi_q(E)$ on the $p$-part of  component groups. It is not hard to prove this  surjectivity by geometric means when $q$ is not 1 mod  $p$, but as far the authors know in  the case when $q$ is a trivial prime for the Galois representation arising from  $E[p]$ there is no geometric proof available.
   
\item The argument here should  generalize easily to the case of totally real fields $F$ using analogs of Ribet sequences relating character groups of reductions of modular curves and Shimura curves as in  \cite{Jarvis1}. It should also generalize without too much difficulty to the case of newforms of weight $k > 2$. There are some related results in \cite{KimOta}, they only consider situations where the Hecke algebras are complete intersections and hence of defect 0. The assumption that $N$ is squarefree  is also not eseential to our methods, the results given here are more illustrative than exhaustive.
   
\item  We assumed in this section that $f$ was a newform of level $NQ$, and so in particular $\rho_f$ ramifies at each prime dividing $N$. It it possible to prove the equality of cohomological and ring theoretic defects somewhat more generally by using the arguments of \cite{DiamondMult1}.
  
Specifically, assume that $f$ is a newform of level $N_\es Q$ for some integer $N_\es$. Then Theorem \ref{change-of-eta} gives an equality $\delta_{\lambda}(M^{\st,Q}(N_\es)) = \delta_\lambda(\T^{\st,Q}(N_\es))$.
  
Now let $\Sigma$ be a finite set of primes not containing any primes dividing $N_\es Q$, and let $N_\Sigma$ be the level considered in \cite[Section 3.2]{DiamondMult1}. The inequalities given in the proof of \cite[Theorem 3.4]{DiamondMult1} (which in our case rely on Ihara's Lemma for the Shimura curves $X^Q_0(N)$) then show that $\delta_\lambda(M^{\st,Q}(N_\Sigma))\le \delta_\lambda(M^{\st,Q}(N_\es))$.
  
But now for each prime $q\in \Sigma$, one has that $R_q^{\min}$ and $R_q^{\square}$ are both complete intersections. The work of Section \ref{sec:deformation} then implies that $\delta_\lambda(\T^{\st,Q}(N_\es)) = \delta_{\lambda}(\T^{\st,Q}(N_\Sigma))$. One then deduces that
\[
\delta_\lambda(\T^{\st,Q}(N_\Sigma))\le\delta_\lambda(M^{\st,Q}(N_\Sigma))\le \delta_\lambda(M^{\st,Q}(N_\es)) = \delta_\lambda(\T^{\st,Q}(N_\es)) =  \delta_{\lambda}(\T^{\st,Q}(N_\Sigma))
\]
and so $\delta_\lambda(\T^{\st,Q}(N_\Sigma))=\delta_\lambda(M^{\st,Q}(N_\Sigma))$ for all $\Sigma$, generalizing Theorem \ref{change-of-eta}. By a similar argument, one can also generalize Proposition \ref{prop:moduledefect}.
\end{itemize}
   
\appendix

\section{A formula of Venkatesh  by  N. Fakhruddin and C. Khare}\label{app}
  
The results of this section are inspired by unpublished notes of
A.~Venkatesh \cite{venkatesh}. Venkatesh's formula was stated (as a
conjecture, but it was checked in many cases) for certain derived
commutative rings, but we prove a version in the context of ordinary
commutative algebra; we briefly explain the connection in Section
\ref{s:derived}. The invariants $c_0$ and $c_1$ are essentially the
same as those defined in \cite{venkatesh}, but our method of proof is
different from the approach taken there. The formula is used in the main text to
  compute the Wiles defect for certain Hecke algebras that are not
  complete intersections.

Let $\mc{O}$ be a complete dvr and consider $B$, a complete local
Noetherian $\mc{O}$-algebra with $\dim(B) = 1$ and with an
augmentation $\pi_B: B \to {\mc{O}}$. Let $E$ be the quotient field of
$\mc{O}$ which we view as a module over any augmented ring using the
augmentation. We assume that the augmentation has a finite cotangent
space, by which we mean that $\ker(\pi_B)/\ker(\pi_B)^2$ is a finite
length $\mc{O}$-module . Let $C$ be the largest Cohen--Macaulay
quotient of $B$ --- if $B$ is finite over $\mc{O}$ then this is simply
the quotient of $B$ by its $\mc{O}$-torsion (which is an ideal)---and
let $\pi_C: C \to {\mc{O}}$ be the augmentation of $C$ induced by
$\pi_B$.
\begin{definition} 
 $c_0(B) := \ell({\mc{O}}/\pi_C(\Ann(\ker(\pi_C))))$.
\end{definition}

Since $B$ is complete, we may write it as a quotient of
$S = {\mc{O}}[[x_1,x_2,\dots,x_n]]$ for some $n \geq 0$. Then by the
prime avoidance lemma (\cite[Lemma 1.2.2]{BH}), we may find
a quotient $A$ of $S$ through which the map to $B$ factors and such
that $A$ is a complete intersection ring with $\dim(A) = 1$. Denote
this map $A \to B$ by $\phi_B$ and the induced map $A \to {\mc{O}}$ by
$\pi_A$. We may (and do) choose $A$ such that
$\ker(\pi_A)/\ker(\pi_A)^2$ is a finite length
$\mc{O}$-module. Furthermore, if $B$ is finite over $\mc{O}$ the lemma
also allows us to choose $A$ finite over~$\mc{O}$.

\smallskip

Let $\mf{x}$ be a sequence of generators of $\ker(\phi_B)$ of length
$\delta$ and consider the Koszul complex\footnote{We use the notation
  and standard properties of the Koszul complex as in
  \cite[\S1.6]{BH}.} $K_A(\mf{x})$. It is a graded-commutative
  differential graded $A$-algebra whose homology modules are
  $B$-modules. Let $H_{\delta}(K_A(\mf{x}))_1$ be the submodule of
  $H_{\delta}(K_A(\mf{x}))$ generated by products of elements of
  $H_1(K_A(\mf{x}))$. The Koszul complex is functorial for ring
  homomorphisms, so we have a map
  $\pi_{A,*}:H_*(K_A(\mf{x})) \to H_*(K_{\mc{O}}(\ov{\mf{x}}))$, where
  $\ov{\mf{x}}$ denotes the image of the sequence $\mf{x}$ in
  $\mc{O}$. However, all terms of this sequence are $0$, so
  $H_*(K_{\mc{O}}(\ov{\mf{x}}))$ is the exterior algebra in $\delta$
  generators (in homological degree $1$). In particular,
  $H_{\delta}(K_{\mc{O}}(\ov{\mf{x}})) \cong {\mc{O}}$.
\begin{definition}
  $c_1(B) := \ell(\pi_{A,*}(H_{\delta}(K_A(\mf{x})))/ \pi_{A,*}(H_{\delta}(K_A(\mf{x}))_1))$.
\end{definition}
We see that this is finite by localising at the prime ideal
corresponding to the kernel of $\pi_A$ and observing that this
localisation map factors through $\pi_A$.
\smallskip

From the definition of the Koszul complex, it follows that
$H_{\delta}(K(\mf{x}))$ is the annihilator of the ideal $I$ generated
by the sequence $\mf{x}$. The $A$-submodule of $H_{\delta}(K(\mf{x}))$
generated by products of elements of $H_1(K(\mf{x}))$ is precisely the
Fitting ideal of $I$ (sitting inside its annihilator). It follows that
\begin{equation}\label{eq:c2}
c_1(B) = \ell(\pi_A (\Ann(\ker(\phi_B)))/\pi_A(\Fitt(\ker(\phi_B)))).
\end{equation}

We now show that $c_1(B)$ is independent of all choices. For a fixed
$A$ as above, the Koszul complex only depends on the number of
generators of the kernel. Moreover, adding more elements in the kernel
to the sequence of generators has the effect of tensoring the Koszul
complex with an exterior algebra in which case it is easy to see that
$c_1$ does not change.

To show that it is independent of the choice of $\phi_B: A \to B$, we
will need the following elementary lemma.

\begin{lemma} \label{l:fibre} Let $\mc{O}$ be any commutative ring,
  $A_1,A_2, B$ be local Noetherian $\mc{O}$-algebras and
  $\phi_i:A_i \to B$, $i=1,2$ surjections of $\mc{O}$-algebras. Then
  \begin{enumerate}
  \item $A :=A_1 \times_B A_2$ is also a local Noetherian
    $\mc{O}$-algebra and $\dim(A) = \max\{\dim(A_1),\dim(A_2)\}$.
  \item If $A_1$ and $A_2$ are complete then so is $A$.
  \item Let $P$ be any prime ideal in $B$, $P_i = \phi_i^{-1}(P)$ the
    corresponding prime ideals of $A_i$ and $P_A = \phi^{-1}(P)$ that
    of $A$ (where $\phi:A \to B$ is the surjection induced by
    $\phi_i$). Then $A_{P_A} = (A_1)_{P_1} \times_{B_P} (A_1)_{P_2}$.
\end{enumerate}
\end{lemma}
  
\begin{proof}
  We have
  $A = \{ (a_1,a_2) \in A_1 \times A_2: \phi_1(a_1) = \phi_2(a_2)
  \}$. The ideal $m_A$ of $A$ consisting of all pairs $(a_1,a_2)$ with
  $a_i \in m_{A_i}$ is the unique maximal ideal of $A$, since the
  surjectivity of $\phi_1, \phi_2$ implies that the complement
  consists of invertible elements, so $A$ is local.  The two
  projections induce surjections $p_i:A \to A_i$. If $I$ is an ideal
  of $A$ then $p_1(I)$ is an ideal of $A_1$. The kernel of the map
  $I \to p_1(I)$ is naturally an ideal of $A_2$. Since $A_1$ and $A_2$
  are Noetherian, this implies that $A$ is Noetherian.

  Now since $A$ is a subring of $A_1 \times A_2$ which is finite as an
  $A$-module (it is generated by $(1,0)$ and $(0,1)$), it follows from
  the going-up theorem \cite[Theorem 5, (i), (ii) and
  (iii)]{Matsumura-CA} that
  $\dim(A)= \dim(A_1 \times A_2) = \max\{\dim(A_1), \dim(A_2)\}$.

  Suppose $A_1$ and $A_2$ are complete. To show that $A$ is complete
  it suffices to prove that the $m_A$-adic topology on $A$ is the same
  as the topology induced from the inclusion of $A$ in
  $A_1 \times A_2$. Since $m_A^n \subset m_{A_1}^n \times m_{A_2}^n$
  for all $n>0$, we only need to show that given any $n'>0$,
  $(m_{A_1}^n \times m_{A_2}^n) \cap A \subset m_A^{n'}$ for all
  $n \gg 0$. This follows immediately by applying the Artin--Rees
  lemma \cite[Theorem 15]{Matsumura-CA}, with $I = m_A$,
  $M = A_1 \times A_2$ and $N = A$, since
  $I^nM = m_{A_1}^n \times m_{A_2}^n$.

  We will use the following elementary fact whose simple proof we
  skip: If $A$ is any commutative ring, $S$ any multiplicative subset
  of $A$, $M_1$, $M_2$ and $N$ any $A$-modules with maps $M_i \to N$,
  $i=1,2$, then the natural map
  $M_1 \times_N M_2 \to (M_2)_S \times_{N_{S}} (M_2)_S$ of $A$-modules
  induces an isomorphism
  $(M_1 \times_N M_2)_S \to (M_1)_S \times_{N_{S}} (M_2)_S$. The
  statement (3) follows from this by taking $M_i$ to be $A_i$, $N$ to
  be $B$ and $S = A \setminus P_A$ and by observing that
  $A_i \otimes_A A_{P_A} = (A_i)_{P_i}$, $i=1,2$ and
  $B \otimes_A A_{P_A} = B_P$.

\end{proof}

It follows from Lemma \ref{l:fibre} and the prime avoidance lemma
already used earlier, that if $A_i$ are complete intersections of the
same dimension with surjections to $B$, then both of them may be
dominated by a complete intersection $A'$ of the same dimension. The
condition on the finiteness of the cotangent space can also be
preserved by (3) of Lemma \ref{l:fibre}. For the independence of the choice
of $A$ in the definition of $c_1(B)$ we will also need:
\begin{lemma} \label{l:ext}
  Let $f:A' \to A$ be a surjection of (complete) complete intersection
  local rings and let $\phi_B:A \to B$ be any surjection of rings. Let
  $\mf{z}$ be any finite sequence of generators of $\ker(f)$, $\mf{x}$
  any sequence of generators of $\ker(\phi_B)$, and $\mf{x}'$ a lift
  of $\mf{x}$ to $A'$.  Then $H_*(K_{A'}((\mf{z}, \mf{x}'))$ is
  isomorphic to $H_*(K_A(\mf{x}))$ tensored with an exterior algebra
  over $A$ with $|\mf{z}| + \dim(A) - \dim(A')$ free generators.
\end{lemma}

\begin{proof}
  Let $g: S \to A'$ be a surjection from a regular local ring $S$
  (which exists because $A'$ is complete), so both $\ker(g)$ and
  $\ker(fg)$ are generated by regular sequences.  Choose a sequence of
  generators ${\mf{y}}$ of $\ker(\phi_Bfg)$ by first choosing a
  regular sequence of generators $\mf{w}$ of $\ker(g)$ and then adding
  lifts $\tilde{\mf{z}}$ of elements of $\mf{z}$ and lifts
  $\tilde{\mf{x}}'$ of lifts $\mf{x}'$ in $A'$ of elements of
  $\mf{x}$. We then set
  $\mf{y} = (\mf{w}, \tilde{\mf{x}}', \tilde{\mf{z}})$ and consider
  $K_S(\mf{y})$. Since the Koszul complex of a regular sequence is a
  resolution of the corresponding quotient ring, by applying this to
  $\mf{w}$ we see that $K_S(\mf{y})$ is quasi-isomorphic (as a
  differential graded $S$-algebra) to
  $K_S((\tilde{\mf{z}}, \tilde{\mf{x}}')) \otimes_S A'$, i.e.,
  $K_{A'}((\mf{z},\mf{x}'))$. On the other hand, since $A$ is also
  complete intersection ring, by choosing a minimal generating set of
  $\ker(fg)$ from among the elements of $(\mf{w}, \tilde{\mf{z}})$,
  one sees that $K_S(\mf{y})$ is quasi-isomorphic to $K_A(\mf{x})$
  tensored with an exterior algebra (since
  $gf(\tilde{\mf{x}}') = \mf{x}$ and the remaining elements of
  $(\mf{w}, \tilde{\mf{z}})$ become $0$ in $A$). On taking homology we
  see that $H_*(K_{A'}((\mf{z},\mf{x}'))$ is isomorphic to
  $H_*(K_A(\mf{x}))$ tensored with an exterior algebra. The number of
  free generators of this exterior algebra is easily seen to be
  $|\mf{z}| + \dim(A) - \dim(A')$ since
  $|\mf{w}| = \dim(S) - \dim(A')$.

\end{proof}

\begin{lemma} \label{l:indc2}
  The invariant $c_1(B)$ is well-defined.
\end{lemma}

\begin{proof}

  By Lemma \ref{l:fibre} and the remarks following it, it suffices to
  show that if $\phi_B:A \to B$ is as above and we have a surjection
  $f:A' \to A$ such that $\phi_B' := \phi_B f$ also satisfies the
  conditions analogous to those imposed on $\phi_B$, then the numbers
  defined using $\phi_B$ and $\phi_B'$ are equal.

  Let $\mf{x}$ be a sequence of generators of $\ker(\phi_B)$,
  $\mf{x}'$ a lift of this sequence to $A'$ and $\mf{z}$ a sequence of
  generators of $\ker(f)$. Let $\mf{w} = (\mf{z}, \mf{x}')$, so
  $\mf{w}$ is a sequence of generators of $\ker(\phi_B')$. Thus,
  $\delta = |\mf{x}|$ and $\delta'$, the corresponding number of
  generators for $\ker(\phi_B')$, equals
  $|\mf{w}| = \delta + |\mf{z}|$.  Note that
  $H_{\delta'}(K_{A'}((\mf{z}, \mf{x}'))$ is canonically isomorphic to
  $\Ann(\ker(\phi_B'))$ and $H_{\delta}(K_A(\mf{x}))$ is canonically
  isomorphic to $\Ann(\ker(\phi_B))$. From the result of Lemma
  \ref{l:ext} (specialised to the case $\dim(A) = \dim(A')$) 
  that $H_*(K_{A'}((\mf{z}, \mf{x}'))$ is isomorphic to
  $H_*(K_A(\mf{x}))$ tensored with an exterior algebra over $A$ with
  $|\mf{z}|$ free generators, it follows that there is an isomorphism
  of $A$-modules $\alpha: \Ann(\ker(\phi_B')) \to \Ann(\ker(\phi_B))$
  such that $\alpha(\Fitt(\ker(\phi_B'))) = \Fitt(\ker(\phi_B))$.

  Now we use the finite cotangent space assumption on $A$ and
  $A'$. This implies that $\pi_A(\Ann(\ker(\phi_B))) \subset \mc{O}$
  is nonzero and equal to the image of
  $\Ann(\ker(\phi_B)) \otimes_A \mc{O}$ in
  $A \otimes_A \mc{O} = \mc{O}$ (and similarly for $A'$ and also for
  the Fitting ideals). The $\mc{O}$-module
  $\Ann(\ker(\phi_B)) \otimes_A \mc{O}$ modulo its torsion is free of
  rank one (and similarly for $A'$) so the lemma follows from
  \eqref{eq:c2} and the above by using the isomorphism
  $\alpha \otimes_A \mc{O}$.

\end{proof}

\smallskip

For any map of rings $R_1 \to R_2$, an $R_2$-module $M$ and
$ i \geq 0$, we denote by $\Der^i_{R_1}(R_2,M)$ the $i$-th
Andr\'e--Quillen cohomology group of $R_2$ with coefficients in $M$.
Let $E$ denote the quotient field of $\mc{O}$ viewed as a $B$-module
via $\pi_B$.

The invariants $c_0(B)$ and $c_1(B)$ defined above are linked by the
following proposition, which may be viewed as a derived version of
Wiles's formula for complete intersections \cite{Wiles},
\cite{Lenstra}, \cite[\S A]{fkr}; a variant of this formula was first
discovered by A.~Venkatesh \cite{venkatesh}.
\begin{proposition} \label{prop:v}
  \begin{equation} \label{eq:vsformula}
    c_0(B) - c_1(B) = \ell(\Der^0_{\mc{O}}(B,E/{\mc{O}})) -
    \ell(\Der^1_{\mc{O}}(B,E/{\mc{O}})).
\end{equation}
\end{proposition}
\begin{proof}
  We denote by $J$ the ideal $\ker(\phi_B)$ with $\phi_B:A \to B$ as
  above.  The sequence of maps ${\mc{O}} \to A \to B$ gives rise to an
  exact sequence of Andr\'e--Quillen cohomology
  \begin{equation} \label{AQ}
    0 \to \Der^0_{\mc{O}}(B,E/{\mc{O}}) \to
    \Der^0_{\mc{O}}(A,E/{\mc{O}}) \to \mr{Hom}_A(J/J^2, E/{\mc{O}})
    \to \Der^1_{\mc{O}}(B,E/{\mc{O}}) \to 0.
\end{equation} 
The $0$ on the left comes from the fact that $\Der^0_A(B,E/{\mc{O}}) = 0$
since $\phi_B$ is surjective (which also gives that
$\mr{Hom}_A(J/J^2, E/{\mc{O}})$ is equal to $\Der^1_A(B, E/{\mc{O}})$). The $0$ on
the right comes from the fact that
$\Der^1_{\mc{O}}(A,E/{\mc{O}}) = \Der^2_{\mc{O}}(A, {\mc{O}}) = 0$, where the first equality
is because $\Der^i_{\mc{O}}(A,E) = 0$ for all $i$ (a consequence of the
finite tangent space condition on $\pi_A$) and the second follows from
\cite[(1.2) Theorem]{avramov-lci} because $A$ is a complete intersection, 
$\mc{O}$ is regular, and we have a surjection from $S$ onto $A$.

We claim that $\mr{Hom}_A(J/J^2, E/{\mc{O}})$ and $\Der^0_{\mc{O}}(A,E/{\mc{O}})$ are finite length $\cO$-modules and that we have equalities $\ell(\mr{Hom}_A(J/J^2, E/{\mc{O}})) = \ell({\mc{O}}/\pi_A(\Fitt(J)))$ and $\ell(\Der^0_{\mc{O}}(A,E/{\mc{O}})) = \ell({\mc{O}}/\pi_A(\Fitt(\ker(\pi_A))))$. Assuming the claim, from sequence \eqref{AQ} we deduce
\begin{equation} \label{eq:DD}
\ell(\Der^0_{\mc{O}}(B,E/{\mc{O}})) - \ell(\Der^1_{\mc{O}}(B,E/{\mc{O}})) =
\ell({\mc{O}}/\pi_A(\Fitt(\ker(\pi_A)))) -  \ell( {\mc{O}}/\pi_A(\Fitt(J))).
\end{equation}
By definition  $c_0(B) = \ell({\mc{O}}/\pi_C(\Ann(\ker(\pi_C))))$, $c_1(B) = \ell(\pi_A
(\Ann(\ker(\phi_B)))/\pi_A(\Fitt(\ker(\phi_B))))$ by \eqref{eq:c2}, and 
Lemma \ref{lem:v} below  implies that 
\[
  \ell(\mc{O}/\pi_A(\Ann_A(\ker(\pi_A)))) =
  \ell(\mc{O}/\pi_A(\Ann_A(\ker(\phi_B)))) +
  \ell(\mc{O}/\pi_C(\Ann_C(\ker(\pi_C)))).
\]
Recalling that $J = \ker(\phi_B)$ and
$\Fitt(\ker(\pi_A)) = \Ann_A(\ker(\pi_A))$ (since $A$ is a complete
intersection), the proposition follows by inserting these three
equalities in \eqref{eq:DD}.

We now prove the claim made above: For the first part, note that $E/\cO$ is an $A$-module via $\pi_A$, so that $\mr{Hom}_A(J/J^2, E/{\mc{O}})\cong\Hom_\cO(J/J^2\otimes_A\cO,E/\cO)$, where $\cO$ is an $A$-module via $\pi_A$, and  it suffices to show that $J/J^2\otimes_A\cO$ is a finite length $\cO$-module. The module $J/J^2\otimes_A\cO$ is of finite type over $\cO$ because $J/J^2$ is of finite type over $A$, and so we need to show that $J/J^2[1/\varpi]\otimes_{A[1/\varpi]}E$ vanishes. Now the map $\phi[1/\varpi]:A[1/\varpi]\to B[1/\varpi]$ is a map of finite-dimensional $E$-algebras and the (compatible) augmentations to $\pi_A[1/\varpi]$ and $\pi_B[1/\varpi]$ give rise to isomorphisms of a single factor with $E$, i.e., $\pi_A[1/\varpi]\otimes_{A[1/\varpi]} E$ is an isomorphism and $J[1/\varpi]\otimes_{A[1/\varpi]}E=0$ because $J[1/\varpi]$ must then be supported on the other factors, and hence $J/J^2[1/\varpi]\otimes_{A[1/\varpi]}E=0$. For the second part we apply the conormal sequence to $\cO\to A\to \cO$ which gives the isomorphism $\ker(\pi_A)/\ker(\pi_A)^2\cong \Omega_{A/\cO}\otimes_A\cO$ due to the splitting of $A\to\cO$. By construction the right hand term in the isomorphism is of finite $\cO$-length, and the second part now follows from $\Der^0_{\mc{O}}(A,E/{\mc{O}})\cong \Hom_A(\Omega_{A/\cO},E/\cO)\cong  \Hom_\cO(\Omega_{A/\cO}\otimes_A\cO,E/\cO)$.

For the first assertion on lengths we need to show that $J/J^2\otimes_A\cO$ and $\cO/\pi_A(\Fitt(J))$ have the same lengths. Because $\pi_A(J)=0$, the image of $J^2\otimes_A\cO$ in $J\otimes_A\cO$ is zero, and hence $J/J^2\otimes_A\cO\cong  J \otimes_A\cO$. Next observe that $\pi_A(\Fitt(J))=\Fitt(J\otimes_A\cO)$, as follows from the definition of the Fitting ideal. The equality of length now follows because for a finite length $\cO$-module over the dvr $\cO$ the theory of elementary divisors gives $\ell(M)=\ell(\cO/\Fitt(M))$. The argument for the second length equality proceeds in the same way. One reduces the equality to showing that $ \ker(\pi_A)/\ker(\pi_A)^2\cong \ker(\pi_A)\otimes_A\cO$ and $\cO/\Fitt(\ker(\pi_A)\otimes_A\cO)$ have the same length.
\end{proof}

\begin{remark}\label{rem:FinitenessOfDer}
The above proof shows in particular, that the terms $\Der^1_{\mc{O}}(B,E/{\mc{O}})$ and $\cO/\pi_A(\Fitt(J))$ are of finite $\cO$-length.
\end{remark}

\begin{remark}
  If $B$ is a complete intersection in Proposition \ref{prop:v} we may
  take $A=B$, so $c_1(B) = 0$, $c_0(B) = \ell({\mc{O}}/\eta_B)$ and
  \eqref{eq:DD} shows that Proposition \ref{prop:v} reduces to Wiles's
  formula.  The proposition shows once again that $c_1(B)$ is
  independent of all choices since all the other terms in the formula
  are clearly so.
\end{remark}

The  following lemma was used in the proof of Proposition \ref{prop:v}.
  
\begin{lemma}\label{lem:v}
  Let $A$ be a Gorenstein local ring with an augmentation
  $\pi_A:A \to \mc{O}$ such that $\ell(\Phi_A) < \infty$.  Assume that
  $\pi_A$ factors through a surjective ring homomorphism
  $\phi_B: A \to B$ and let $C$ be the largest quotient of $B$ which
  is Cohen--Macaulay, so there are surjections $\phi_C:A \to C$,
  $\pi_B: B \to \mc{O}$ and $\pi_C: C \to \mc{O}$. Then
\[
\pi_A(\Ann_A(\ker(\pi_A))) = \pi_A(\Ann_A(\ker(\phi_B))) \,
\pi_C(\Ann_C(\ker(\pi_C))).
\]
\end{lemma}

\begin{proof}
We may apply Lemma  A.10 of \cite{fkr} to the map $\phi_C$, since $C$ is
Cohen--Macaulay, to deduce that 
\[
\pi_A(\Ann_A(\ker(\pi_A))) = \pi_A(\Ann_A(\ker(\phi_C))) \,
\pi_C(\Ann_C(\ker(\pi_C))),
\]
so it suffices to to prove that
$\Ann_A(\ker(\phi_C)) = \Ann_A(\ker(\phi_B))$.  We have
$\ker(\phi_B) \subset \ker(\phi_C)$ and the quotient is a finite
length $A$-module by the definition of $C$. The quotient map
\[
  \Ann_A(\ker(\phi_B)) \ker(\phi_C) \to
  \Ann_A(\ker(\phi_B))(\ker(\phi_C)/\ker(\phi_B))
\]
is an isomporphism since $\Ann_A(\ker(\phi_B))\ker(\phi_B) = (0)$, so
$\Ann_A(\ker(\phi_B)) \ker(\phi_C)$, being a submodule of a finite
length $A$-module, is also of finite length. On the other hand, it is
a submodule of $A$ and ${\rm depth}(A) = 1$, so it
must be $(0)$. Thus, $\Ann_A(\ker(\phi_C)) = \Ann_A(\ker(\phi_B))$.
\end{proof}

\subsection{} \label{s:derived}
We briefly explain how the formula \eqref{eq:vsformula} can be viewed as a derived version of Wiles's formula:

Suppose we have a presentation
$B = \mc{O}[[x_1,x_2,\dots, x_n]]/(f_1,f_2,\dots, f_{n+\delta})$ with
$\delta \geq 0$. We may use this to construct a ``derived'' ring
\[
  \mc{B} = \mc{O}[[x_1,x_2,\dots,x_n]]
  \ {\otimes}_{\mc{O}[[y_1,y_2,\dots,y_n,\dots, y_{n+\delta}]]}\  \mc{O},
\]
where the tensor product is defined as in \cite[Definition
3.3]{gal-ven}. Here the $x_i, y_j$ are in ``degree 0'' and the map
from $\mc{O}[[y_1,y_2,\dots,y_n,\dots, y_{n+\delta}]]$ to
$\mc{O}[[x_1,x_2,\dots,x_n]]$ is given by $y_j \mapsto f_j$ and to
$\mc{O}$ by $y_j \mapsto 0$.

If we assume that $A = \mc{O}[[x_1,\dots,x_n]]/(f_1,f_2,\dots, f_{n})$
is a one dimensional complete intersection, then the derived ring has
``defect'' equal to $\delta$. 
The invariant $c_1(B)$ may then be viewed as coming from
$\pi_*(\mc{B})$, since this may be computed in terms of a Koszul
complex. Venkatesh views \eqref{eq:vsformula} as an analogue of
Wiles's formula for the derived ring $\mc{B}$, which is a ``derived
complete intersection''. (However, as we have shown, all the terms in
the formula only depend on $B =\pi_0(\mc{B})$, so it may also be
viewed as a generalisation of Wiles's formula to rings which are not
necessarily complete intersections.)

 \bibliographystyle{amsalpha}
\bibliography{refs}

  \noindent {\bf Address of authors:} 

(GB) Interdisciplinary Center for Scientific Computing, Universit\"at Heidelberg, Heidelberg, Germany. \textit{Email address}: \texttt{boeckle@uni-hd.de}

(CK) Department of Mathematics, UCLA, Los Angeles, CA 90095-1555,
  USA. \textit{Email address}: \texttt{shekhar@math.ucla.edu}

(JM) Department of Mathematics, UCLA, Los Angeles, CA 90095-1555,
  USA. \textit{Email address}: \texttt{jmanning@math.ucla.edu}

\end{document}